\documentclass[a4paper, fontsize=12pt, titlepage = false, toc=bib]{scrartcl}

\usepackage{math-en}
\usepackage{pdfpages}
\usepackage{multirow}

\usepackage{newpxtext,newpxmath}
 \tikzset{every picture/.style={>=Straight Barb}} 
 \tikzset{commutative diagrams/arrow style=tikz}

 \renewcommand{\ko}{\curly O}
 
 \newcommand{\tsum}{{\textstyle \sum}}

\usetikzlibrary{positioning, intersections,calc}
\usetikzlibrary{fit}
\tikzset{
  C1col/.style = { MidnightBlue, thick},
  C1/.style ={C1col, thick},
  C2col/.style = {Orange, thick },
  C2/.style ={C2col,thick},
  C12col/.style = {ForestGreen, thick},
  C12/.style ={C12col,thick},
  CC/.style = {thick, C2col},
  L1/.style = {MidnightBlue, thick},
  L2/.style = {ProcessBlue,thick},
  L12/.style = {NavyBlue, thick},
  L3col/.style = {LimeGreen, thick},
  L3/.style = {L3col},
  L4/.style = {OliveGreen, thick},
  L34/.style = {ForestGreen, thick},
  Qs/.style ={black!50!white, opacity=1, every node/.style={black, opacity=1,
      font  = \scriptsize}}, 
  basepoint/.style ={black, opacity = 1, every node/.style={black}},
  pfeil/.style = {->, every node/.style = { font = \scriptsize}}
}

\tikzset{
  Subset/.style={
    draw=none,
    every to/.append style={
      edge node={node [sloped, allow upside down, auto=false]{$\subset$}}}
  }
}

\tikzset{
  Cblue/.style = { MidnightBlue, thick},
  Corange/.style = {Orange, thick },
  Cgreen/.style ={ForestGreen,thick},
  Cblack/.style ={Black,thick},
}

\usepackage[unicode,bookmarks, backref = page]{hyperref}
\hypersetup{colorlinks=true,citecolor=NavyBlue,linkcolor=NavyBlue,urlcolor=Orange, pdfpagemode=UseNone, breaklinks=true}

\newcommand{\caseP}[1]{$(P_{#1})$}
\newcommand{\casedP}{$(dP)$}

\usepackage{rotating} 
\usepackage{tabularx}

\usepackage{frcursive}
\usepackage[T1]{fontenc}

  \title{Gorenstein stable surfaces with $K_X^2 =1$ and $ \chi(\ko_X) =2$}
  \author{Anh Thi Do\footnote{FB 12/Mathematik und Informatik, Philipps-Universit\"at Marburg, 
 Hans-Meerwein-Str. 6,
 35032 Marburg, 
 Germany, ando@mathematik.uni-marburg.de} \and S\"onke Rollenske\footnote{FB 12/Mathematik und Informatik, Philipps-Universit\"at Marburg, 
 Hans-Meerwein-Str. 6,
 35032 Marburg, 
 Germany, rollenske@mathematik.uni-marburg.de}}
  \date{}

\begin{document}

\maketitle

\begin{abstract}
 We classify --- as far as possible --- Gorenstein stable surfaces with $K_X^2 = 1$ and $\chi(\ko_X) = 2$, describing several strata in the moduli space quite in detail. 
\end{abstract}

\setcounter{tocdepth}{2}
\tableofcontents


\section{Introduction}
In this article we complement the work of Franciosi and Pardini with the second author \cite{FPR15a,FPR15inv, FPR17, FPR18} on Gorenstein stable surface with small invariants with a detailed study of Gorenstein stable surfaces with $K_X^2 = 1$ and $\chi(\ko_X) = 2$.

The moduli space of stable surfaces is a modular compactification of Gieseker's moduli space of (canonical models of) surfaces of general type \cite{gieseker1977global}. 
Even though the actual construction of the moduli space was delayed for several decades because of formidable technical obstacles to be overcome, it was clear from the beginning \cite{ksb88}  that the objects parametrised by $\overline{\gothM}$ should be surfaces with semi-log-canonical singularities and ample canonical divisor, for short \emph{stable surfaces}. Nowadays, the existence of the compactification is known, and it is worthwhile to study individual components to get a feeling for the geometry of stable surfaces. Here we restrict to the open subset $\overline{\gothM}^{(\text{Gor})}$ parametrising Gorenstein stable surfaces where the canonical divisor is Cartier. 

 It had been classically known that canonical models of smooth surfaces with our fixed invariants are complete intersections of degree $(6,6)$ in weighted projective space $\IP(1,2,2,3,3)$ and that the bicanonical map realises them as
quadruple covers of the projective plane \cite{catanese79, Cat80, CE96}. This description extends to Gorenstein stable case \cite{FPR17, thesisAnh}. 

Although we have these two simple geometric descriptions and the small invariants do not leave much room for numerical cases, we find a wealth of different strata in 
$\overline{\gothM}^{(\text{Gor})}_{1,2}$, which we usually distinguish by the geometry of the minimal resolution. An overview is given in Table \ref{tab: cases}.  Many cases had been identified previously via examples in \cite{FPR17}, usually constructed as bi-double covers (see section \ref{sect: bidouble}). Our contribution consists in carrying this analysis further and providing, if humanly possible, a description of the general member of each  stratum. More often than not this consists in exploring the geometry of a particular class of surfaces containing a particular configuration of curves. In particular,  the case of non-normal surfaces with normalisation the symmetric product of an elliptic curves has finally gotten an explicit realisation via equations.

\begin{sidewaystable} \caption{Overview over the  cases}
 \label{tab: cases}
 \begin{center}
  \begin{tabular}{ c c c c c c}
\toprule 
\multicolumn{6}{c}{normal cases}\\
\midrule
min. resolution &  $\kappa(\tilde{X})$  	&	 $ p_g(\tilde X)$ & $q(\tilde X)$ &  \textbf{Elliptic sing.} $(d_1,..,d_r)$ & Reference \\
\midrule 
   gen. type &  2 & 1 & 0   &   & \\
    \multirow{2}*{min. prop. ell.} & 1 & 1 & 1 & (1) & 
     Section \ref{typeA}  \\
    & 1 & 0 & 0 & (1) & 
     Section \ref{typeB}  \\
  Enriques & 0 & 0 & 0&(2)& Section \ref{sect: Enriques case} \\
  Abelian variety   & 0 & 1 &2& (1,1)& Section \ref{sect: torus case}\\
   Bielliptic   & 0 & 0 & 1 & (1,1) & Section \ref{sect: bielliptic case}\\
 rational&  $-\infty$ & 0  & 0 &	 $(d)$ &  Section \ref{sect: kappa negative case}\\
 Ruled over ell. curve & $ -\infty $  & 0& 1  &$(d_1,d_2)$&  Section \ref{sect: kappa negative case}\\
 \bottomrule
 ~\\
 ~\\
 ~\\
~\\
 \toprule
\multicolumn{6}{c}{non-normal cases}\\
\midrule
normalisation $\bar X$ & \multicolumn{4}{c}{(general) conductor} & Reference \\
\midrule 
{$\IP^2$} & \multicolumn{4}{c}{special quartic with at least three nodes}  &{Section \ref{sect: caseP}} \\
  del Pezzo of degree $1$\footnote{possibly with one elliptic singularity} & \multicolumn{4}{c}{bi-elliptic curve in $|-2K_{\bar X}|$}   & {Section \ref{sect: cased dP and E_}}\\
 symmetric square of elliptic curve & \multicolumn{4}{c}{$D\in |3C_0-F|$ curve of genus $2$} &  Section \ref{sect: caseE}\\
    \bottomrule
  \end{tabular}
\end{center}  
\end{sidewaystable}
 
\subsection*{Acknowledgments.}
This article grew out of the first authors PhD thesis supported by German Academic Exchange Service (DAAD). 
The second author is grateful for partial  support of the DFG through the Emmy Noether program. 

We are grateful to Andreas Krug for help with Lemma \ref{lem: lemma for type A} and to Hans-Christian von Bothmer for invaluable hints how to perform the computer algebra computation necessary in Section \ref{sect: caseE}. We also enjoyed discussions around the subject of this article with Ben Anthes, Marco Franciosi, and  Rita Pardini.


\section{Preparations}
In this section we recall some definitions and give a short overview over the previous results on  Gorenstein stable surface $X$ with $K_X^2 =1$ and $\chi(\ko_X) =2$.

\subsection{Stable surfaces}
Our main reference for the material presented here is \cite[Section 5]{Kollar2013}, see also \cite{kollar12}. 

A demi-normal scheme is a finite type scheme $X$ over $\IC$ that satisfies the following
conditions:
\begin{enumerate}
\item $X$ satisfies the $S_2$  condition, i.e.,  for every $x \in X$ we have
\[ \text{depth}_{\ko_{X,x}} (\ko_{X,x}) \geq \text{min}\{2, \dim(\ko_{X,x})\}\]
\item At each point $x$ of codimension one in $X$, $x$ is either regular or is an ordinary double point.
\end{enumerate}
We denote by $\pi \colon \bar X \to X$ the normalisation of $X$. The conductor ideal
$ \ki_{\bar D} = \shom(\pi_*\ko_{\bar X}, \ko_X)$ is an ideal sheaf both in $\ko_{X}$ and
$\ko_{\bar X}$ and as such defines subschemes $D \subset X$ and $\bar D \subset
\bar X$, both reduced and of codimension one. We often refer to $D$ as the non-normal locus of $X$.
\begin{defin}
The demi-normal surface $X$ is said to have \textit{semi-log-canonical (slc)}   singularities if it satisfies the following conditions:
\begin{enumerate}
\item The canonical divisor $K_X$ is $\IQ$-Cartier,
\item The pair $(\bar X, \bar D)$ has log-canonical (lc) singularities.
\end{enumerate}
\end{defin}
It is called a stable surface if in addition $K_X$ is ample. In that case we
define the geometric genus of $X$ to be $p_g(X) = h^0(X, \omega_X) = h^2(X,
\ko_X)$ and the irregularity as $q(X) = h^1(X, \omega_X) = h^1(X, \ko_X)$. A
Gorenstein stable surface is a stable surface such that $K_X$ is an ample Cartier divisor.

\
We will discuss surfaces in the following hierarchy of open inclusions of moduli spaces of surfaces with fixed invariants $a=K_X^2$ and $b=\chi(\ko_X)$:
\begin{center}
\begin{tikzpicture}[commutative diagrams/every diagram]
\matrix[matrix of math nodes, name=m] {
\gothM_{a,b}\\
\overline\gothM_{a,b}^{(Gor)}\\ \overline\gothM_{a,b}\\
};
\path[commutative diagrams/.cd, every arrow, every label]
(m-1-1) edge[commutative diagrams/hook] (m-2-1)
(m-2-1) edge[commutative diagrams/hook] (m-3-1);
\path (m-1-1) ++ (1,0) node[right] {= \text{Gieseker moduli space of surfaces of general type}};
\path (m-2-1) ++ (1,0) node[right] {= \text{moduli space of Gorenstein stable surfaces}};
\path (m-3-1) ++ (1,0) node[right] {= \text{moduli space of stable surfaces}};
\end{tikzpicture}
\end{center}
The openness of second inclusion follows from \cite[Cor.~3.3.15]{Bruns-Herzog}.
For the time being there is no  self-contained reference for the existence of the moduli space of stable surfaces with fixed invariants as a projective scheme, and we will not use this explicitly. A major obstacle in the construction is that in the  definition of the moduli functor one needs additional conditions beyond flatness to guarantee that invariants are constant in a family. For Gorenstein surfaces these problems do not play a role; we refer to \cite{kollar12} and the forthcoming book \cite{KollarModuli} for details.

\subsection{Gorenstein stable surfaces with $K_X^2 = 1$ and $\chi(X) = 2$}
The following result from \cite{FPR17} extends the classical case \cite{catanese79}.
\begin{thm}\label{thm: structure from FPR}
 Let $X$ be a Gorenstein stable surface with $K_X^2 = 1$ and $\chi(X) = 2$. Then $q(X)=0$ and the canonical ring $R(X, K_X)\isom \IC[x, y_1, y_2, z_1, z_2]/(f_1,f_2)$, where $\deg x = 1, \deg y_i = 2, \deg z_i =3$ and 
  \begin{equation}
    \label{CanRing}
    \begin{split}
    f_1 &= z_1^2 + z_2x_0a_1(x_0, y_1, y_2) + b_1(x_0, y_1, y_2) \\
    f_2 &= z_2^2 + z_1x_0a_2(x_0, y_1, y_2) + b_2(x_0, y_1, y_2)
    \end{split}
  \end{equation}
are weighted homogeneous of degree $6$ and $b_1, b_2 $ have no common factor. Hence,   $X$ is canonically embedded as a complete intersection of bidegree $(6,6)$ in (the smooth locus of) $\IP(1,2,2,3,3)$. 

The moduli space $\overline\gothM_{1,2}^{(\text{Gor})}$ is irreducible and rational of dimension $18$. 
\end{thm}

\begin{cor}\label{cor: bicanonical section defines a quadruple cover}
The bi-canonical map $\varphi: X \stackrel{|2K_X|}{\longrightarrow} \IP^2$ is a quadruple cover.
\end{cor}
\begin{proof}
It follows from the two equations of $X$ in the previous theorem that the map $|2K_X|: X \to \IP^2$ is induced from the inclusion of rings $\IC[x^2, y_1, y_2] \into R(K_X)$.
\end{proof}

\begin{rem}
 There is a well developed theory of Gorenstein quadruple covers by Casnati and Ekedahl \cite{CE96} and the exact relation to the structure equations of the  canonical ring was worked out in the first authors PhD thesis \cite{thesisAnh}. Unfortunately, we were not able to convert this alternative description into results. 
\end{rem}

\subsubsection{Automorphisms}
Surfaces with many symmetries are easier to deal with, so we describe,  which surfaces admit automorphisms. This was discussed for smooth surfaces also in \cite[\S 1, Prop. 10]{catanese79}. 
\begin{prop}\label{prop: aut}
Let $X$ be a Gorenstein stable surface with $K_X^2 = 1$ and $\chi = 2$. Then 
\item[1)] $\Aut(X/\IP^2)$ is one of groups $0, \IZ/2, (\IZ/2)^2$.
\item[2)] $\Aut(X/\IP^2) = (\IZ/2)^2$ if and only if one can choose $a_1 = a_2 = 0$ in the equations \eqref{CanRing}.
\item[3)] $\Aut(X/\IP^2) = \IZ/2$ if and only if only one can choose one of the $a_i = 0$ but not both.
\end{prop}
\begin{proof}
First of all note that in case $a_1 =0$ in the coordinates as in \eqref{CanRing}, then 
\[(x,y_1, y_2, z_1, z_2) \mapsto (x,y_1, y_2, -z_1, z_2)\]
defines a $\IZ/2$-action on the canonical ring and on $X$ over $\IP^2$, and similarily for $a_2 = 0$. So it remains to exclude the case $\IZ/4$ and to prove the ``only if'' part.

Let $G$ be a non-trivial group acting on $X$ over $\IP^2$, so that we have a factorisation $X \to X/G \to \IP^2$. Since the bi-canonical map is of degree $4$, we see that the order of $G$ is  $2$ or $4$.
The action of $G$ on $X$ induces an action on the canonical ring $R$, which leaves $R_2$ invariant, that is $\IC[x^2, y_1, y_2] \subset R^G$.

Our arguments now rely on the examination of the decompositions  of $R_3$ and $R_6$ as $G$-representations:
\begin{align*}
 R_1 & = \langle x\rangle\\
 R_2 &= R_1^{\tensor 2} \oplus \langle y_1, y_2\rangle \text{ trivial representation}\\
 R_3 &= R_1^{\tensor 3} \oplus R_1\tensor\langle y_1, y_2\rangle \oplus U, \quad \text{ $U$ possibly reducible but effective }\\
 R_6 &= S^2 U \oplus R_1\tensor R_2\tensor U \oplus S^3R_2/\langle f_2, f_2\rangle
\end{align*}
Since the relations $f_i$ are of the form \eqref{CanRing}, we see that they both contain a non-trivial $G$-invariant summand, namely the $b_i$ and thus are in the subspace $R_6^G$. Since $z_i \not \in R_1\tensor 
R_2$ the projection $R_6 \to S^2 U$ gives a $2$-dimensional $G$-invariant subspace of $S^2U$ spanned by the images of the $z_i ^2$. 

It is easy 
 to see that for an effective, $2$-dimensional  $\IZ/4$-representation $U$ this cannot happen, so $G\neq \IZ/4$.

Let $\sigma \in G$ be a non-trivial element with necessarily $\sigma^2=1$. Let $z_1, z_2$ be eigenvectors for the action of $\sigma$ on $U$. This possibly changes our choice of coordinates, but we will shortly see that the form of equations \eqref{CanRing} is maintained. 
Then 
 \begin{align*}
 \sigma: \IC[x, y_1, y_2, z_1, z_2] &\to \IC[x, y_1, y_2, z_1, z_2]\\
 (x, y_1, y_2, z_1, z_2) &\mapsto v = (\pm x, y_1, y_2, \pm z_1, \pm z_2)
 \end{align*}
 Note that the action on $X$ is nontrivial and that, up to  the (weighed)
 $\IC^*$-action and renumbering,  we have the following 3 cases to consider:
 \begin{itemize}
 \item $v= (x,y_1, y_2, z_1, z_2). $ In this case $\sigma $ acts trivially on $X$ which was excluded.
 \item $v= (-x,y_1, y_2, -z_1, z_2). $  In this case, the only invariant monomials of degree $6$ involving $z_i$ are
 $z_1^2, z_2^2, xz_1$, so we have relations of the form
 \[ z_1^2 + 2\tilde a_1xz_1 +\tilde  b_1 = (z_1+xa_1)^2 + \tilde b_1', \quad z_2^2 + \tilde a_2 xz_1 + \tilde b_2,\]
 which after a coordinate change give the desired equation with $a_1 = 0$.
 
 \item $v = (-x, y_1, y_2, z_1, z_2)$.
  In this case, the only invariant monomials of degree $6$ involving $z_i$ are
 $z_1^2, z_1z_2, z_2^2$, so we have relations of the form
 \[ z_1^2 +  b_1 , \quad z_2^2  + b_2,\]
 after possibly a linear change of coordinates in the subspace $U$.
 \end{itemize}
 To conclude we only have to note that in the case of an effective
 $(\IZ/2)^2$-action on $X$, we will necessarily have an involution which acts as
 in the last case. 
\end{proof}


\subsubsection{Bi-double covers}\label{sect: bidouble}

The easiest way to construct examples is often to consider the case where the bicanonical map is a  $\IZ_2^2$-cover (``bi-double cover''). Non-normal abelian covers are studied in  \cite{alexeev-pardini12} in detail but we extract the minimum needed in our special case from the discussion in \cite[Sect. 5.1]{FPR17}, to which we refer for more details.

\begin{prop}
 A Gorenstein stable surface $X$ with $K_X^2 = 1$,  $\chi(X)  = 2$ and $\Aut(X/\IP^2) = \IZ/2^2$ is equivalent to the datum of a line $D_0$ and two cubic curves $D_1, D_2$ in the plane, the \emph{building data}, such that
 \begin{enumerate}
  \item the pair $\left( \IP^2, \Delta\right)$ is log-canonical, where $\Delta = \frac12(D_0+D_1+D_2)$ is the Hurwitz divisor, 
  \item the triple intersection $D_0\cap D_1\cap D_2 = \emptyset$.
 \end{enumerate}
Furthermore, $X$ is non-normal above an irreducible curve  $\Gamma$ of $\IP^2$ if and only if $\Gamma$ appears in $\Delta$ with coefficient  $1$.
\end{prop}
In terms of the equations \eqref{CanRing}, $x_0^2$ is the equation of the line $D_0$ and the two cubics are given by $D_i = \{ b_i =0\}$.

\begin{rem}\label{rem: normalise}
It is often convenient to find a resolution of a singular bi-double cover by taking a log-resolution of $2\Delta = D_0+D_1+D_2$ and then applying the normalisation algorithm (see \cite[\S 3]{pardini91}) to the pullback of the building data. In our case this is straightforward and consists of two steps: let $\tilde D_i$ be the pullback of $D_i$ and note that the triple intersection is still empty. 
\begin{enumerate}
 \item Reduce the coefficients of $\tilde D_i$ modulo $2$ to get a reduced divisor.
 \item If $\{i,j,k\} = \{0,1,2\}$ and a curve $\Gamma$ is contained in $\tilde D_i$ and $\tilde D_j$, then consider the new building data $\tilde D_i -\Gamma$, $\tilde D_j - \Gamma$, $\tilde D_k+\Gamma$. 
\end{enumerate}
Repeating the second step as necessary, we arrive at a normal bi-double cover, the normalisation of the pullback. 
\end{rem}

The  possible singularities  of slc $(\IZ/2)^r$-covers such that the support of the Hurwitz divisor $\Delta$  has ordinary singularities have been classified in \cite[Table~1--4]{alexeev-pardini12}. In our restricted situation only two different normal singularities 
can occur, because $r=2$ and  $D_1$  and $D_2$ can have at most  three local branches through a point $P$:
\begin{itemize}
 \item $D=2\Delta$ has an ordinary quadruple point at $P$, such that three of the local components are in the same $D_i$. The resulting singularity is an elliptic singularity of degree $1$ (see case 4.5 in loc.cit.).
 \item Both $D_1$ and $D_2$ have an ordinary double point at $P$ such that $D$ has an ordinary quadruple point at $P$. The resulting singularity is an elliptic singularity of degree $4$ (see case 4.6 in loc.cit.).  
\end{itemize}

\section{Strata of normal surfaces}\label{ch: normal surfaces}
We attack the strata parametrising normal surface by direct methods. The starting points are the numerical restrictions following from \cite{FPR15a}.
\begin{theo}\label{thm: normalCase}
  Let $X$ be a normal Gorenstein stable surface with $K^2= 1$  and $\chi(X) = 2$.
  Let $\epsilon\colon \tilde{X} \to X$  and $\eta \colon \tilde X \to  \tilde
  X_{min}$ be a morphism to a minimal model.

  Then only the following cases can occur; we list some invariants  in Table \ref{tab: normal cases}.
  \begin{description}
  \item[$\kappa(\tilde{X}) = 2$] In this case $X$ has canonical singularities and $\tilde X = \tilde X_{min}$ is the corresponding minimal surface of general type.
  \item[$\kappa(\tilde{X}) = 1$] $\tilde{X} = \tilde X_{min}$ is a minimal properly elliptic surface and $X$ has precisely one elliptic singularity of degree $1$.
  \item[$\kappa(\tilde{X}) = 0$] There exists an effective nef divisor $D_{min}$ on $\tilde X_{min}$ and a point $P$ such that:
    \begin{enumerate}
    \item $D^2_{min}= 2$ and $p\in D_{min}$ has multiplicity $2$.
    \item $\eta \colon \tilde{X}\to \tilde X_{min}$ is the blow up at $P$.
    \item $X$ is obtained from $\tilde{X}$ by blowing down the strict transform of $D_{min}$ and it has either one elliptic singularity of degree 2 or two elliptic singularities of degree 1.
    \end{enumerate}
  \item[ $\kappa(\tilde{X}) = -\infty$] There are two possibilities:
    \begin{enumerate}
    \item $\chi({\tilde{X}}) = 1$ and $\tilde{X}$ has  one elliptic singularity;
    \item $\chi(\tilde{X}) = 0$, $\tilde{X}$ has two elliptic singularities; in this case, the exceptional divisors arising from the elliptic singularities are smooth elliptic curves.
    \end{enumerate}
    In both cases, the degree of the elliptic singularities is bounded by $4$.
  \end{description}

\end{theo}
\begin{table}[h]\caption{Strata of normal surfaces}\label{tab: normal cases}
  \begin{center}
    \begin{tabular}{ c c c c c}
      \toprule
      $\kappa(\tilde{X})$  &  \textbf{Elliptic sing.} $(d_1,..,d_r)$	&	 $ \chi(\tilde X)$ & \textbf{type} & Reference\\
      \midrule
      2 &  & 2  & gen. type   & \\
      \midrule
      1 & (1) & 1 & min. prop. ell & Section \ref{sect: properly ell case}\\
      \midrule
      0 & (2) & 1 &Enriques & Section \ref{sect: Enriques case} \\
                           & (1,1) & 0 & Torus  & Section \ref{sect: torus case}\\
                           & (1,1) & 0 & bielliptic  & Section \ref{sect: bielliptic case}\\
      \midrule
      $-\infty$ & $(d)$  & 1 & rational & Section \ref{sect: kappa negative case}\\
                           & $(d_1,d_2)$ & 0 & ruled over ell. curve & Section \ref{sect: kappa negative case}\\
      \bottomrule
    \end{tabular}
  \end{center}
\end{table}
\begin{proof}
  This follows quite directly from \cite[Theorem 4.1]{FPR15a}, combined with the Enriques classification of surfaces.
  Assume that $X$ has $r$ elliptic singularities. Then $\chi(X) - r = \chi(\tilde X) = \chi(\tilde X_{min}) $, so we can identify the number of elliptic singularities required in each of the numerical case.

  To bound the degree of the elliptic singularities, note that by Theorem \ref{thm: structure from FPR} the surface $X$ is embedded as a local complete intersection of codimension two in smooth variety. By the classification of elliptic singularities \cite{reid97}, or lci slc singularities (see \cite[Lemma 2.6]{tziolas09}) this excludes elliptic points of degree higher than 4.
\end{proof}
We will now consider some of these numerical cases in more detail.

\subsection{Surfaces with properly elliptic minimal resolution}\label{sect: properly ell case}
We now consider the following situation: Let $X$ be a normal Gorenstein stable surface with $K_X^2 = 1$ and $\chi(X) = 2$ such that its minimal resolution $\tilde X$ has Kodaira dimension $1$. Then by Theorem \ref{thm: normalCase} the surface $\tilde X$ is minimal properly elliptic and $\epsilon \colon \tilde X \to X$ contracts a unique  curve $E$---smooth elliptic or a cycle of rational curves---with $E^2 = -1$ and possibly some ADE configurations.
\begin{lem}\label{lem: invariants properly elliptic case AB}
  Let $\pi\colon\tilde X \to B$ be the elliptic fibration on $\tilde X$. Consider the sheaf $R^1\pi_*\ko_{\tilde{X}}$ and  denote its dual by $L : = {R^1\pi_*\ko_{\tilde X}}^{\vee} =\pi_*\omega_{\tilde X/B}$. Then we have
  \[ \epsilon^*K_X = K_{\tilde X} + E, K_{\tilde X}^2 = 0,
  \;
  E^2 = -1, K_{\tilde X} E = 1,\;
  \chi(\tilde X) = 1,\;
  p_g(\tilde X) = q(\tilde X), \]
  $L$ is a line bundle on $B$ of degree $1$ and one of the following cases occurs
  \begin{description}
  \item[Type A] $p_g(\tilde X) = q(\tilde X) = 0 = g(B)$
  \item[Type B] $p_g(\tilde X) = q(\tilde X) = 1 = g(B)$
  \end{description}
  Moreover, in both cases we have $h^0(2K_{\tilde X})=2$.
\end{lem}
\begin{proof}
  First of all, the map $\epsilon\colon \tilde{X} \to X$ contracts a unique curve $E$, we have
  $\epsilon^*K_X = K_{\tilde X} + E$, $\chi(\ko_{\tilde{X}}) = \chi(\ko_X) - 1 =
  1$ and $ K_{\tilde{X}}^2 = K_X^2 -1 = 0$. The rest follows from the
  intersection numbers, adjunction,  Theorem \ref{thm: normalCase} and the equality $\deg L = \chi(\tilde X)$ from \cite[Sect. 7, Lemma 14]{Friedman}.
  
  To estimate $p_g(\tilde{X})$ we use that we have an injection
  $H^0(\tilde X, K_{\tilde X})\to H^0(\tilde X, \epsilon^*K_X)$, thus $p_g(\tilde X)\leq p_g(X) = 1$.


  For the last statement note that
  \[H^0(2K_{\tilde X})=H^0(2\epsilon^*K_{ X}-2E) \subset  H^0(2\epsilon^*K_X)=\epsilon^*H^0(2K_X).\] By Corollary \ref{cor: bicanonical section defines a quadruple cover},  $|2K_X|$ defines a quadruple cover of $\IP^2$. Let $x_0$ be the image of the elliptic singularity, that is, the image of $E$ under the composition $\tilde X \to X \to \IP^2$. Then the general line in $\IP^2$ does not contain $x_0$, so $h^0(2K_{\tilde X})\leq 2$. On the other hand, if $l$ is a line through $x_0$ then the pullback of $l$ contains at least twice the exceptional divisor $E$, since the elliptic singularity is a double point. Hence $h^0(2K_{\tilde X})= 2$ as claimed.
\end{proof}

Let us fix some notation: denote by  $F_i$ the reduced multiple fibres of $\pi$ with multiplicities $m_i$.
By the canonical bundle formula \cite[Thm. 15, Sect. 7]{Friedman} we have
\begin{equation}\label{eq: can bundle formula} K_{\tilde X} = \pi^*\left(K_B+L\right) + \tsum_{i=1}^r (m_i-1) F_i
\end{equation}

Let $p_i$ be the image of $F_i$ in $B$. By \cite[Section 7, Exercise 2]{Friedman} we have
\begin{equation}
  \label{eq: pluricanonical sections} H^0\left(mK_{\tilde X}\right)  = \pi^*H^0\left(B, m(K_B + L) + \tsum_{i=1}^r \lfloor \frac{m(m_i-1)}{m_i} \rfloor p_i\right).
\end{equation}

\subsubsection{Type A}\label{typeA}
In this case $g(B) = 0$ so  $B\isom \IP^1$. Plugging $\deg L = 1$ into \eqref{eq: pluricanonical sections} and using Lemma \ref{lem: invariants properly elliptic case AB} we compute
\[
  \begin{split}
    2 = h^0(2K_{\tilde{X}}) &= h^0\left(\IP^1, \ko_{\IP^1}\left(-2+\sum_{i=1}^r\left\lfloor \frac{2(m_i-1)}{m_i} \right\rfloor\right)\right)\\
    &= h^0\left(\IP^1, \ko_{\IP^1}(-2+r)\right)
    \\
    & = 1+r-2.
  \end{split}
\]
Thus there are exactly three multiple fibres.

Note that since $\pi$ is minimal, the curve $E$ cannot be contained in a fibre, so it is a $k$-multisection with $k\geq 2$, because $E$ has arithmetic genus $1$ and $B$ has genus $0$. Then $E.F_i = k/m_i$ and by \eqref{eq: can bundle formula} and Lemma \ref{lem: invariants properly elliptic case AB} we have
\[
  \begin{split}
    1 = K_{\tilde X}E & =  \pi^*(K_B+L).E + \tsum_{i=1}^3 (m_i-1) F_i.E  \\
    &= k\left(-1+\sum_{i=1}^3\frac{m_i-1}{m_i}\right)\\
    & \geq 2\left(-1+\sum_{i=1}^3\frac{m_i-1}{m_i}\right).
  \end{split}
\]
Clearly $(m_i-1)/m_i\geq 1/2$ and thus the only possiblility is $k = m_1 = m_2 = m_3 = 2$.
We have proved
\begin{lem}
  If $\pi\colon \tilde X \to B$ is of Type A then $E$ is a bisection and $\pi$ has exactly three double fibres with reductions $F_1, F_2, F_3$. In particular,  $F_iE = 1$.
\end{lem}

\begin{lem}
  There exist points $q_i\in F_i$ such that $q_i\not\in E$ and
  $(K_{\tilde X/B} +E)|_{F_i}$ is linearly equivalent to $q_i$.
\end{lem}
\begin{proof}
  Note that $\ko_{F_i}(F_i)$ is a non-trivial $2$-torsion bundle on $F_i$ by \cite[Thm. 15, Sect. 7]{Friedman}. We have
  \[K_{\tilde X/B} = K_{\tilde{X}} - \pi^*K_B = \pi^*L + \tsum_j F_j,\] thus
  \[\left(K_{\tilde X/B} +E\right)|_{F_i} = \left(\pi^*L +\tsum_j F_j + E\right)|_{F_i} = \left(E+F_i\right)|_{F_i}\]
  which has degree $1$ and thus is linearly equivalent to a unique effective divisor $q_i$. We have $q_i\not \in E$ because $F_i|_{F_i}$ is non-trivial.
\end{proof}

\begin{figure}
  \centering
  \resizebox{1.0\textwidth}{!}{
    \begin{tikzpicture}
      [arrow/.style = {->, every node/.style = { font = \scriptsize}}]
      \begin{scope}[xshift = 0cm, yshift = 4cm]
        \foreach \s in {0,...,2}{
          \draw[C2col]
          ($(-1,-1) + .75*(\s,0)$) -- ($(-1,1) + .75*(\s,0)$);
        };
        \foreach \s in {0,...,2}{
          \coordinate (P\s) at ($(-1,0.125) + .75*(\s,0.25)$);
        };
        \draw[looseness=1]
        (P0) to[out=180-70, in=5] +(180-30:0.75) node[](E){}
        (P0) to[out=-70,in=180+50] (P1)
        (P1) to[out=50,in=180-60] (P2)
        (P2) to[out=180+30, in=-20] (P1)
        (P1) to[out=180-20, in=40] (P0)
        (P0) to[out=180+40, in=-5] +(180+20:0.75);
        \draw[looseness=4]{
          (P2) -- +(-60:0.5) to[out=-60, in=+30] +(+30:0.2) -- (P2)
        };
        \foreach \s in {0,...,2}{
          \coordinate (q\s) at ($(P\s) - (0,0.875)$);
        };
        \foreach \s[count=\si] in {0,...,2}{
          \draw[C12col] (q\s) -- +(50:0.375);
          \draw[C12col] (q\s) to[out=180+50,in=100] +(180+70:0.875)
          node[xshift=-0.25cm]{$G_\si$};
        };
        \foreach \s[count=\si] in {0,...,2}{
          \node[C2col] at ($(P\s) + (0,1)$) {$\hat{F}_\si$};
        };
        \node[] at ($(E) +(-0.75,-0.5)$) {$\hat{E}=E$};
        \draw[rounded corners] (-3.0, 1.5) rectangle (1.5, -1.75);
        \node[] at (1.7,1.6) {$\hat{X}$};
        \draw[arrow] (-3.25,0) to
        node[below,rotate=40] {}
        node[above,xshift=-1.25cm] {blow up $q_1, q_2, q_3$}
        +(180+40:3.0);
        \draw[arrow] (1.75,0) to
        node[below] {}
        node[above,xshift=1.25cm] {contract $\hat{F}_i$ and ADE}
        +(-40:3);
        \draw[arrow] (-0.75,-2) to
        node[below,xshift=3.5em,yshift=1.2em] {$\pi^*\pi_{*}M \to M$}
        node[above,xshift=-1em] {$\hat{\vartheta}$}
        +(0,-2.75);
      \end{scope}

      \begin{scope}[xshift = -5.5cm, yshift = 0.25cm]
        \foreach \s in {0,...,2}{
          \draw[C2col]
          ($(-1,-1) + .75*(\s,0)$) -- ($(-1,1) + .75*(\s,0)$);
        };
        \foreach \s in {0,...,2}{
          \coordinate (P\s) at ($(-1,0.125) + .75*(\s,0.25)$);
        };
        \draw[looseness=1]
        (P0) to[out=180-70, in=5] +(180-30:0.75) node[](E){}
        (P0) to[out=-70,in=180+50] (P1)
        (P1) to[out=50,in=180-60] (P2)
        (P2) to[out=180+30, in=-20] (P1)
        (P1) to[out=180-20, in=40] (P0)
        (P0) to[out=180+40, in=-5] +(180+20:0.75);
        \draw[looseness=4]{
          (P2) -- +(-60:0.5) to[out=-60, in=+30] +(+30:0.2) -- (P2)
        };
        \foreach \s in {0,...,2}{
          \coordinate (q\s) at ($(P\s) - (0,0.875)$);
        };
        \foreach \s[count=\si] in {0,...,2}{
          \fill[C12col] (q\s) circle (2pt);
          \draw[C12col] (q\s) node[xshift=0.3cm]{$q_\si$};
        };
        \foreach \s[count=\si] in {0,...,2}{
          \node[C2col] at ($(P\s) + (0,1)$) {$F_\si$};
        };
        \node[] at ($(E) +(0,0.25)$) {$E$};
        \draw[rounded corners] (-2.0, 1.5) rectangle (1.5, -1.25);
        \draw[] (1.7,1.6) node[]{$\tilde{X}$};
        \draw[arrow,dashed] (0.0,-1.5) to
        node[below] { }
        node[above] { $\vartheta$}
        +(-20:3.);
        \draw[arrow] (-0.5,-1.5) to
        node[below] { }
        node[above,xshift=0.5em] { $\pi$}
        +(-40:6.0);
      \end{scope}

      \begin{scope}[xshift = 4.5cm, yshift = 0.25cm]
        \foreach \s in {0,...,2}{
          \draw[C12col]
          ($(-1,-1) + .75*(\s,0)$) -- ($(-1,1) + .75*(\s,0)$);
        };
        \foreach \s in {0,...,2}{
          \coordinate (P\s) at ($(-1,0.125) + .75*(\s,0.25)$);
        };
        \draw[looseness=1]
        (P0) to[out=180-70, in=5] +(180-30:0.75) node[](E){}
        (P0) to[out=-70,in=180+50] (P1)
        (P1) to[out=50,in=180-60] (P2)
        (P2) to[out=180+30, in=-20] (P1)
        (P1) to[out=180-20, in=40] (P0)
        (P0) to[out=180+40, in=-5] +(180+20:0.75);
        \draw[looseness=4]{
          (P2) -- +(-60:0.5) to[out=-60, in=+30] +(+30:0.2) -- (P2)
        };
        \node[C2col,font=\scriptsize](label) at ($(P1)+(0,0.875)$)
        {elliptic singularities};
        \draw[arrow,C2col] (label) -- ($(P0) +(0.1,0.125)$);
        \foreach \s[count=\si] in {0,...,2}{
          \fill[C2col] (P\s) circle (2pt);
        };
        \node[] at ($(E) +(0,-1.0) $) {$\bar{E}$};
        \draw[rounded corners] (-2.0, 1.5) rectangle (1.5, -1.25);
        \draw[] (1.7,1.6) node[]{$\bar{X}$};
        \draw[arrow,dashed] (-0.5,-1.5) to
        node[above, xshift=-0.5cm, font=\scriptsize] {double cover}
        node[below] { $\bar{\vartheta}$}
        +(200:3.);
        \draw[arrow] (0,-1.5) to
        node[below] { }
        node[above,xshift=-0.5em] { $\bar{\pi}$}
        +(180+40:6.0);
      \end{scope}

      \begin{scope}[xshift = -0.5cm, yshift = -2.5cm]
        \foreach \s in {0,...,2}{
          \draw[C12col]
          ($(-1,-1) + .75*(\s,0)$) -- ($(-1,1) + .75*(\s,0)$);
          \draw[C2col,dashed]
          ($(-1,-1) + .75*(\s,0)$) -- ($(-1,1) + .75*(\s,0)$);
        };
        \foreach \s in {0,...,2}{
          \coordinate (P\s) at ($(-1,0.125) + .75*(\s,0.25)$);
        };
        \foreach \s[count=\si] in {0,...,2}{
          \draw[C2col,looseness=1]
          {(P\s) to[out=100, in=-10] +(180-40:0.375)}
          {(P\s) to[out=-100, in=10] +(180+40:0.375)}
          {(P\s) to[out=80, in=190] +(40:0.375)}
          {(P\s) to[out=-80, in=170] +(-40:0.375)};
        };
        \foreach \s[count=\si] in {0,...,2}{
          \fill[C2col] (P\s) circle (2.pt);
        };
        \draw[] (-1.5,-0.5) coordinate (T0) -- (1.0,-0.5);
        \node[font=\scriptsize] at ([yshift=0.25cm]T0) {$-1$};
        \node[font=\scriptsize] at ([yshift=-0.25cm]T0){$C_0$};
        \draw[rounded corners] (-2.0, 1.5) rectangle (1.5, -1.25);
        \node[] at (1.8,1.7) {$\IP(\pi_{*}M) \cong \IF_1 $};
        \draw[arrow] (-0.25,-1.5) to
        node[below] { }
        node[above,rotate=-40]{}
        +(-90:1) coordinate (IP1);
        \node at ([yshift=-0.375cm]IP1) {$\IP^1$};
      \end{scope}
    \end{tikzpicture}
  } 
  \caption{Properly elliptic case, type A} \label{fig: properly elliptic type A}
\end{figure}
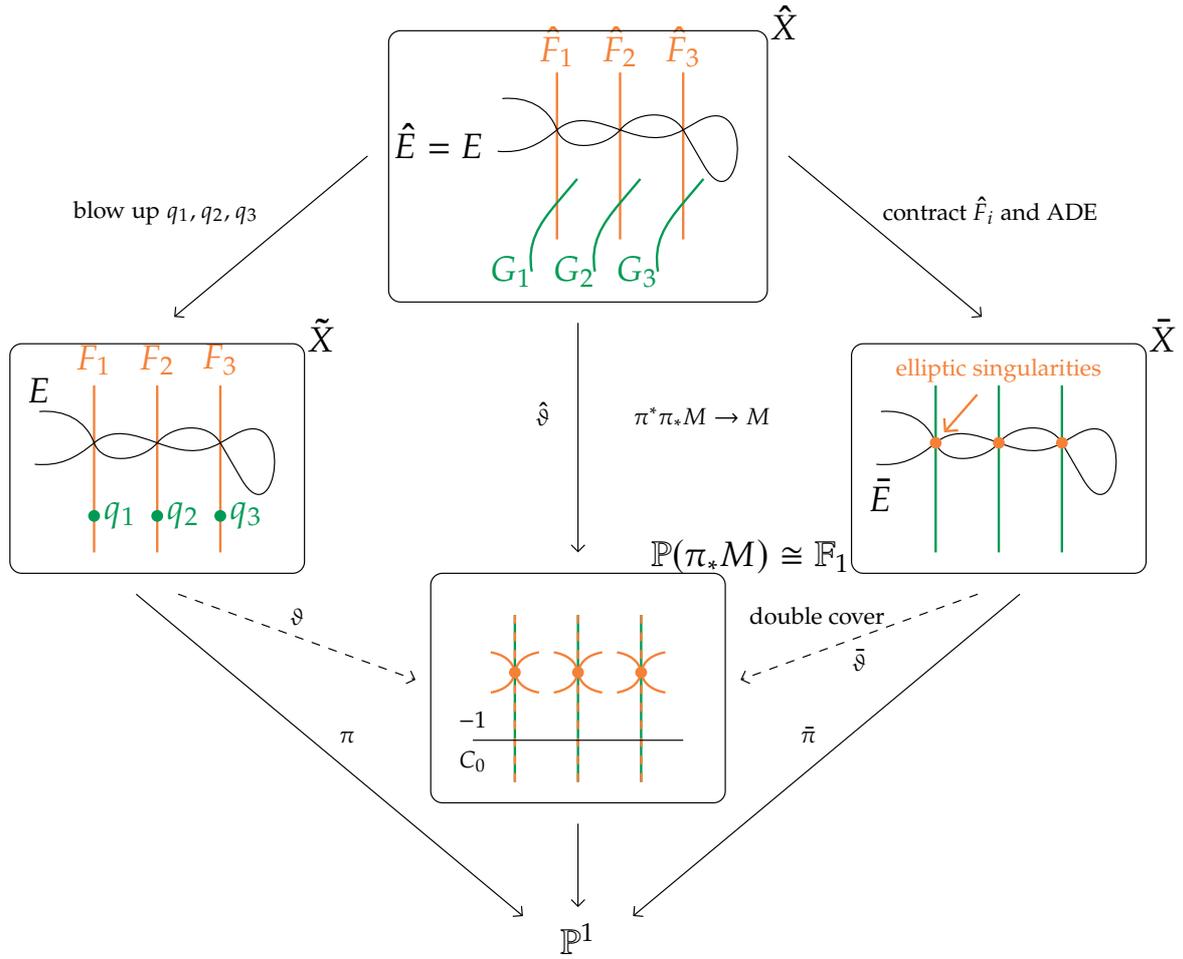

Now we consider $\sigma \colon \hat X = \mathrm{Bl}_{q_1, q_2, q_3}(\tilde X)\to \tilde X $ and denote the exceptional curve over $q_i$ with $G_i$. Let $\hat \pi = \pi \circ \sigma\colon \hat X \to B$ be the induced fibration.

Let $\hat E$ respectively $\hat F_i$ be the strict transforms of $E$ and the $F_i$ in $\hat X$.
Let $\bar\sigma\colon \hat X \to \bar X$ be the contraction of the curves $\hat F_i$ and possibly of ADE-configurations in the singular fibres of $\hat \pi$, which do not intersect the bi-section $\hat E$.

\begin{lem}\label{lem: lemma for type A}
  Consider on $\hat X$ the line bundle
  \begin{equation}\label{eq: M} M  = K_{\hat X/B} + \hat E -2\sum_i G_i = \hat \pi ^* L +\hat E +\sum_i \hat F_i.
  \end{equation}
  Then the following properties hold:
  \begin{enumerate}
  \item $M|_{\hat F_i} \isom \ko_{\hat F_i}$ and $\ko_{\hat F_i}(\hat F_i) \isom \ko_{\hat F_i}(- \hat E)$;
  \item $M|_{n\hat F_i} \isom \ko_{n\hat F_i}$ for every $n\in \IN$ (and every $i=1,2,3$); \footnote{We are grateful to Andreas Krug for help with this item.}
  \item The sheaf $\bar M \isom \bar\sigma_*M$ is a line bundle and $M \isom \bar \sigma^*\bar M $;
  \item For every (scheme-theoretic) fibre $\bar X_b$ of $\bar\pi \colon \bar X \to \IP^1$ the line bundle $\bar M|_{\bar X_b}$ has two sections which define a base-point free pencil;
  \item
    Then the natural map $\hat\pi^*\hat\pi_*M \to M$ is surjective and induces a morphism $\hat\theta\colon \hat X \to \IP(\hat\pi_*M)$, which factors over $\bar X$ such that $\bar\theta\colon \bar X \to \IP(\hat\pi_*M)$ is a double cover.
  \end{enumerate}
\end{lem}
In total the following diagram arises, compare also Figure \ref{fig: properly elliptic type A}:
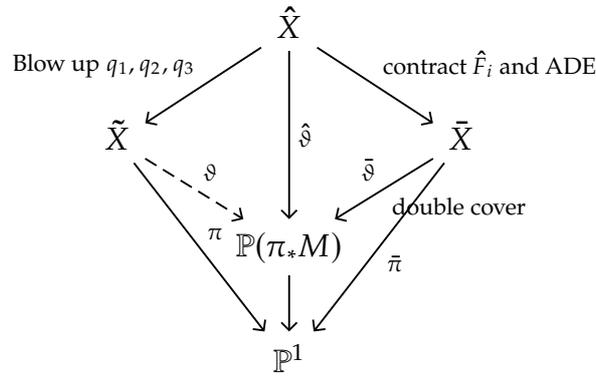
\[
  \begin{tikzcd}
    {}& \hat X \arrow{dl}[swap]{\text{Blow up $q_1, q_2, q_3$}} \arrow{dr}{\text{contract $\hat F_i$ and ADE}}\arrow{dd}{\hat\theta}\\
    \tilde X\arrow[dashed]{dr}{\theta}\arrow{ddr}{\pi}& & \bar X \arrow{dl}{\text{double cover}}[swap]{\bar\theta}\arrow{ddl}{\bar\pi}\\
    & \IP(\pi_*M)\dar\\
    & \IP^1
  \end{tikzcd}
\]

\begin{proof}
  Denote by $\bar G_i = \bar\sigma(G_i)$ the image of $G_i$ in $\bar X$ and also $\bar E = \bar \sigma(\hat E)$.  Note that $\hat \pi\colon \hat X \to \IP^1$ factors over a map $\bar \pi \colon  \bar X\to \IP^1$.

  \begin{enumerate}
  \item For this item we use the first description of $M$,  $M=K_{\hat{X}/B} + \hat{E} - \sum 2G_i$.
    By properties of blow-ups we have
    \[
      K_{\hat{X}} = \sigma^{*}K_{\tilde{X}} + \tsum G_i.
    \]
    Thus
    \[
      (K_{\hat{X}/B} + \hat{E} -\sum 2G_i) = \sigma^{*}\left(K_{\tilde{X}/B} + E\right) - \tsum G_i.
    \]
    The strict transform of $F_i$ is $\hat{F}_i \isom F_i$. Then
    $(K_{X/B} +E)|_{F_i} = q_i$,
    \begin{align*}
      M|_{\hat{F}_i}& =\left(\sigma^{*}(K_{\tilde{X}/B}+E)|_{F_i} -\tsum G_i\right)|_{\hat{F}_i}\\
      &= q_i - q_i = 0.
    \end{align*}
    Thus $M|_{\hat F_i} = \ko_{\hat F_i}$ and $\ko_{\hat F_i}(\hat F_i) \isom \ko_{\hat F_i}(- \hat E)$.
  \item We prove the assertion by induction on $n$. The $n=1$ case is already done. We set $F:=\hat F_i$ and consider the short exact sequence
    \begin{equation}\label{ses1}
      0\to \ko_F\left(-(n-1)F\right)\to \ko_{nF}\to \ko_{(n-1)F}\to 0\,.
    \end{equation}
    Tensoring with $M$ gives
    \begin{equation}\label{ses2}
      0\to \ko_F\bigl(-(n-1)F\bigr)\to M|_{nF}\to \ko_{(n-1)F}\to 0\,.
    \end{equation}
    where the triviality of the outer terms is due to the induction hypothesis.
    In the following, we will show that
    $\Ext^1_{\ko_{nF}}\Bigr(\ko_{(n-1)F},\ko_F\bigl(-(n-1)F\bigr)\Bigl)=\IC$, i.e.,
    there are two non-trivial extensions.
    The assertion then follows by comparing the two exact sequences \eqref{ses1} and
    \eqref{ses2}, as \eqref{ses2} cannot split since $M_{|nF}$ is a line
    bundle).

    Let $\alpha\colon F\hookrightarrow (n-1)F$ and $\iota\colon
    (n-1)F\hookrightarrow nF$ be the closed embeddings. We now use the language of derived functors for a convenient computation of 
    $\Ext^1_{\ko_{nF}}\Bigl(\ko_{(n-1)F},\ko_F\bigl(-(n-1)F\bigr)\Bigr)$.
    We have
\[(L\iota^*)\iota_*\ko_{(n-1)F}\cong \ko_{(n-1)F}[0]\oplus
    \ko_{(n-1)F}(-(n-1)F)[1]\] by  \cite[Thm.\ 0.7]{Arinkin2012} and note that
    $\ko_{(n-1)F}(-(n-1)F)$ is the conormal bundle of the embedding $\iota$.
    This gives
    \begin{align*}
      &\quad\Ext^1_{\ko_{nF}}\Bigr(\iota_*\ko_{(n-1)F},\iota_*\alpha_*\ko_F\bigl(-(n-1)F\bigr)\Bigl)\\
      &\cong
        \Ext^1_{\ko_{(n-1)F}}\Bigr((L\iota^*)\iota_*\ko_{(n-1)F},\alpha_*\ko_F\bigl(-(n-1)F\bigr)\Bigl)\\
      &\cong  \Ext^1_{\ko_{(n-1)F}}\Bigr(\ko_{(n-1)F},\alpha_*\ko_F\bigl(-(n-1)F\bigr)\Bigl)\\
      &\qquad \oplus \Ext^0_{\ko_{(n-1)F}}\Bigr(\ko_{(n-1)F}(-(n-1)F),\alpha_*\ko_F\bigl(-(n-1)F\bigr)\Bigl)\\
      &\cong
        H^1\bigl(F,\ko_F(-(n-1)F)\bigr)\oplus H^0\bigl(F,\ko_F\bigr)\\
      & \cong \IC\,.
    \end{align*}
  \item
    Let $p_i\in \bar X$ be the image of $\hat F_i$ in $\bar X$. By \cite[Exercise
    7.5]{Eisenbud2013} it is enough to prove that the completion of
    $\bar\sigma^*M$ at $p_i$ is free.

    By  the theorem of formal functions \cite[Thm. III.11.1]{Hartshorne} combined with the fact that the chosen  subscheme structure on the fibre does not affect the limit \cite[Rem. II.9.3.1]{Hartshorne} we have
    \begin{equation}\label{eq: formal functions}
      \widehat{ \bar \sigma_* M}^{p_i}
      \isom \varprojlim_n H^0(nF_i, M|_{nF_i})\isom \varprojlim_n H^0(nF_i, \ko_{nF_i})\isom \widehat{ \bar \sigma_* \ko_{\hat{X}}}^{p_i},
    \end{equation}
    where the middle isomorphism comes from the previous item.

    Since the contraction of an elliptic curve with self-intersection $-1$ leads to
    a hypersurface singularity \cite[Ch. 4]{reid97}  we
    have $\bar\sigma_*\ko_{\hat X} \isom \ko_{\bar X}$. So indeed the right hand
    side of \eqref{eq: formal functions} is free.

  \item This is easily computed on the fibres of the form $2\bar G_i$ and clear on the general fibre.
  \item By base change and the previous step  $\hat\pi_*M  = \bar\pi_*\hat M $ is a vector bundle of rank $2$ and $\bar\pi^*\bar\pi_* \bar M \to \bar M$ is surjective because it is fibrewise base-point free.
    Fibrewise $\bar X \to \IP(\hat\pi_*M)$ is a double cover, so it is a double cover.\qedhere
  \end{enumerate}
\end{proof}

\begin{lem}
  We have $\hat\pi_*M \isom \ko_B(1)\oplus \ko_B(2)$ and thus $P = \IP(\hat\pi_*M)\isom \IF_1$. Let $C_0$ be the unique $(-1)$-curve in $P$ and $F$ a general fibre of the projection to $B = \IP^1$. Then the tautological bundle of $P$ is $\ko_P(1) = \ko_P(C_0 + 2F)$ and $\theta_*\hat E$ is  an irreducible curve in $|C_0+F|$.
\end{lem}
\begin{proof}
  First note that for each $i$, $\sigma^* F_i = \hat F_i + G_i$, thus
  \[\sigma_* \ko_{\hat X}(\tsum_i \hat F_i ) =  \sigma_*\ko_{\hat X}(\tsum_i \sigma^*F_i - G_i) =  \ki_{\{q_1, q_2, q_3\}}(\tsum_i F_i)\subset \ko_{\tilde X}(\tsum_i F_i),\]
  so that by \cite[Ch.7, Exercise 2]{Friedman} we have
  \[ \hat\pi_* \ko_{\hat X}(\tsum_i \hat F_i )\subset \pi_*\ko_{\tilde X}(\tsum_i F_i) = \ko_B.\]
  Since the left hand side has a global section, this inclusion is an equality. This implies that the pushforward $\hat\pi_*(\hat \pi^*L + \sum_i \hat F_i ) = L.$
  Now consider the exact sequence
  \[ 0 \to \hat\pi^*L + \tsum_i \hat F_i \to M \to M|_{\hat E} = K_{\hat E/B}\to 0.\]
  Applying $\hat\pi_*$ we get by the projection formula, the above computation  and using both descriptions of $M$ from \eqref{eq: M},
  \begin{equation}\label{eq: ex seq type A}
    \begin{split}
      & 0\to L = \hat\pi_*\left(\hat \pi^*L + \tsum_i \hat F_i \right) \to \hat\pi_* M \to \hat \pi_*K_{\hat E/B}  = \left(\hat\pi_*\ko_{\hat E}\right)^\vee \to\\
      &\qquad  R^1\hat\pi_*\left(K_{\hat X/B} -\tsum_i 2G_i\right) \to R^1\hat\pi_* M\to ...
    \end{split}
  \end{equation}
  By relative duality we have
  \[R^1\hat\pi_*\left(K_{\hat X/B} -\tsum_i 2G_i\right) = R^1\hat\pi_*\shom\left( \tsum_i 2G_i ,  K_{\hat X/B}\right) = \shom\left(\hat\pi_*\left(\tsum_i 2G_i\right), \ko_B\right)\isom \ko_B,\] where the last identification is proved by pushing forward the exact sequence
  \[ 0 \to \ko_{\hat X} \to\ko_{\hat X}\left(\tsum_i 2G_i\right) \to \ko_{\sum_i 2G_i}\left(\tsum_i 2G_i\right)\to 0 .\]
  Indeed, $\hat\pi_* \ko_{\sum_i 2G_i}(\sum_i 2G_i) $ is a sum of skyscraper sheaves supported at the images of the $G_i$ with stalks $H^0(2G_i, \ko_{2G_i}(2G_i))$. These are zero, because $G_i^2 = -1$.

  Repeating this for $M = K_{\hat X/B} -\sum_i 2G_i + \hat E$ we get
  \[R^1\hat\pi_* M = \left(\hat\pi_*\ko_{\hat X}(\tsum_i 2G_i-\hat E)\right)^\vee=0\]
  because $\ko_{\hat X}(\sum_i 2G_i-\hat E)$ restricted to the general fibre has negative degree and thus no sections, so $\hat\pi_*\ko_{\hat X}(\sum_i 2G_i-\hat E)$ is a torsion sheaf and its dual is trivial.

  Therefore the  sequence \eqref{eq: ex seq type A} is isomorphic to
  \[ 0 \to \ko_B(1) \to \hat\pi_*M \to \ko_B \oplus \ko_B(2) \to \ko_B\to 0,\]
  and since $\Hom(\ko_B(2), \ko_B) = 0$ and $\Ext^1(\ko_B(2), \ko_B(1)) = 0$ we have $\hat\pi_*M \isom \ko(1)\oplus \ko_B(2)$.

  We now use the theory of ruled surfaces, compare  \cite[ Section V] {Hartshorne}.
  Let $\IF_1 = \IP(\ko_B(-1) \oplus \ko_B)$ so that the relative tautological bundle is $\ko_{\IF_1}(1) = \ko_{\IF_1}(C_0)$. Since $\hat\pi_*M = (\ko_B(-1) \oplus \ko_B)\tensor \ko_B(2)$ we have $\ko_P(1) = C_0 + 2F$.

  Since $\bar X \to P$ is a double cover $\hat\theta (\hat E)$ is a section of the projective bundle, that is, is an irreducible curve in a linear system $|C_0 + kF|$ for some $k$. To determine $k$ we compute by the projection formula
  \begin{multline*} k-1 = (kF +C_0)C_0 = \hat\theta_*\hat E C_0 = 2\hat E. \hat \theta^*C_0\\
    = 2 \hat E. \hat \theta^*(\ko_P(1) -2F) = 2\hat E.(M - \pi^*\ko_B(2)) ,
  \end{multline*}
  By the definition of $M$ and the projection formula we get $\hat E.(M - \pi^*\ko_B(2))  = 0$, thus $k = 1$ as claimed.
\end{proof}

We now determine the geometry and class of the ramification divisor $R \subset P$ of the double cover $\bar \theta \colon \bar X \to P$. Since the general fibre of $\hat\pi$ is an elliptic curve, the ramification divisor intersects the general fibre $P$ in four points and we have $R \sim 4C_0 + kF$ for some $k$.

Now to determine $k$  we write $\bar \theta_*\ko_{\bar X} \isom \ko_P \oplus  \inverse{\curly L} $ so that $R \in |2\curly L| = |2(2C_0 + k/2 F)|$ and  compute
\[ 1 = \chi(\tilde X ) = \chi(\hat X)  = \chi(\bar X) -3 = \chi(\bar\theta_*\ko_{\bar X}) - 3.\]
This implies $ \chi(\inverse{\curly L}) = 3$ and, using Riemann Roch with $K_P = - 2C_0 - 3 F$, we have
\[ 3 = \chi(\inverse{\curly L})= 1 + \frac12\left( -2C_0 -\frac k2 F\right)\left(-2 C_0 - \frac k2 F + 2 C_0  + 3 F \right) \implies k = 10 .\]

In other words $R \in |4C_0 + 10 F|$. We can say  more about the ramification $R$ of the double cover $\bar \theta$:
first of all $\bar\pi\colon \bar X \to B$ has exactly three double fibres while $P$ has none, so we can write $R = \tilde L_1 + \tilde L_2 + \tilde L_3 +\tilde C$ for a curve $C\in |7F + 4C_0|$ and the three fibres of $P\to B$ sitting over $p_1, p_2, p_3$.

Note that $\bar X$ has three elliptic singularities of degree $1$ and that image of $\hat E$ passes through all three of them. In addition, these elliptic singularities are contained in double fibres of $\bar \pi\colon \bar X \to B$.
By the classification of singularities of double covers \cite{FPR17, Anthes2018}
this means that $R$ has possibly degenerate $[3,3]$ points at the intersection
points $\tilde L_i \cap \hat \theta(\hat E)$ and ADE singularities elsewhere,
because $\bar X$ has no further non-canonical singularities.

Note that $C_0$ and $E$ are disjoint sections, because all irreducible curves in $|C_0+F|$ do not meet $C_0$.

Now let $\alpha\colon P \to \IP^2$ be the blow down of $C_0$. Then $\alpha(\hat\theta(\hat E))$ is a line in $\IP^2$, disjoint from the point we blow up. We may choose coordinates such that
\begin{itemize}
\item $C_0$ maps to $p =(0:0:1)$,
\item  $\alpha(\hat\theta(\hat E))$ is the line $\{z = 0\}$,
\item $\alpha(\tilde L_1)  = L_1 = \{ x = 0\}$,
\item $\alpha(\tilde L_2)  = L_2 = \{ y = 0\}$,
\item $\alpha(\tilde L_3)  = L_3 = \{ x-y = 0\}$,
\end{itemize}
Write $4C_0 + 7F = 7(C_0 + F) - 3C_0 = 7\hat \theta_*{\hat E} - 3 C_0 =7\alpha^*\ko_{\IP^2}(1) - 3C_0$ we have that
\begin{align*}
  H^0(P, 4C_0 + 7F) &\isom H^0(P, 7\alpha^*\ko_{\IP^2}(1) - 3C_0 ) \\
                    &\isom H^0(\IP^2, \alpha_*(\alpha^*\ko_{\IP^2}(7) - 3C_0 ))\\
                    &\isom H^0(\IP^2, \ko_{\IP^2}(7)\tensor \alpha_*(\ko_P(- 3C_0) ))\\
                    &\isom H^0( \IP^2, \ki_p^3(7)),
\end{align*}
and collecting the information from above we see that $\alpha(\tilde C) = C$ is a plane septic with the following properties
\begin{enumerate}
\item $C$ has at least a triple point at $p$ but $\alpha^*C-3C_0$ has ADE singularities near $C_0$,
\item $C$ has (possibly degenerate) $[2,2]$ points  at $(1:0:0)$, $(0:1:0)$, and $(1:1:0)$ which are tangent to the $L_i$, that is, $C + \sum_i L_I$ has three $[3,3]$ points at these points,
\item elsewhere $C$ has at most ADE singularities.
\end{enumerate}

Going backwards we have proved.
\begin{prop}
  Every surface of type A arise from a plane septic via the above description.
\end{prop}

We should finally remark, that we have not been talking about the empty set. 

\begin{exam}
This is a reproduction of \cite[Example $Z_1$]{FPR17}. We use the notation from Section \ref{sect: bidouble}. 
  Let $D_1$ be a union of three general lines through a point $P$ on  $D_2$. The bi-double cover $X$ has an elliptic singularity of degree $1$ and, by Theorem \ref{thm: normalCase}, the minimal resolution $\tilde X$ is  a minimal  properly elliptic surface. The pencil of lines through $P$ induces an elliptic pencil on $\tilde X$ and each component of $D_1$ gives a double fibre. We illustrate our description of Type A in this case in Figure~\ref{fig:Type_A_ex}. 
The three points $q_i$ are intersection points of $D_0$ and $D_1$. 

 
  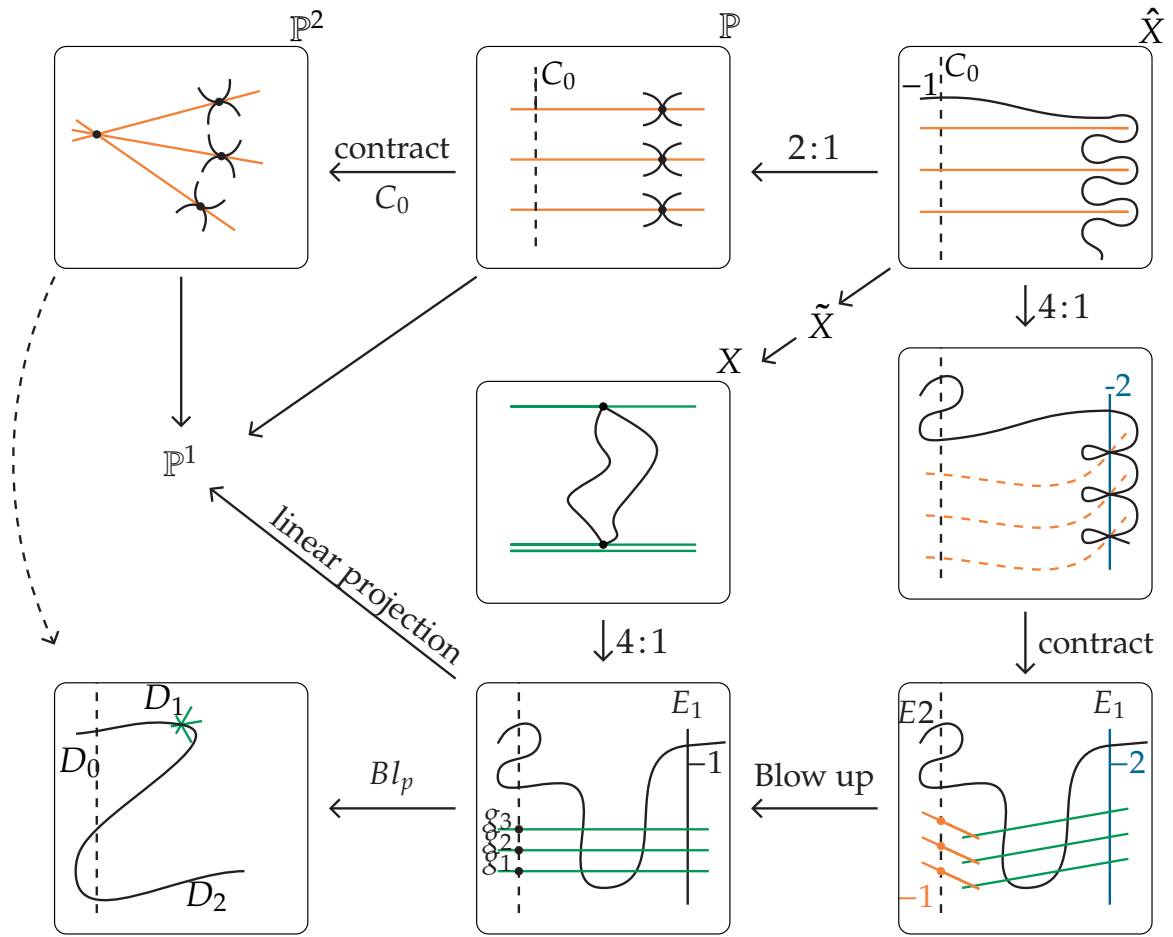
\begin{figure}
    \centering
    \resizebox{1.0\textwidth}{!}{
      \begin{tikzpicture}
        \begin{scope}[xshift = 10cm, yshift = 0cm]
          \draw[Cblack, dashed] (0, 0.2) coordinate (c0) to +(0, 2.5);

          \foreach \ii in {0, ..., 12}{
            \coordinate (x\ii) at ($(2., 0.4) + \ii*(0,0.125)$);
          }

          \foreach \ii in {1, 5, 9}{
            \draw[Corange] ([yshift=0.25cm]x\ii) -- +(0.22,0) -- +(-2.25, 0);
          }

          \draw[Cblack, looseness=1.2]
          (x12)
          to[out=180+2,in=5] ($(x12)+(-2.25,0.225)$) coordinate(t0)
          \foreach \ii in {12, 8}
          {
            (x\ii)
            to[out=20, in=90] ($(x\ii) + (0.325,-0.125)$)
            to[out=-90, in=-30] ($(x\ii) - (0, 0.25)$)
            to[out=150, in=90] ($(x\ii) - (0.325,0.375)$)
            to[out=-90, in=180+30] ($(x\ii) - (0,0.5)$)
          };
          \draw[Cblack, looseness=1.2]
          (x4)
          to[out=20, in=90] ($(x4) + (0.325,-0.125)$)
          to[out=-90, in=-30] ($(x4) - (0, 0.25)$)
          to[out=150, in=90] ($(x4) - (0.325,0.375)$)
          to[out=-90, in=70] ($(x0) +(-0.1,-0.2)$);
          \node[] at (-0.25,2.3) {\small{$-1$}};
          \node[] at (0.25, 2.5) {\small{$C_0$}};
          \node[] at (2.5,3.0) {$\hat{X}$};
          \draw[->, thick, Cblack] (-0.75, 1.25) to node[above] {$2\!:\!1$}  +(-1.5,0);
          \draw[->, thick, Cblack] (1.0, -.1) to node[right] {$4\!:\!1$}  +(0,-0.5);
          \draw[->, thick, Cblack] (-0.6, 0.1) -- ++(180+35:0.75cm) coordinate (px);
          \node[xshift=-0.20cm, yshift=-0.2cm] at (px) {$\tilde{X}$};
          \draw[->, thick, Cblack] ([xshift=-0.5cm, yshift=-0.4cm]px) -- ++(180+35:0.5cm);
          \draw[rounded corners] (-0.5, 2.75) rectangle (2.5, 0.1);
        \end{scope}

        \begin{scope}[xshift = 5.0cm, yshift = 0cm]
          \coordinate (P) at (0.2,2.0);
          \foreach \s[count=\si] in {0,0.6,1.2}{
            \draw[Corange,rotate=-90]
            {($(P)+(\s, 0)$) -- ++(0,1.5) coordinate (p\si) -- +(0,.5)}
            {($(P)+(\s, 0)$) -- +(0,-0.3)};
            \draw[Cblack,rotate=\s]
            {(p\si) to[out=100, in=-10] +(180-40:0.3)}
            {(p\si) to[out=-100, in=10] +(180+40:0.3)}
            {(p\si) to[out=80, in=190] +(40:0.3)}
            {(p\si) to[out=-80, in=170] +(-40:0.3)};
            \fill[Cblack] (p\si) circle(0.05);
          };
          \draw[Cblack, dashed] (P) -- +(0,0.5) -- +(0,-1.7);
          \node[] at ([yshift=0.4cm, xshift=0.27cm]P) {\small{$C_0$}};
          \node[] at (2.5,3.0) {$\IP$};
          \draw[->, thick, Cblack] (-0.75, 1.25) to
          node[above] {\small{contract}} node[below]{\small{$C_0$}}  +(-1.5,0);
          \draw[->, thick, Cblack] (-0.5, -0) -- +(180+35:3.3cm);
          \draw[rounded corners] (-0.5, 2.75) rectangle (2.5, 0.1);
        \end{scope}

        \begin{scope}[xshift = 0cm, yshift = 0cm]
          \coordinate (P) at (0,1.7);
          \foreach \s[count=\si] in {-75,-100,-125}{
            \draw[Corange,rotate=\s]
            {(P) -- ++(0,1.5) coordinate (p\si) -- +(0,.5)}
            {(P) -- +(0,-0.3)};
            \draw[Cblack,rotate=\s]
            {(p\si) to[out=100, in=-10] +(180-40:0.3)}
            {(p\si) to[out=-100, in=10] +(180+40:0.3)}
            {(p\si) to[out=80, in=190] +(40:0.3)}
            {(p\si) to[out=-80, in=170] +(-40:0.3)};
            \fill[Cblack] (p\si) circle(0.05);
          };
          \fill[Cblack] (P) circle (0.05);
          \node[] at (2.5,3.0) {$\IP^2$};
          \draw[->, thick, Cblack] (1,-0.0) -- +(0, -1.8) node[yshift=-0.4cm]{$\IP^1$};
          \draw[->, thick, dashed, Cblack] (-0.5,0) arc (180-25:180+25:5.2);
          \draw[rounded corners] (-0.5, 2.75) rectangle (2.5, 0.1);
        \end{scope}

        \begin{scope}[xshift = 5.0 cm, yshift = -4cm]
          \coordinate (pt) at (1, 2.45);
          \coordinate (pb) at (1, 0.8);
          \draw[Cgreen] (pt) -- +(-1.1,0) -- +(1.1,0);
          \draw[Cgreen] (pb) -- +(-1.1,0) -- +(1.1,0);
          \draw[Cgreen] ([yshift=-0.075cm]pb) -- +(-1.1,0) -- +(1.1,0);

          \draw[Cblack, looseness=2] (pt) to[out=-120,in=30] ++(-100:0.9)
          to[out=180+30,in=120] (pb);
          \draw[Cblack, looseness=2] (pt) to[out=-30,in=50] ++(-80:1.2)
          to[out=180+50,in=20] (pb);

          \fill (pt) circle(0.05);
          \fill (pb) circle(0.05);
          \node[] at (2.5,3.0) {$X$};
          \draw[->, thick, Cblack] (1.0, -.1) to node[right] {$4\!:\!1$}  +(0,-0.5);
          \draw[rounded corners] (-0.5, 2.75) rectangle (2.5, 0.1);
        \end{scope}

        \begin{scope}[xshift = 10cm, yshift = -3.6cm]
          \draw[Cblack, dashed] (0, 0) coordinate (c0) to +(0, 2.75);
          \draw[Cblue] (2., 0.1) coordinate (c1) to +(0, 2.1);

          \foreach \ii in {0, ..., 13}{
            \coordinate (x\ii) at ($(2., 0.375) + \ii*(0,0.125)$);
          }
          \foreach \ii in {1, 5, 9}{
            \draw[Corange, dashed, shorten <=3]
            (x\ii) -- +(50:0.3)
            (x\ii)
            to[out=180+50,in=0] +(-2.25,-0.25)
            ;
          }
          \draw[Cblack, looseness=1.2]
          (x13)
          to[out=180,in=5] ($(x12)+(-2.,-0.225)$) coordinate(t0)
          to[out=-175,in=-90] ($(t0) +(-0.25,0.125)$)
          to[out=90,in=180+80] ($(t0) +(0.25,0.55)$)
          to[out=80, in=20] ($(t0) +(0,0.75)$)
          to[out=180+20, in=60]($(t0) +(-0.25,0.525)$)
          \foreach \ii in {13, 9, 5}{
            (x\ii)
            to[out=-10, in=90] ($(x\ii) + (0.325,-0.25)$)
            to[out=-90, in=10] ($(x\ii) - (0, 0.5)$)
            to[out=160, in=90] ($(x\ii) - (0.325,0.5)$)
            to[out=-90, in=180+20] ($(x\ii) - (0,0.5)$)
          }
          -- +(-20:0.25);
          \node[Cblue] at (2.1,2.3) {\small{-2}};
          \draw[rounded corners] (-0.5, 2.75) rectangle (2.5, -0.25);
          \draw[->, Cblack, thick] (1.0, -.4) to node[right] {\small{contract}}  +(0,-0.75);
        \end{scope}

        \begin{scope}[xshift = 10cm, yshift = -7.6cm]
          \draw[Cblack, dashed] (0, 0) coordinate (c0) to +(0, 2.75);
          \draw[Cblue] (2., 0.1) to +(0, 2.1);

          \coordinate (t0) at (0,1.5);
          \draw[Cblack, looseness=1.2]
          (t0)
          to[out=-175,in=-90] ($(t0) +(-0.25,0.125)$)
          to[out=90,in=180+80] ($(t0) +(0.25,0.55)$)
          to[out=80, in=20] ($(t0) +(0,0.75)$)
          to[out=180+20, in=60]($(t0) +(-0.25,0.525)$)
          (t0)
          to[out=5, in=180] ($(t0)+(0.5,0.05)$)
          to[out=-0,in=180] +(0.5,-1.25)
          to[out=0,in=180+5] (2,2)
          -- +(5:0.45);

          \foreach \ii in {1,2,3}{
            \draw[Cgreen] (0.25,\ii*0.3) -- +(10:2cm);
          }
          \foreach \ii[count=\iic] in {0.5, 0.8, 1.1}{
            \draw[Corange] (0,\ii) coordinate (ec\iic) -- +(-25:0.5cm)
            -- +(180-25:0.25cm);
            \fill[Corange] (ec\iic) circle(0.05cm);
          }
          \node[Cblack] at (2,2.5) {\small{$E_1$}};
          \node[Cblue] at (2.2,1.8) {\small{$-2$}};
          \node[Corange] at (-0.3, 0.2) {\small{$-1$}};
          \node[Cblack] at (-0.3, 2.4) {\small{$E2$}};
          \draw[->, thick, Cblack] (-0.75, 1.25) to node[above] {\small{Blow up}}  +(-1.5,0);
          \draw[rounded corners] (-0.5, 2.75) rectangle (2.5, -0.25);
        \end{scope}

        \begin{scope}[xshift = 5.0 cm, yshift = -7.6cm]
          \draw[Cblack, dashed] (0, 0) coordinate (c0) to +(0, 2.75);
          \draw[Cblack] (2., 0.1) to +(0, 2.1);

          \coordinate (t0) at (0,1.5);
          \draw[Cblack, looseness=1.2]
          (t0)
          to[out=-175,in=-90] ($(t0) +(-0.25,0.125)$)
          to[out=90,in=180+80] ($(t0) +(0.25,0.55)$)
          to[out=80, in=20] ($(t0) +(0,0.75)$)
          to[out=180+20, in=60]($(t0) +(-0.25,0.525)$)
          (t0)
          to[out=5, in=180] ($(t0)+(0.5,0.05)$)
          to[out=-0,in=180] +(0.5,-1.25)
          to[out=0,in=180+5] (2,2)
          -- +(5:0.45)
          ;

          \foreach \ii/\iic in {0.5/1, 0.75/2, 1/3}{
            \draw[Cgreen] (-0.25,\ii) node[yshift=0.1cm, Cblack] {\small{$g_\iic$}} -- (2.25,\ii);
            \fill[Cblack] (0,\ii) circle(0.05);
          }
          \node[Cblack] at (2,2.5) {\small{$E_1$}};
          \node[Cblack] at (2.2,1.8) {\small{$-1$}};
          \draw[rounded corners] (-0.5, 2.75) rectangle (2.5, -0.25);
          \draw[->, thick, Cblack] (-0.75, 2.8) to
          node[right, rotate=-38,yshift=0.2cm, xshift=-1.2cm]{\small{linear projection}}
          +(180-38: 3.7cm);
          \draw[->, thick, Cblack] (-0.75, 1.25) to node[above] {\small{$Bl_{p}$}}  +(-1.5,0);
        \end{scope}

        \begin{scope}[xshift = 0cm, yshift = -7.6cm]
          \coordinate (O) at (1,2.25);
          \foreach \ag in {10, -60, 60}{
            \draw[Cgreen] (O) -- +(\ag:0.25)
            (O) -- +(180+\ag:0.125);
          }

          \draw[Cblack]  (O)
          to[out=180-10,in=5] +(185:1.25)
          (O)
          to[out=-10, in=45] (0,1)
          to[out=180+45, in =90] +(-0.25,-0.5)
          to[out=-90, in=180] +(2,0)
          ;
          \node[] at (-0.2,1.8) {$D_0$};
          \node[xshift=-0.2cm, yshift=0.3cm] at (O) {$D_1$};
          \node[] at (1.3,0.2) {$D_2$};
          \draw[Cblack, dashed] (0, 0) coordinate (c0) to +(0, 2.75);
          \draw[rounded corners] (-0.5, 2.75) rectangle (2.5, -0.25);
        \end{scope}
      \end{tikzpicture}
    }
    \caption{ Surface with properly elliptic minimal resolution, Type A }
    \label{fig:Type_A_ex}
  \end{figure}

\end{exam}
\subsubsection{Type B}\label{typeB}

\begin{prop}
  If $\pi\colon \tilde X\to B$ is as in case B of Lemma \ref{lem: invariants properly elliptic case AB} then
  $B$ is an elliptic curve, $E$ is a section, hence smooth elliptic, $K_{\tilde X} = \pi^*L$.
\end{prop}
\begin{proof}
  Again $E$ cannot be contained in a fibre since the $\pi$ is a minimal elliptic fibration, so $E$ is a $k$-multisection.
  If we apply the canonical bundle formula \eqref{eq: can bundle formula}   with trivial $K_B$ then we get
  \[1 = E.K_{\tilde X} = E. \pi^*L + \sum_i (m_i-1) F_i . E= k +\sum_i (m_i-1) \]
  and hence $k = 1$ and no multiple fibres.
\end{proof}
\begin{rem}
  Surfaces of type B have Weierstrass models as in \cite[Sect. 7, Thm. 20]{Friedman} and are thus parametrised by an irreducible family.
\end{rem}

\begin{exam}
This is a reproduction of \cite[Example $Z_1'$]{FPR17}. We use the notation from Section \ref{sect: bidouble}. 

  Let $D_1$ be a union of three general lines  through $P \in D_0$ and $D_2$ a
  general cubic. Then $X$ has a unique elliptic singularity of degree $1$. Blowing
  up at $P$ and applying the normalisation algorithm from Remark \ref{rem: normalise} we see that the pencil of lines through $P$ induces the canonical pencil on $\tilde X$ (Figure~\ref{fig:type_B}).  It was claimed in \cite{FPR17}, that the elliptic fibration has $4$ multiple fibres, but examining the building data of the normalisation more closely, we see that the induced map $\tilde X\to \IP^1$ has disconnected fibres and the four lines $D_0 + D_1$ map to the branch points of the Stein factorisation. 
  \end{exam}
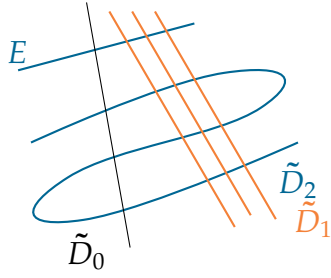
\begin{figure}
  \centering
  \begin{tikzpicture}
    [pfeil/.style = {->, every node/.style = { font = \scriptsize}},
    scale = .6]
    \begin{scope}[xshift=0cm, yshift=0cm]
      \draw[C1col, yshift=-1.0cm, xshift=3cm, rotate=20]
      plot [smooth, tension=1]
      coordinates { (-2.5,4) (2,4) (2.5,3)
        (-2,3) (-2,2) (3,2)} node[yshift=-0.5cm]{$\tilde{D}_2$};
      \draw[C1col] (-1,3.5) -- +(15:4cm);
      \draw[C1col] (-1,3.9) node[]{$E$};
      \draw[] (0.5,5) -- +(-80:5.5cm);
      \draw[C2col] (1,4.8) -- +(-60:5.5cm);
      \draw[C2col] (1.5,4.8) -- +(-60:5.2cm);
      \draw[C2col] (2.0,4.8) -- +(-60:5.3cm) node[xshift=0.5cm]{$\tilde{D}_1$};
      \draw[] (0.5cm,-0.5cm) node[]{$\tilde{D}_0$};
    \end{scope}
  \end{tikzpicture}
  \caption{ Building data for the minimal resolution of a surface with properly elliptic minimal resolution, Type B }
  \label{fig:type_B}
\end{figure}

\subsection{Enriques case}

\label{sect: Enriques case}
In this case, $X$ has a unique elliptic singularity of degree $2$, the minimal resolution $\tilde X$ is an Enriques surface blown up in one point, which is a node on a nodal curve of genus $2$. Thus, at least as a set, an open subset of  this stratum is in bijection to the isomorphism classes of pairs
\[ \ke = \left\{ (Y, C)\mid Y\text{ Enriques surface}, C \text{ nodal,non-smooth,  ample curve}, p_a(C) = 2\right\}.\]

This suggests to study this stratum via the presumably finite and dominant map to the moduli space of Enriques surfaces. While quite some information is available on the latter (see e.g. \cite{CossecDolgachev, GritsenkoHulek}), the construction and exploration of $\ke$ goes beyond the scope of this article.

Note however that $\ke$ is non-empty: either one can argue that in the linear system associated to a degree $2$ polarisation on an Enriques surface not every member can be smooth for topological reasons, so in general there is a curve with just one node and arithmetic genus $2$. Alternatively, an explicit example of a stable surface in this stratum was constructed in \cite[Example $Z_2^{E}$]{FPR17}. 

    \subsection{Torus case}
            \label{sect: torus case}
            Now we cover the case, where the minimal resolution is birational to an abelian variety, i.e., a projective complex torus.

            \begin{prop}
            Let $X$ be a Gorenstein stable surface with $K_X^2 = 1$ and $\chi(X) =2$ such that the minimal model of a resolution is a torus. Then there exist elliptic curves $E_1, E_2$ and a commutative diagram
            \[
                \begin{tikzcd}
                {}
                &\tilde X
                \arrow{dr}[swap]{\eta}
                \arrow{dl}{\epsilon}[swap]{\text{blow down $\hat E_1,\hat E_2$}}\\
                X\arrow{dd}[swap]{\phi}& &
                E_1\times E_2\ar{dd}{p_1\times p_2}\\
                & {\text{\scriptsize $(\IZ/2)^2$-covers}}\\
                \IP^2
                \ar[dashed]{rr}{\text{projections}}[swap]{\text{from $P$ and $Q$}}
                && \IP^1\times \IP^1
                \end{tikzcd}
            \]
            such that
            \begin{enumerate}
            \item the bicanonical map $\phi$ is a bi-double cover with building data a line $D_0$ containing two points $P$ and $Q$, $D_1$ three lines through $P$ and $D_2$ three lines through $Q$.
            \item $p_i\colon E_i \to \IP^1$ is the natural double cover.
            \item $\eta$ is the blow up of the intersection of point of $E_1\cup E_2 \subset E_1 \times E_2$.
            \item $\epsilon$ is the contraction of the strict transform of $E_1\cup E_2 $.
            \end{enumerate}
            Figure \ref{fig:Torus} depicts the change of building data under resolution of the birational transformation $\IP^2 \dashrightarrow \IP^1\times \IP^1$.
            \end{prop}
            \begin{proof}
            Let $\epsilon\colon  \tilde X \to X $ be the minimal resolution and $\eta \colon \tilde X  \to \tilde X_{\text{min}}$ be a map to a minimal model.
            By Theorem \ref{thm: normalCase} and our assumptions, $\tilde X_\text{min}$ is  a torus containing two elliptic curves $E_1$ and $E_2$ such that $E_1E_2 = 1$. Using the intersection point as origin for everything, the addition map $E_1\times E_2 \to \tilde X_\text{min}$ is an isomorphism.

            We have proved the properties attributed to the upper part of the diagram.

            Denoting the exceptional curve of $\eta$ with $F$ we have
            \begin{align*}
                \epsilon^*{K_X} &  = K_{\tilde X} + \hat E_1 + \hat E_2\\
                                & = \eta^*K_{E_1\times E_2} + F + \hat E_1 + \hat E_2\\
                                & = \eta^*(E_1+E_2) - F
            \end{align*}
            and thus
            \begin{align*}
                H^0(X, 2K_X) &\isom H^0(\tilde X, \epsilon^*2K_X)\\
                            & = H^0(\tilde X, 2\eta^*(E_1+E_2) - 2F)\\
                            & \isom H^0(E_1\times E_2, \ki_{E_1\cap E_2}^2(2E_1+2E_2))\\
                            & \subset  H^0(E_1\times E_2, 2E_1+2E_2)\\
                            & =H^0(E_1, p_1^*\ko_{\IP^1}(1)) \tensor H^0(E_2, p_2^*\ko_{\IP^1}(1))\\
                            & = p_1^*H^0(\IP^1, \ko_{\IP^1}(1))\tensor p_2^*H^0(\IP^1, \ko_{\IP^1}(1))\\
                            & = \langle x_1, y_1\rangle \tensor \langle x_2, y_2\rangle.
            \end{align*}
            where we choose the sections such that $x_i$ cuts out the divisor $2E_i$.
            With these coordinates and identifications, the bicanonical map of $X$ is defined by the sections $\langle x_1\tensor x_2, x_1\tensor y_2, y_1\tensor x_2\rangle$. On $\IP^1\times \IP^1$; these sections define exactly the inverse of the lower horizontal map in the diagram.

            Thus the diagram commutes and all remaining claims follow easily.
            \end{proof}

            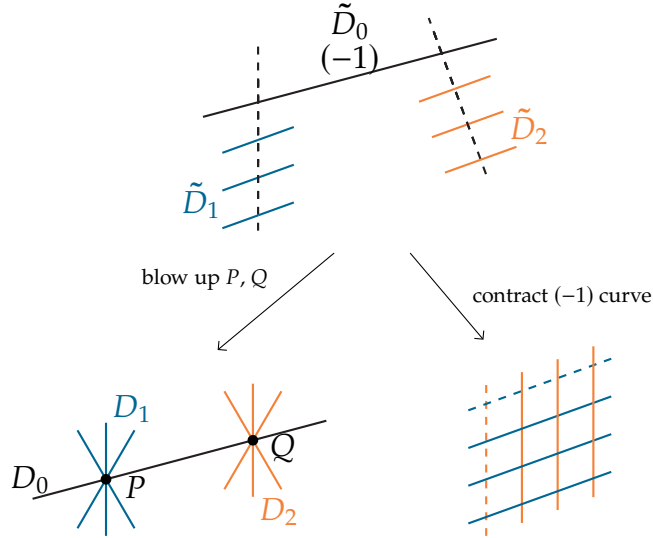
\begin{figure}
            \centering
            \begin{tikzpicture}
                [pfeil/.style = {->, every node/.style = { font = \scriptsize}},
                extended line/.style={shorten >=-#1,shorten <=-#1}]

                \begin{scope}
                \coordinate (P) at  (-1.5, -0.5 );
                \coordinate (Q) at  ($(P) +(15:2.5)$);
                \draw[Cblack, extended line = 0.75cm] (P) -- (Q);
                \draw[Cblack, dashed] (P) -- +(90:0.75);
                \foreach \s in {1, 2, 3}{
                    \draw[Cblack, dashed, extended line = 0.3cm]
                    (P) -- +(270:0.5*\s) coordinate (t\s);
                    \draw[C1col] (t\s) -- +(20:0.5);
                    \draw[C1col] (t\s) -- +(180 + 20:0.5);
                };
                \draw[Cblack, dashed] (Q) -- +(110:0.5);
                \foreach \s in {1, 2, 3}{
                    \draw[Cblack, dashed, extended line = 0.3cm] (Q) -- +(180+110:0.5*\s) coordinate (t\s);
                    \draw[C2col] (t\s) -- +(20:0.5);
                    \draw[C2col] (t\s) -- +(180 + 20:0.5);
                };
                \draw (P) +(180+60:1.5) node[C1col]{$\tilde{D}_1$};
                \draw (Q) +(-40:1.5) node[C2col]{$\tilde{D}_2$};
                \draw ($(P)!0.5!(Q)$)
                +(0.0,0.75) node[]{$\tilde{D}_0$}
                +(0, 0.25) node[]{$(-1)$};
                \draw[pfeil] (-0.5,-2.5) to
                node[above, xshift=-2.2em]{blow up $P$, $Q$}
                +(180+40:2.);
                \draw[pfeil] (.5,-2.5) to
                node[xshift=3.6em]{contract $(-1)$ curve}
                +(-50:1.5);
                \end{scope}

                \begin{scope}[xshift = -2.0cm, yshift = -5.0cm]
                \coordinate (P) at  (-1.5, -0.5 );
                \coordinate (Q) at  ($(P) +(15:2.0)$);
                \draw[Cblack, extended line = 1cm] (P) -- (Q);
                \foreach \angle in {60, 90, 120}{
                    \draw[C1col] (P) -- +(\angle:0.75);
                    \draw[C1col] (P) -- +(180 + \angle:0.75);
                };
                \foreach \angle in {60, 90, 120}{
                    \draw[C2col] (Q) -- +(\angle:0.75);
                    \draw[C2col] (Q) -- +(180 + \angle:0.75);
                };
                \fill (P) circle (2pt) +(0.4cm,-0.1cm) node[]{$P$};
                \fill (Q) circle (2pt) +(0.4cm,-0.1cm) node[]{$Q$};

                \draw (-2.5, -0.5) node[]{$D_0$};
                \draw (P) +(70:1) node[C1col]{$D_1$};
                \draw (Q) +(-70:1) node[C2col]{$D_2$};
                \end{scope}

                \begin{scope}[xshift = 3cm, yshift = -4cm]
                \coordinate (O) at (-1.5,-0.5);
                \draw[Corange, dashed] ($(O)+(90:0.25)$) -- ($(O)+(-90:1.75)$);
                \draw[Cblue, dashed] ($(O)+(180+20:0.25)$) -- ($(O)+(20:1.75)$);
                \foreach \s in {1,2,3}{
                    \coordinate (c\s) at ($(O)+(20:0.5*\s)$);
                    \draw[Corange] ($(c\s)+(90:0.25)$) -- ($(c\s)+(-90:1.75)$);
                    \coordinate (r\s) at ($(O)+(-90:0.5*\s)$);
                    \draw[Cblue] ($(r\s)+(180+20:0.25)$) -- ($(r\s)+(20:1.75)$);
                }
                \end{scope}
            \end{tikzpicture}
            \caption{ building data in the Torus case}
            \label{fig:Torus}
            \end{figure}

\subsection{Bielliptic surface case}
\label{sect: bielliptic case}
In this case, $\tilde X$ has two elliptic singularities of degree $1$ and its minimal resolution $\tilde X$ has Kodaira dimension $0$ and as minimal model a bi-elliptic surface.


We quickly recall the classification of bi-elliptic surfaces, which are the surfaces of Kodaira dimension $0$ with $\chi(\tilde X) = 0 = p_g(\tilde X)$.

Let $A, B$ be elliptic curve and $G$ a finite group acting on $A$ by translations and on $B$ such that $B/G\isom \IP^1$.
Then $\tilde X$ is of the form $\tilde X = A\times B/G$ and the possible cases where classified by  Bagnera and de Franchis \cite{BaSa07}. We give this classification in  Table \ref{tab: bielliptic}, where the $m_i$ are   the multiplicities of the multiple fibres of $\pi_2 \colon \tilde X \to B/G \isom \IP^1$.  Then let 
\[\mu = \text{lcm}\{m_1, ..., m_s\}\text{ and } \gamma = \text{ order of $G$},\].
Note that a basis of Num$(S)$ consists of divisors $A/ \mu$ and $(\mu/ \gamma)B$ where $A^2 = 0, B^2 = 0, AB = \gamma.$

Note that $\tilde X$ does not contain any rational curves, as it is finitely covered by an an abelian variety. 
\begin{lem}\label{lem: elliptic curve}
  If $E\subset \tilde X$ is an elliptic curve, then numerically $E$ is a positive multiple of $ A/ \mu$ or  $(\mu / \gamma)  B$.
\end{lem}
\begin{proof}
  By the adjunction Formula, we have $2g(E) - 2 = E (E + K_{\tilde X})$. The canonical divisor $K_{\tilde X}$ of each bielliptic surface is numerically  trivial. It follows that $E^2 = 0$, which  happens only when $E$ is a multiple of $ A/ \mu$ or $ (\mu / \gamma) B$; it  has to be positive, since $A.E\geq 0 $ and $B.E \geq 0$.
\end{proof}
\begin{center}
  \begin{table}\caption{Classification of bielliptic surfaces}\label{tab: bielliptic}
    \begin{center} \begin{tabular}{c|c|c|c}
                     Type of a bielliptic surface & G & $m_1, ..., m_s $& Basis of Num(S)\\
                     \hline
                     1& $\IZ_2$                     & 2, 2, 2, 2       & $A/2, B$\\
                     2 & $\IZ_2 \times \IZ_2$ & 2, 2, 2, 2      & $A/2, B/2$\\
                     3 & $\IZ_4$                     & 2, 4, 4         & $A/4, B$\\
                     4 & $\IZ_4 \times \IZ_2$ & 2, 4, 4          & $A/4, B/2$\\
                     5 & $\IZ_3$                     & 3, 3, 3         & $A/3, B$\\
                     6 & $\IZ_3 \times \IZ_3$ & 3, 3, 3         & $A/3, B/3$\\
                     7 & $\IZ_6$                     & 2, 3, 6         & $A/6, B$
                   \end{tabular}
                 \end{center}
               \end{table}
             \end{center}
             We get two fibrations on $\tilde X$, namely,
             \[ \begin{tikzcd}
                 {} & \tilde X \ar{dr}{\pi_2}  \ar{dl}[swap]{\pi_1}\\
                 A/G & & B/G
               \end{tikzcd}
             \]

             By Theorem \ref{thm: normalCase} we are looking for two (smooth) elliptic curves $E_1, E_2\subset \tilde X$ such that $E_1.E_2 = 1$.

             \begin{lem}
               Such a configuration exists if and only if $\tilde X$ is an odd bielliptic surface, that is, in cases 1,3,5,7 in Table \ref{tab: bielliptic}.

               Moreover, in this case up to automorphism $E_2 = B = \inverse \pi_1(0)$ and $E_1$ is the reduction of a multiple fibre of maximal multiplicity of $\pi_2$.
             \end{lem}
             \begin{proof}
               Assume there are two elliptic curves $E_1, E_2$ on $\tilde X$ such that $E_1 E_2 = 1$. By Lemma \ref{lem: elliptic curve} we can  write $E_1 = a  A/\mu$, $E_2 = b (\mu / \gamma) B$ for some positive integers $a,b$. Then
               \[1=E_1E_2 = \frac{ab}{\gamma} AB = ab,\]
               so we have $a=b=1$. By \cite[Lemma 2.7]{MR3538518} we get $\mu / \gamma = 1$, because numerically $\mu/\gamma B = E_2 $ is effective.  This happens exactly in the  cases $1, 3, 5, 7$ in the Table \ref{tab: bielliptic} as in \cite{BHPV} and the the only curves in these numerical classes are are the given ones.

             \end{proof}

             \begin{prop}
               There are three $1$-dimensional and one $2$-dimensional strata of normal  Gorenstein stable surfaces with $K_X^2 =1$ and $\chi(X) =2$ with minimal resolution birational to a bi-elliptic surface.
             \end{prop}
             \begin{proof}
               Note that in all cases in the table the elliptic curve $A$ is arbitrary. In case $1$, the curve $B$ can also be arbitrary, but in the other cases $B$ admits a larger group of automorphisms, hence is isomorphic to $\IC/\IZ[i]$ or $\IC/\IZ[\exp(2\pi i /3)]$. So the number of parameters is two in the first case and one in the other cases.
             \end{proof}

             \begin{exam}
               This example comes from \cite{FPR17} and again we use the notation from Section \ref{sect: bidouble}.  Let $D_1$ be a union of three general
               lines through $P \in D_0$ and $D_2$ be a union of three lines passing  through
               a general point $Q \in D_1$, see Figure~\ref{fig:Bielliptic}. Then $X$ has two elliptic singularities of degree
               $1$. Blowing up at $P$ and $Q$, the strict transform of the component of $D_1$ containing both $P$ and $Q$ is a $(-1)$-curve $E$. Contracting $E$ we get an induced bi-double cover $\phi\colon S \to \IP^1\times \IP^1$.  Then $\varphi_*\ko_S = \ko_{\IP^1 \times \IP^1} \oplus \sum_{i=1}^{3} L_i^{-1}$
               where
               \begin{align*}
                 L_1 &= \ko_{\IP^1 \times \IP^1}\left (\frac{2}{2}, \frac{4}{2}\right)\\
                 L_2 &= \ko_{\IP^1 \times \IP^1} (2, 0)\\
                 L_3 &= \ko_{\IP^1 \times \IP^1} (1,2)
               \end{align*}
               and 
               we conclude that $h^1(\ko_S) = 1$ and thus $S$ is a bielliptic surface.
              
               This  explicit example varies in a two-dimensional family, because the crossratio in the points where four lines meet can be arbitrary. Thus it  gives the family of surfaces of Type 1 in Table \ref{tab: bielliptic}.

               \begin{figure}
                 \centering
                 \begin{tikzpicture}[pfeil/.style = {->, every node/.style = { font = \scriptsize}},
                   extended line/.style={shorten >=-#1,shorten <=-#1},
                   scale=0.9]
                   \begin{scope}
                     \coordinate (p00) at (-1.5,1.5);
                     \foreach \ver in {0,1,2,3}
                     \foreach \hor in {0,1,2,3}
                     {
                       \coordinate (p\ver\hor) at ($(p00) +(0.75*\hor,-0.75*\ver)$);
                     };
                     \coordinate (pxx) at ($(p13) + (0,0.25)$);

                     \draw[C2col, dashed, extended line=0.5cm] (p00) -- (p03);
                     \draw[C1col, dashed, extended line=0.5cm] (p33) -- (p13);
                     \draw[extended line=0.5cm] (p00) -- (p30);
                     \draw[C1col, extended line=0.5cm] (p01) -- (p31);
                     \draw[C1col, extended line=0.5cm] (p02) -- (p32);
                     \draw[C2col, extended line=0.5cm] (p10) -- (p13);
                     \draw[C2col, extended line=0.5cm] (p20) -- (p23);
                     \draw[C2col, extended line=0.5cm] (p30) -- (p33);

                     \draw[C1col, looseness=1] (pxx) -- +(180+40:0.15); 
                     \draw[C1col, looseness=1] (pxx) -- +(40:0.1) to[out=40,in=180+60] +(60:1cm) node[yshift=0.5em]{$(-1)$};

                     \draw[pfeil] (-1.0,-1.5) to node[xshift=-2.8em]{blow up at P, Q} +(180+50:2.5);
                     \draw[pfeil] (0.5,-1.5) to node[xshift=3.6em]{contract -1 curve} +(-50:2.0);
                   \end{scope}

                   \begin{scope}[xshift = -3cm, yshift = -5cm]
                     \coordinate (P) at  (-1.5, -0.5 );
                     \coordinate (Q) at  ($(P) +(15:3.0)$);

                     \draw[C1col, extended line = 1cm] (P) -- (Q);
                     \draw[] (P) -- +(60:1.5);
                     \draw[] (P) -- +(180 + 60:0.75);
                     \foreach \angle in {90, 120}
                     {
                       \draw[C1col] (P) -- +(\angle:0.75);
                       \draw[C1col] (P) -- +(180 + \angle:0.75);
                     };
                     \foreach \angle in {60, 90, 120}
                     {
                       \draw[C2col] (Q) -- +(\angle:0.75);
                       \draw[C2col] (Q) -- +(180 + \angle:0.75);
                     };
                     \fill (P) circle (2pt) +(0.4cm,-0.1cm) node[]{$P$};
                     \fill (Q) circle (2pt) +(0.4cm,-0.1cm) node[]{$Q$};

                     \draw (P) +(75:1) node[]{$D_0$};
                     \draw (-2.5, -0.5) node[C1col]{$D_1$};
                     \draw (Q) +(-70:1) node[C2col]{$D_2$};
                   \end{scope}

                   \begin{scope}[xshift = 4cm, yshift = -5cm, minimum width = 10cm]
                     \coordinate (p00) at (-1.5,1.5);
                     \foreach \ver in {0,1,2,3}
                     \foreach \hor in {0,1,2,3}
                     \coordinate (p\ver\hor) at ($(p00) +(0.75*\hor,-0.75*\ver)$);
                     \coordinate (pxx) at ($(p13) + (0,0.25)$);

                     \draw[C2col, dashed, extended line=0.5cm] (p00) -- (p03);
                     \draw[C1col, dashed, extended line=0.5cm] (p33) -- (p03);
                     \draw[extended line=0.5cm] (p00) -- (p30);
                     \draw[C1col, extended line=0.5cm] (p01) -- (p31);
                     \draw[C1col, extended line=0.5cm] (p02) -- (p32);
                     \draw[C2col, extended line=0.5cm] (p10) -- (p13);
                     \draw[C2col, extended line=0.5cm] (p20) -- (p23);
                     \draw[C2col, extended line=0.5cm] (p30) -- (p33);

                     \fill[] (p03) circle(2pt);
                   \end{scope}
                 \end{tikzpicture}
                 \caption{Bielliptic case}
                 \label{fig:Bielliptic}
               \end{figure}
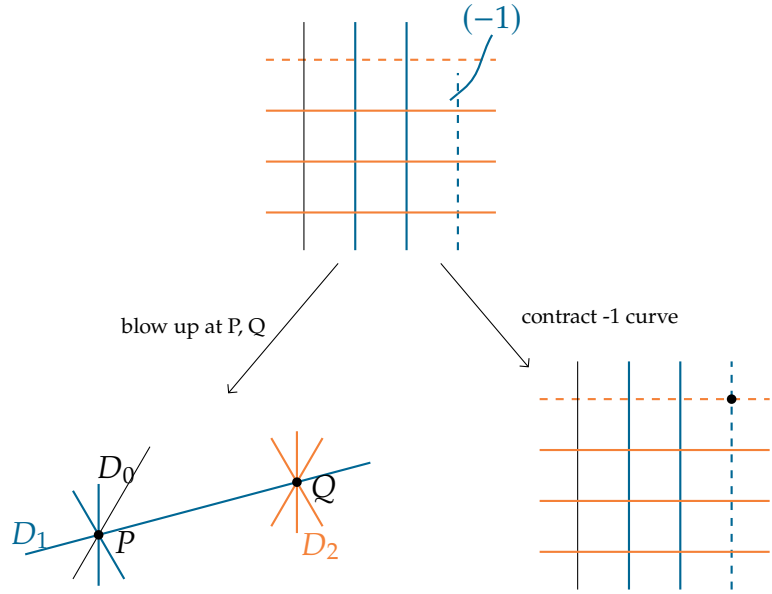
             \end{exam}

             \subsection{Minimal resolution of Kodaira dimension $-\infty$}
             \label{sect: kappa negative case}
              From Theorem \ref{thm: normalCase} we see that there is a finite number of possible cases. An example of a rational surface with a unique elliptic singularity of degree $4$ has been constructed as a bi-double cover \cite[Example $Z_4$]{FPR17}. An iterated double cover also gives a rational surface with a unique elliptic singularity of degree $2$, \cite[Example $Z_2^E$]{FPR17}. 
             
             A general understanding of these surfaces remains elusive for the moment. For example, can the minimal resolution be ruled over an elliptic curve?


\section{Strata of non-normal surfaces}
\label{ch: non normal surfaces}
To identify the non-normal Gorenstein stable surfaces with $K_X^2 = 1$ and $\chi(X) = 2$ we closely follow the strategy from \cite{FPR18}.
\subsection{Normalisation and gluing: starting point of the classification}\label{ssec: glue}

Let $X$ be a non-normal stable surface and $\pi\colon \bar X\to X$ its
normalisation. Recall that the non-normal locus $D\subset X$ and its preimage
$\bar D\subset \bar X$ are pure of codimension $1$, that is, curves. Since $X$ has
ordinary double points at the generic points of $D$ the map on normalisations
$\bar D^\nu\to D^\nu$ is the quotient by an involution $\tau$.
Koll\'ar's gluing principle says that $X$ can be uniquely reconstructed from
$(\bar X, \bar D, \tau\colon \bar D^\nu\to \bar D^\nu)$ via the following two
push-out squares:
\begin{equation}\label{diagr: pushout}
  \begin{tikzcd}
    \bar X \dar{\pi}\rar[hookleftarrow]{\bar\iota} & \bar D\dar{\pi} & \bar D^\nu \lar[swap]{\bar\nu}\dar{/\tau}
    \\
    X\rar[hookleftarrow]{\iota} &D &D^\nu\lar[swap]{\nu}
  \end{tikzcd}
\end{equation}

Applying this principle to non-normal Gorenstein stable surfaces, we deduce by \cite[Thm.~5.13]{Kollar2013} and \cite[Addendum in Sect.3.1.2]{FPR15a} that a triple $(\bar X, \bar D, \tau)$ corresponds to a Gorenstein stable surface with $K_X^2 = 1$ and $\chi(X) = 2$ if and only if the following four conditions are satisfied:
\begin{description}
\item[lc pair condition] $(\bar X, \bar D)$ is an lc pair, such that $K_{\bar X}+\bar D$ is an ample Cartier divisor.
\item[$K_X^2$-condition] $(K_{\bar X}+\bar D)^2=1$.
\item[Gorenstein-gluing condition] $\tau\colon \bar D^\nu\to \bar D^\nu$ is an involution that restricts to a fixed-point free involution on the preimages of the nodes of $\bar D$.
\item[$\chi$-condition]  The holomorphic Euler-characteristic of the non-normal locus $D$ is  \[\chi(D) = 2-\chi(\bar X)+\chi(\bar D).\]
\end{description}

In \cite{FPR15a} Gorenstein log canonical pairs   $(\bar{X}, \bar{D})$ with $(K_{\bar X}+\bar D)^2=1$ were classified:
\begin{enumerate}\label{list}
\item[($P$)] $\bar{X} = \IP^2$ and $\bar{D}$ is a nodal quartic. Here $p_a(\bar{D}) = 3$ and $K_{\bar{X}} + \bar{D} = \ko_{\IP^2}(1).$

\item[($dP$)] $\bar{X}$ is a (possibly singular) Del Pezzo surface of degree 1, namely $\bar{X}$ has at most canonical singularities, $-K_{\bar{X}}$ is ample and $K^2_{\bar{X}} = 1$. The curve $\bar{D}$ belong to the system $|-2K_{\bar{X}} |$, hence $K_{\bar{X}} + \bar{D} = -K_{\bar{X}}$ and $p_a(\bar{D}) = 2$.

\item[($E_{-}$)] Let $E$ be an elliptic curve and let $a\colon \tilde{X} \to E$
  be a geometrically ruled surface that contains an irreducible section $C_0$
  with $C_0^2 = -1$. Namely, $\tilde{X} = \IP(\ko_E + \ko_E(-x))$, where $x \in
  E$ is a point and $C_0$ is the only one curve on the system
  $|\ko_{\bar{X}}(1)|$. Set $F = a^{-1}(x):$ the normal surface $\bar{X}$ is
  obtained from $\tilde{X}$ by contracting $C_0$ to an elliptic Gorenstein
  singularity of degree 1 and $\bar{D}$ is the image of a curve $\bar{D}_0 \in
  |c(C_0 + F)|$ disjoint from $C_0$, so $p_a(\bar{D}) = 2$. The line bundle
  $K_{\bar{X}} + \bar{D}$ pulls back to $C_0 + F$ on $\tilde{X}$. 

  In fact, the anti-canonical divisor $-K_{\bar X}$ is ample, so that one can consider $\bar X$ as a singular del Pezzo surface, i.e., a degeneration of the case $(dP)$
\item[($E_{+}$)] $\bar{X} = S^2E$, where $E$ is an elliptic curve. Let $a\colon \bar{X} \to E$ be the Albanese map, which is induced by the addition map $E \times E \to E$, denote by $F$ the class of a fibre of $a$ and by $C_0$ the image in $\bar{X}$ of the curve $\{0\} \times E + E \times \{0\}$, where $0 \in E$ is the origin, so that $C_0F = C_0^2 = 1$. Then $\bar{D}$ is a divisor numerically equivalent to $3C_0 - F, p_a(\bar D) = 2$ and $K_{\bar{X}} + \bar{D}$ is numerically equivalent to $C_0$.
\end{enumerate}

\subsection{Case \caseP{}}
\label{sect: caseP}

It turns out that the classification of this case entails a detailed study of some plane quartics. For future reference we begin with an elementary lemma.
\begin{lem}\label{lem: reducible quartics}
  Let $\bar D$ be a nodal plane quartic.
  \begin{enumerate}
  \item If $\bar D$ has at most two nodes, then it is irreducible.
  \item If $\bar D$ has three nodes, then it is reducible if and only if the nodes are colinear if and only if it is the union of a smooth cubic and a general line.
  \item If $\bar D$ has   four nodes, then it is the union of two smooth conics or the union of a nodal cubic and a line.
  \item If $\bar D$ has five nodes, then it is the union of a smooth conic and two general lines.
  \item If $\bar D$ has at least six nodes, then it has exactly six nodes and is the union of four lines in general position.
  \end{enumerate}
\end{lem}
\begin{proof}
  All statements are elementary by B\'ezout. Let us only point out that three colinear nodes force the line through the nodes to be contained in $\bar D$. If there are four nodes, then in the pencil of conics through the nodes there is at least one that intersects $\bar D$ with multiplicity larger than $8$, and is thus contained in $\bar D$.
\end{proof}

Let us  set up the notation for the rest of the section: we consider a Gorenstein stable surface $X$ with $K_X^2 = 1$ and $\chi(\ko_X)=2$. With notation as in \eqref{diagr: pushout} the normalisation $\bar X=\IP^2$ and $\bar D$ is a nodal quartic.
Denote by
\begin{itemize}
\item $\mu_1$, the number of degenerate cusps in $X$.
\item $\rho$, the number of ramification points of the map ${\bar D}^{\nu} \to D^{\nu}$
\item $\bar \mu $, the number of nodes of $\bar D$.
\end{itemize}

\begin{rem}\label{rem: pushout equivalence relation}
  For further use, note that the points of $D^\nu$ correspond to equivalence classes of points on $\bar D^\nu$ with respect to the relation generated  by $x\sim y $ if $\bar \nu(x) =\bar\nu (y)$ or $\tau(x) = y$.
  By the classification of  Gorenstein semi-log-canonical singularities (see the
  proof of Lemma 3.5 in \cite{FPR15a}) nodes of $\bar D$ map to degenerate cusp
  singularities of $X$ and preimages of degenerate cusp singularities are nodes
  of $\bar D$. Thus, the number $\mu_1$ of degenerate cusps in $X$ equals the
  number of equivalence classes of preimages of nodes in $\bar D^\nu$ under the
  above relation.
\end{rem}

In our situation,  by the $\chi$-condition and  \cite[Lem.~3.5]{FPR15a}, we have the equality
\[-1 =\chi(D) = \frac 12\left( \chi({\bar D})-\bar \mu\right)+\frac\rho 4+\mu_1,\]
which gives
\begin{equation}\label{eq: nodes on barD}
  \bar \mu = \mu_1 = \rho  = 0 \quad\text{ or }\quad \bar\mu = \frac\rho 2+2\mu_1\geq 2.
\end{equation}

Since a plane quartic can have at most $6$ nodes, in total we get $2\leq \bar\mu\leq6$ unless $\bar D$ is smooth.

\begin{prop}\label{prop: case P exclude etale maps}
  Let $C$ be a nodal curve of arithmetic genus $3$ with a fixed-point free involution $\tau$, or equivalently, an \'etale map $\pi \colon C \to D$ to a nodal curve of arithmetic genus $2$.
  Then the image of the canonical map $C \overset{|K_C|}{\dashrightarrow} \IP^2$ is contained in a conic.

  In particular, $C$ is not a plane curve.
\end{prop}
\begin{proof}
  By assumption  there is a torsion divisor $L$ on $D$ such that $2L = 0 $ defining the double cover and the projection formula gives a splitting
  \[ H^0(K_{\bar D}) = H^0(K_D) \oplus H^0(K_D + L).\]
  Writing $H^0(K_D)  = \langle x,y \rangle$ and $H^0(K_D + L)= \langle z\rangle$, we see that $z^2\in H^0(2(K_D+L)) = H^0(2K_D) = \langle x^2, xy, y^2\rangle$, so there is a quadratic relation between the section defining the canonical map.

  To conclude that $C$ is not a plane curve note that a plane curve of arithmetic genus $3$ is a plane quartic and hence canonically embedded by the adjunction formula.
\end{proof}

\begin{cor}\label{cor: case P at least 3 nodes}
  In the above situation, the plane quartic $\bar D$ has at least three nodes.
\end{cor}
\begin{proof}
  Let $\bar D$ be a nodal quartic with $\bar \mu \leq 2$ nodes. By \eqref{eq: nodes on barD} the case $\bar\mu = 1$ need not be considered.

  We have to show, that an involution $\tau$ on $\bar D^\nu$ satisfying the Gorenstein gluing condition and yielding $\chi(D) = -1$ cannot exist. By Proposition \ref{prop: case P exclude etale maps} it is enough to show that such an involution would descend to a fixed-point-free involution on $\bar D$ itself.

  This is clear if $\bar D$ is smooth as in this case $\bar D = \bar D^\nu$.

  So we are left with $\bar \mu = 2$, in which case $\bar D^\nu$ is an elliptic curve with four marked points $P_1, P_2, Q_1, Q_2$  mapping to the nodes $P$ and $Q$ of $\bar D$.
  Assuming a suitable involution exists,  by \eqref{eq: nodes on barD} the involution $\tau$ cannot have fixed points on $\bar D^\nu$ and all points $P_i, Q_i$ map to a unique degenerate cusp.

  As explained in Remark \ref{rem: pushout equivalence relation} this means that, up to renaming we have $\tau(P_i) = Q_i$. This is exactly the condition for $\tau$ to descend to a fixed-point-free involution on $\bar D$. By Proposition \ref{prop: case P exclude etale maps} this is impossible for a plane curve.
\end{proof}

We  now show explicitly that examples with three nodes exist.
\begin{exam}[$\bar D$ irreducible with three nodes]\label{exam: case P 3 nodes}
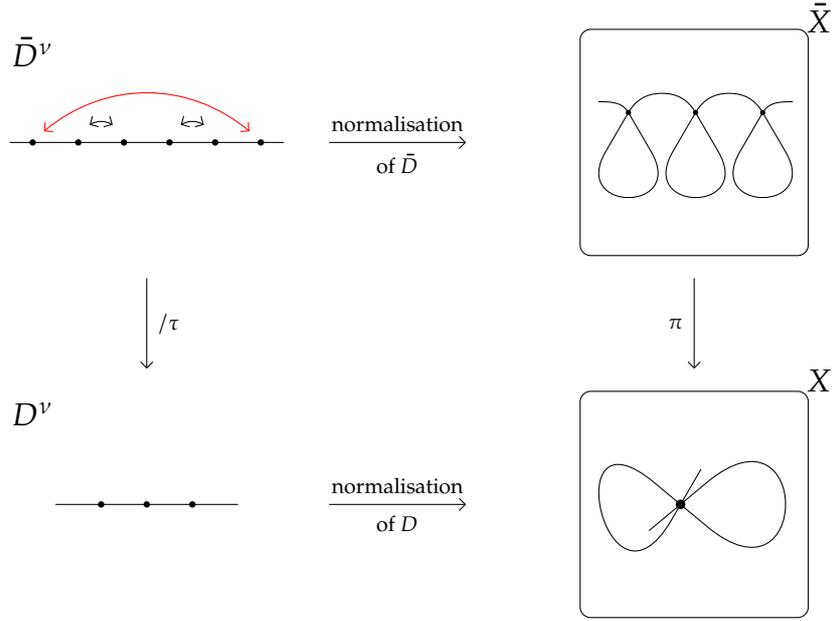
\begin{figure}
  \centering
  \begin{tikzpicture} [pfeil/.style = {->, every node/.style = { font
        =\scriptsize}}, scale = .6]
    \begin{scope}[xshift = 12cm, yshift = 8cm]
      \node at (2.75, 2.75) {$\bar{X}$};

      \draw[rounded corners] (-2.5, 2.5) rectangle (2.5, -2.5);

      \path (-2.1,0.9) coordinate(X);
      \draw (X) to[out=0, in=120,looseness=1]
      ++(-20:0.7) node[circle,fill=black, inner sep=0pt, minimum size=2pt](1){}
      -- +(30-90:1) to[out=30-90, in=270-30,looseness=4] ++(270-30:1) --
      ++(270-30:-1) to[out=60,looseness=1, in= 180] ++(30:0.85) to[out=0,
      in=120,looseness=1] ++(-30:0.85) node[circle,fill=black, inner sep=0pt,
      minimum size=2pt](2){} -- +(30-90:1) to[out=30-90, in=270-30,looseness=4]
      ++(270-30:1) -- ++(270-30:-1) to[out=60,looseness=1, in= 180] ++(30:0.85)
      to[out=0, in=120,looseness=1] ++(-30:0.85) node[circle,fill=black, inner
      sep=0pt, minimum size=2pt](3){} -- +(30-90:1) to[out=30-90,
      in=270-30,looseness=4] ++(270-30:1) -- ++(270-30:-1)
      to[out=60,looseness=1, in= 180] ++(20:0.7);

      \draw[pfeil] (0, -3) to node[left] {$\pi$} node[right] {} ++(0,-2);
      
    \end{scope}
    \begin{scope}[xshift = 0cm, yshift = 8cm]
      \node at (-2.5,2) {$\bar D^\nu$};

      \draw (-3,0) -- (3,0);
      \fill \foreach \x in {1,...,6}
      {(-3.5+\x,0) circle(2pt) node[above](\x) {}};
      \draw [<->] (2) to [out=30,in=150] (3);
      \draw [<->] (4) to [out=30,in=150] (5);
      \draw [<->,red] (1.east) to [out=40,in=140] (6.west);

      \draw[pfeil] (4,0) to node[below] { of $\bar D$} node[above] {
        normalisation} ++(3,0);

      \draw[pfeil] (0, -3) to node[left] {}
      node[right] { $/\tau$} ++(0,-2);
    \end{scope}
    \begin{scope}
      \node at (-2.5,2) {$ D^\nu$};

      \draw (-2, 0) -- (2,0);
      \fill \foreach \s in {1,...,3} { (-2 +\s,0) circle(2pt) node[above]{}};
   
      \draw[pfeil] (4,0) to node[below] { of $ D$} node[above] { normalisation} ++(3,0);

    \end{scope}
    \begin{scope}
      [xshift = 12cm, yshift =0cm, looseness=8]
      \node at (2.75, 2.75) {$X$};

      \draw[rounded corners] (-2.5, 2.5) rectangle
      (2.5, -2.5);

      \path (-0.3,0) coordinate(X);
      \draw {(X) -- +(40:0.75) to[out=40, in=-40] +(-40:0.75) -- (X)}
      {(X) -- +(180+40:0.9) }
      {(X) -- +(180-40:0.75) to [out=180-40, in=180+60] +(180+60:0.4) -- (X)}
      {(X) -- +(60:0.9) };
      \fill (X) circle (3pt);
    \end{scope}
  \end{tikzpicture}
  \caption{Case \caseP{}, $\bar D$ has three nodes}
  \label{fig: Case ($P_2$)}
\end{figure}
Assume $\bar D$ is an irreducible plane quartic with three nodes, which we may assume to be at $P_1=(1:0:0), P_2=(0:1:0)$ and $ P_3= (0:0:1)$. 
Its normalisation is $\bar D^\nu \isom \IP^1$ with six marked points mapping to the nodes of $\bar D$. 
Assume that there is an involution $\tau$ restricting to a fixed-point-free involution in the marked points. Then one can choose coordinates $(u:v)$ such that $\tau (u:v) = (av:u)$ and the six points are
\begin{align*}
 (0:1), && \tau(0:1) = (1:0),\\
 (1:1), && \tau(1:1) = (a:1),\\
 (b:1), && \tau (b:1) = (a:b),
\end{align*}
for some $a, b\in \IC\setminus\{0, 1\}$ with $a\neq b$. If we denote the homogeneous coordinates of the projective plane with $(x:y:z)$ then each coordinate function  vanishes at four of the six points exactly once, thus determining which point maps to which node. 

In order for the triple $(\IP^1, \bar D, \tau)$ to satisfy the $\chi$-condition, we infer from \eqref{eq: nodes on barD} that there is a uniqe degenerate cusp, that is, all six points are in the same class with respect to the equivalence relation explained in \ref{rem: pushout equivalence relation}.
Thus up to permuting and rescaling the coordinates the map $\bar\nu \colon \bar D^\nu \to \IP^2$ should be given  by 
\begin{align*}x &= uv(u-v)(u-bv),\\ y& = u(u-v)(u-av)(bu-av),\\ z& = v(u-av)(u-bv)(bu-av),
 \end{align*}
such that
\[\inverse{\bar\nu}(P_1) = \{(a:1), (a:b)\}, \inverse{\bar\nu}(P_2) = \{(1:0), (b:1)\},\text{ and }\inverse{\bar\nu}(P_3) = \{(0:1), (1:1)\}.\]
Therefore, a curve admitting such an involution on the normalisation exists and the equation of the image can be computed using Macaulay2 to be
\begin{align*}
f_{a,b}= (-a b^{3}+b^{4}+a^{2} b-a b^{2}) x^{2} y^{2} &+(a^{2} b^{3}-a^{3} b-a b^{3}-a^{3}+3 a^{2} b-a b^{2}) x^{2} y z\\
&+(a b^{2}-2 b^{3}-a^{2}+ab+b^{2}) x y^{2} z\\
&+(a^{4}-a^{3} b-a^{3}+a^{2} b) x^{2} z^{2}\\
&+(2 a^{2} b-a b^{2}-a^{2}-a b+b^{2}) x y z^{2}\\
&+(b^{2}-b) y^{2} z^{2}.
\end{align*}

  By construction, the triple $(\IP^2, \bar D, \tau)$ defines a Gorenstein stable surface with $K_X^2 = 1$ and $\chi(\ko_X) = 2$, depending on the parameters $a,b$.
\end{exam}

To sum up what we have done so far we state:
\begin{prop}
  Let $X$ be a Gorenstein stable surface with $K_X^2 = 1$ and $\chi(X) = 2$. If the normalisation $\bar X = \IP^2$ and $\bar D\subset \bar X$ is irreducible, then $X$ arises as in Example \ref{exam: case P 3 nodes}.
\end{prop}
\begin{proof}
  Combining Corollary \ref{cor: case P at least 3 nodes} and Lemma \ref{lem: reducible quartics} we see that neccesarily $\bar D$ is a plane quartic with three non-colinear nodes.

  As argued in Example \ref{exam: case P 3 nodes} there is, up to the choice of coordinates only one way to pick the involution $\tau$ satisfying the $\chi$-condition, in other words, $X$ arises as in Example \ref{exam: case P 3 nodes} as claimed. \end{proof}

\subsubsection{Case \caseP{} with $\bar D$ reducible}\label{sssect: case P reducible}
From now on let $\bar D$ be a reducible quartic. The possibilities are given in Lemma \ref{lem: reducible quartics}. We treat the cases separately, starting with the most degenerate ones.

The Gorenstein gluing condition imposes that the induced involution on the preimages of the nodes does not have fixed point, which we will tacitly use to exclude some involutions.

\begin{description}

\item[$\bar D = \text{four general lines}$]
  \begin{figure}
    \centering
    \begin{tikzpicture} [scale = .6,extended line/.style={shorten >=-#1,shorten <=-#1}]

      \begin{scope}[xshift = 12cm, yshift = 8cm]
        \draw[rounded corners] (-2.5, 2.5) rectangle (2.5, -2.5);
        \node at (2.75, 2.75) {$\IP^2$};

        \draw[C1col, name path=L1] (-1.5,-1.5) -- (0.75,1.5) node[right] {\footnotesize{$L_1$}};
        \draw[C1col, name path=L2] (1.5,-1.5)  -- (-0.75,1.5) node[left] {\footnotesize{$L_2$}};
        \draw[C2col, name path=L3] (-1.5,0) -- (2.25,-1.5) node[below,xshift=-0.1em] {\footnotesize{$L_3$}};
        \draw[C2col, name path=L4] (1.5,0)  -- (-2.25,-1.5) node[below,xshift=0.1em] {\footnotesize{$L_4$}};

        \foreach \s in {2,3,4}
        \path [name intersections={of=L1 and L\s,by=p1\s}];
        \foreach \s in {2,3,4}
        \fill[Qs] (p1\s) circle (2pt);
        \foreach \s in {3,4}
        \path [name intersections={of=L2 and L\s,by=p2\s}];
        \foreach \s in {3,4}
        \fill[Qs] (p2\s) circle (2pt);
        \path [name intersections={of=L3 and L4,by=p34}];
        \fill[Qs] (p34) circle (2pt);

        \draw[pfeil] (0, -3) to  node[left] {$\pi$} node[right] { normalisation} ++(0,-2);
      \end{scope}

      \begin{scope}[xshift = 0cm, yshift = 8cm]
        \node at (-3,2.5) {$\tilde{\IP}^2$};

        \draw[C1col] (-2,2) node[above]{\footnotesize{$L_1$}} -- (-2,-2) ;
        \draw[C1col] (2.5,2) node[above]{\footnotesize{$L_2$}} -- (2.5,-2);

        \fill[Qs]
        \foreach \s[count=\si] in {12,13,14}
        {(-2, 2.5-\si) circle (2pt) node[](p\s){} };

        \fill[Qs]
        \foreach \s[count=\si] in {21,23,24}
        {(2.5, 2.5-\si) circle (2pt) node[](p\s){}};

        \draw[C2col] (-1.5,0) -- (2,0);
        \draw[C2col] (-1.5,-1) -- (2,-1);

        \fill[Qs]
        \foreach \s[count=\si] in {31,34,32}
        {(-1.5 +\si, 0) circle (2pt) node[](p\s){}};

        \fill[Qs]
        \foreach \s[count=\si] in {41,43,42}
        {(-1.5 +\si, -1) circle (2pt) node[](p\s){}};

        \draw[C2col] (p34) node[above,xshift=-0.7em] {\footnotesize{$L_3$}};
        \draw[C2col] (p43) node[below,xshift=-0.7em] {\footnotesize{$L_4$}};

        \draw[dashed,extended line=1em] (p12) -- (p21);
        \draw[dashed,extended line=1em] (p13) -- (p31);
        \draw[dashed,extended line=1em] (p14) -- (p41);
        \draw[dashed,extended line=1em] (p23) -- (p32);
        \draw[dashed,extended line=1em] (p24) -- (p42);
        \draw[dashed,extended line=1em] (p34) -- (p43);

        \draw[pfeil] (4,0) to  node[below] { nodes } node[above] { blow up} ++(3,0);
        \draw[pfeil] (0, -3) to  node[left] {glue} node[right,text width=3em] { $L_{1} \sim L_{2} $ $L_{3} \sim L_{4} $} ++(0,-2);
      \end{scope}

      \begin{scope} 
        \node at (-2.5,2) {};

        \draw[C1col] (-2, 1)-- (2,1) node[C1col,right]{\footnotesize{$L_{12}$}};
        \draw[C2col] (-2, -1)-- (2,-1) node[C2col,right]{\footnotesize{$L_{34}$}};

        \fill [Qs]
        \foreach \s[count=\si] in {$P_{12}\sim P_{21}$,$P_{13}\sim P_{23}$,$P_{14}\sim P_{24}$}
        {
          {($(-3,1)+ \si*(1.5,0)$) circle (2pt)
            node[](p1\si){}
          }
        };

        \fill [Qs]
        \foreach \s[count=\si] in {$P_{31}\sim P_{41}$,$P_{32}\sim P_{42}$,$P_{34}\sim P_{43}$}
        {
          {($(-3,-1)+ \si*(1.5,0)$)  circle (2pt)
            node[](p2\si){}
          }
        };

        \draw[dashed,extended line=1.5em] (p12) -- (p22);
        \draw[dashed,extended line=1.5em] (p12) -- (p23);
        \draw[dashed,extended line=1.5em] (p13) -- (p22);
        \draw[dashed,extended line=1.5em] (p13) -- (p23);

        \draw[pfeil] (4,0) to  node[below] { } node[above] { contract E} ++(3,0);
      \end{scope}

      \begin{scope}
        [xshift = 12cm, yshift =0cm,  looseness=5]
        \node at (2.75, 2.75) {$X_{2,1}$};

        \draw[rounded corners] (-2.5, 2.5) rectangle (2.5, -2.5);
        \path (0,0) coordinate(X);
        \draw[C2col] {(X) -- +(80:0.75) to[out=80, in=10] +(10:0.75) -- (X)}
        {(X) -- +(180+80:2.0) }
        {(X) -- +(180+10:0.6) };

        \draw[C1col] {(X) -- +(180-30:0.75) to[out=180-30, in=180+40] +(180+40:0.75) -- (X)}
        {(X) -- +(-30:2.0) }
        {(X) -- +(40:0.6) };

        \fill[Qs] (X) +(-30:1.2) circle (2pt);
        \fill[Qs] (X) +(180+80:1.2) circle (2pt);
        \fill[Qs] (X) circle (3pt);
      \end{scope}
    \end{tikzpicture}
    \caption{Construction of $X_{2,1}$ and $X_{2,2}$}\label{fig:X21_construction}
  \end{figure}
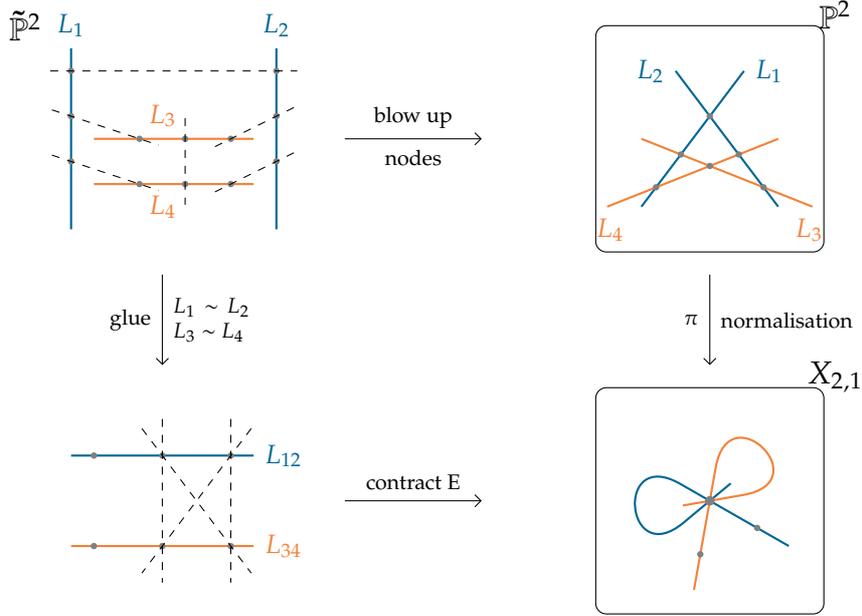

  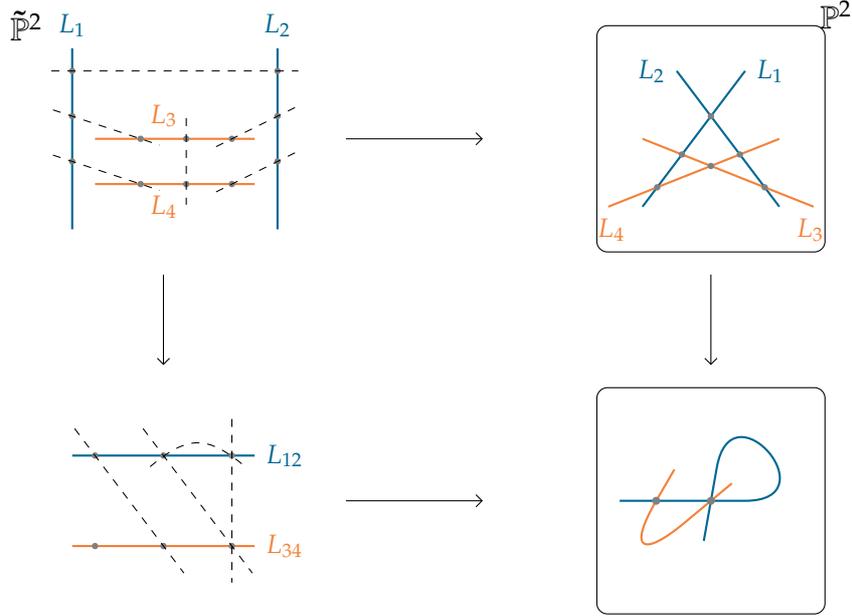
\begin{figure}
    \centering
    \begin{tikzpicture} [scale = .6,extended line/.style={shorten >=-#1,shorten <=-#1}]
      \begin{scope}[xshift = 12cm, yshift = 8cm]
        \draw[rounded corners] (-2.5, 2.5) rectangle (2.5, -2.5);
        \node at (2.75, 2.75) {$\IP^2$};

        \draw[C1col, name path=L1] (-1.5,-1.5) -- (0.75,1.5) node[right] {\footnotesize{$L_1$}};
        \draw[C1col, name path=L2] (1.5,-1.5)  -- (-0.75,1.5) node[left] {\footnotesize{$L_2$}};
        \draw[C2col, name path=L3] (-1.5,0) -- (2.25,-1.5) node[below,xshift=-0.1em] {\footnotesize{$L_3$}};
        \draw[C2col, name path=L4] (1.5,0)  -- (-2.25,-1.5) node[below,xshift=0.1em] {\footnotesize{$L_4$}};

        \foreach \s in {2,3,4}
        \path [name intersections={of=L1 and L\s,by=p1\s}];
        \foreach \s in {2,3,4}
        \fill[Qs] (p1\s) circle (2pt);
        \foreach \s in {3,4}
        \path [name intersections={of=L2 and L\s,by=p2\s}];
        \foreach \s in {3,4}
        \fill[Qs] (p2\s) circle (2pt);
        \path [name intersections={of=L3 and L4,by=p34}];
        \fill[Qs] (p34) circle (2pt);

        \draw[pfeil] (0, -3) to  node[left] {} node[right] { } ++(0,-2);
      \end{scope}

      \begin{scope}[xshift = 0cm, yshift = 8cm]
        \node at (-3,2.5) {$\tilde{\IP}^2$};

        \draw[C1col] (-2,2) node[above]{\footnotesize{$L_1$}} -- (-2,-2) ;
        \draw[C1col] (2.5,2) node[above]{\footnotesize{$L_2$}} -- (2.5,-2);

        \fill[Qs]
        \foreach \s[count=\si] in {12,13,14}
        {(-2, 2.5-\si) circle (2pt) node[](p\s){} };

        \fill[Qs]
        \foreach \s[count=\si] in {21,23,24}
        {(2.5, 2.5-\si) circle (2pt) node[](p\s){}};

        \draw[C2col] (-1.5,0) -- (2,0);
        \draw[C2col] (-1.5,-1) -- (2,-1);

        \fill[Qs]
        \foreach \s[count=\si] in {31,34,32}
        {(-1.5 +\si, 0) circle (2pt) node[](p\s){}};

        \fill[Qs]
        \foreach \s[count=\si] in {41,43,42}
        {(-1.5 +\si, -1) circle (2pt) node[](p\s){}};

        \draw[C2col] (p34) node[above,xshift=-0.7em] {\footnotesize{$L_3$}};
        \draw[C2col] (p43) node[below,xshift=-0.7em] {\footnotesize{$L_4$}};

        \draw[dashed,extended line=1em] (p12) -- (p21);
        \draw[dashed,extended line=1em] (p13) -- (p31);
        \draw[dashed,extended line=1em] (p14) -- (p41);
        \draw[dashed,extended line=1em] (p23) -- (p32);
        \draw[dashed,extended line=1em] (p24) -- (p42);
        \draw[dashed,extended line=1em] (p34) -- (p43);

        \draw[pfeil] (4,0) to  node[below] {  } node[above] { } ++(3,0);
        \draw[pfeil] (0, -3) to  node[left] {} node[right,text width=3em] { } ++(0,-2);
      \end{scope}

      \begin{scope}
        \node at (-2.5,2) {};

        \draw[C1col] (-2, 1)-- (2,1) node[C1col,right]{\footnotesize{$L_{12}$}};
        \draw[C2col] (-2, -1)-- (2,-1) node[C2col,right]{\footnotesize{$L_{34}$}};

        \fill [Qs]
        \foreach \s[count=\si] in {$P_{12}\sim P_{21}$,$P_{13}\sim P_{23}$,$P_{14}\sim P_{24}$}
        {
          {($(-3,1)+ \si*(1.5,0)$) circle (2pt)
            node[](p1\si){}
          }
        };

        \fill [Qs]
        \foreach \s[count=\si] in {$P_{31}\sim P_{41}$,$P_{32}\sim P_{42}$,$P_{34}\sim P_{43}$}
        {
          {($(-3,-1)+ \si*(1.5,0)$)  circle (2pt)
            node[](p2\si){}
          }
        };

        \draw[dashed,extended line=1.5em] (p11) -- (p22);
        \draw[dashed,extended line=1.5em] (p12) -- (p23);
        \draw[dashed,extended line=1.5em] (p13) -- (p23);
        \draw[dashed,extended line=1.0em] (p12) to[out=40, in =180-40] (p13);

        \draw[pfeil] (4,0) to  node[below] { } node[above] { } ++(3,0);
      \end{scope}

      \begin{scope}
        [xshift = 12cm, yshift =0cm,  looseness=5]
        \draw[rounded corners] (-2.5, 2.5) rectangle (2.5, -2.5);

        \path (0,0) coordinate(X);
        \draw[C1col] {(X) -- +(80:0.75) to[out=80, in=0] +(0:0.75) -- (X)}
        {(X) -- +(180+80:0.9) }
        {(X) -- +(180:2.0) };

        \draw[C2col,looseness=4]
        {(X) -- +(180+40:0.75) to[out=180+40, in=180+60] ++(180:1.2) node[](p1){} -- +(60:0.8)}
        {(X) -- +(40:0.6)};

        \fill[Qs] (p1) circle(2.5pt);
        \fill[Qs] (X) circle (2.5pt);
      \end{scope}

    \end{tikzpicture}
    \caption{Construction of $X_{2,3}$}\label{fig:X23_construction}
  \end{figure}
  This case has been classified in \cite[Sect.~4.2]{FPR15inv} and we follow their notation. Let $\bar D = L_1 +L_2 + L_3 + L_4$ be the union of four general lines. We denote $P_{(ij)}$ the intersection point of $L_i$ and $L_j$. The normalisation of $\bar D$ is $\bar D^{\nu} = \sqcup L_i$ and we denote by $P_{ij}$ the point of $L_i \subset \bar D^{\nu}$ that maps to $P_{(ij)}$.

  Since every component of $\bar D^{\nu}$ contains three such points, $\tau$ cannot preserve any of the $L_i$, so we may assume that it maps $L_1$ to $L_2$ and $L_3$ to  $L_4$. Then $\tau$ is uniquely determined by two bijections
  \[\varphi_{12}\colon \{P_{12}, P_{13}, P_{14} \} \to \{P_{21}, P_{23}, P_{24} \}.
  \]
  \[\varphi_{34}\colon \{P_{31}, P_{32}, P_{34} \} \to \{P_{41}, P_{42}, P_{43} \}.
  \]
  By loc.\ cit.\ $X$ is isomorphic to one (and only one) of the surfaces $X_{2,1}, X_{2,2}$ and $X_{2,3}$ corresponding to the involutions listed in Table \ref{tab: four general lines}.
  \begin{table}[h]\caption{Surfaces from four lines in the plane from \cite{FPR15a}}\label{tab: four general lines}
    \begin{center}
      \begin{tabular}{ c c c c }
        \toprule
        Surface  & $\varphi_{12}$ and	 $\varphi_{34}$ & Degenerate cusps  \\
        \midrule
        $X_{2,1}$
                 &  $\varphi_{12} =\begin{pmatrix} P_{12} & P_{13} & P_{14}\\ P_{21} & P_{24} & P_{23}\end{pmatrix}$
                 & $\{P_{(12)}\}$, $\{P_{(34)}\}$,
        \\
                 &  $\varphi_{34} =\begin{pmatrix} P_{31} & P_{32} & P_{34}\\ P_{41} & P_{42} & P_{43}\end{pmatrix}$
                 &$\{P_{(13)}, P_{(14)}, P_{(23)}, P_{(24)}\}$
        \\
        \midrule
        $ X_{2,2}$
                 &  $\varphi_{12} =\begin{pmatrix} P_{12} & P_{13} & P_{14}\\ P_{21} & P_{23} & P_{24}\end{pmatrix}$
                 & $\{ P_{(12)}\}$, $\{P_{(34)}\}$,
        \\
                 &  $\varphi_{34} =\begin{pmatrix} P_{31} & P_{32} & P_{34}\\ P_{42} & P_{41} & P_{43}\end{pmatrix}$
                 &$\{P_{(13)}, P_{(14)}, P_{(23)}, P_{(24)}\}$
        \\
        \midrule
        $ X_{2,3}$
                 & $\varphi_{12} =\begin{pmatrix} P_{12} & P_{13} & P_{14}\\ P_{23} & P_{24} & P_{21}\end{pmatrix}$
                 & $\{P_{(12)}, P_{(23)}, P_{(14)}\}$,
        \\
                 &  $\varphi_{34} =\begin{pmatrix} P_{31} & P_{32} & P_{34}\\ P_{42} & P_{41} & P_{43}\end{pmatrix}$
                 &$\{P_{(13)}, P_{(24)}\}$, $\{P_{(34)}\}$
        \\
        \bottomrule
      \end{tabular}

    \end{center}
  \end{table}

\item[$\bar D = \text{a conic and two lines}$] The gluing involution $\tau$ has to preserve the conic and exchange the two lines, because they contain a different number of preimages of nodes. Thus we have five nodes in total, two fixed points of $\tau$ on the conic and thus by  \eqref{eq: nodes on barD}  exactly two degenerate cusps.

  We will now determine all possible involutions $\tau$, using the notation
  from Figure
 \ref{fig: conic and two lines A'}.

  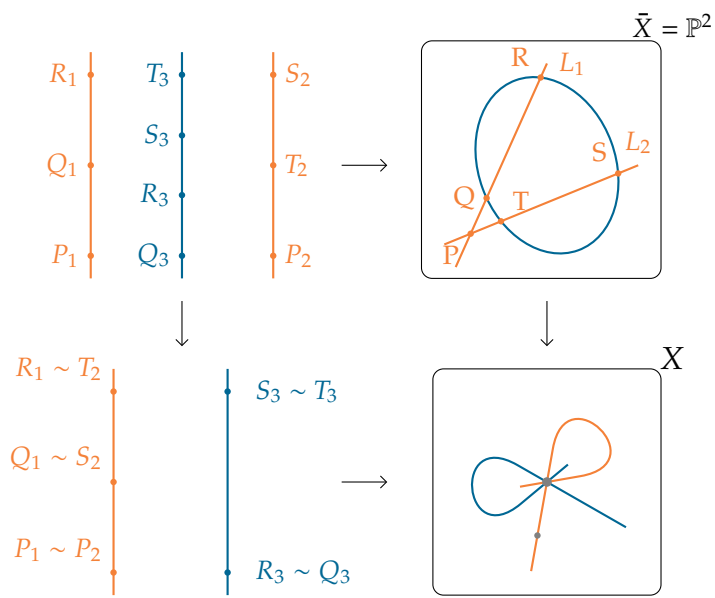
\begin{figure}
    \centering
    \begin{tikzpicture} [scale = .6,extended line/.style={shorten >=-#1,shorten <=-#1}]

      \begin{scope}[xshift = 0cm, yshift = 7cm]
        \draw[C2col] (-2,2.5) -- (-2,-2.5);
        \fill[C2col] (-2,2.0) circle (2pt) node[left]{\footnotesize{$R_1$}}
        (-2,0.0) circle (2pt) node[left]{\footnotesize{$Q_1$}}
        (-2,-2.0) circle (2pt) node[left]{\footnotesize{$P_1$}};

        \draw[C2col] (2,2.5) -- (2,-2.5);
        \fill[C2col] (2,2.0) circle (2pt) node[right]{\footnotesize{$S_2$}}
        (2,0.0) circle (2pt) node[right]{\footnotesize{$T_2$}}
        (2,-2.0) circle (2pt) node[right]{\footnotesize{$P_2$}};

        \draw[C1col] (0,2.5) -- (0,-2.5);
        \fill[C1col] (0,2.0) circle (2pt) node[left]{\footnotesize{$T_{3}$}}
        (0,0.66) circle (2pt) node[left]{\footnotesize{$S_{3}$}}
        (0,-0.66) circle (2pt) node[left]{\footnotesize{$R_{3}$}}
        (0,-2.0) circle (2pt) node[left]{\footnotesize{$Q_{3}$}};

        \draw[pfeil] (3.5, 0) to  ++(1,0);
        \draw[pfeil] (0, -3.0) to  ++(0,-1);
      \end{scope}

      \begin{scope}[xshift = 8.cm, yshift = 7cm]
        \node at (2.75, 3.1) {\footnotesize{$\bar{X}=\IP^2$}};
        \draw[rounded corners] (-2.75, 2.75) rectangle (2.5, -2.5);

        \draw[C1col,rotate=110, name path=eclipse] (0,0) ellipse (2cm and 1.5cm);
        \draw[C2col, name path=L1] (-2,-2.25) -- (0.0,2.25) node[right]{\footnotesize{$L_1$}};
        \draw[C2col, name path=L2] (-2.25,-1.75) -- (2,0) node[above]{\footnotesize{$L_2$}};
        \path [name intersections={of=L1 and eclipse}];
        \coordinate[label={[C2col]left:\footnotesize{Q}}] (p11)  at (intersection-1);
        \coordinate[label={[C2col]above left:\footnotesize{R}}] (p12) at (intersection-2);
        \path [name intersections={of=L2 and eclipse}];
        \coordinate[label={[C2col]above right:\footnotesize{T}}] (p21)  at (intersection-1);
        \coordinate[label={[C2col]above left:\footnotesize{S}}] (p22) at (intersection-2);

        \fill[C2col] (p11) circle (2pt)
        (p12) circle (2pt)
        (p21) circle (2pt)
        (p22) circle (2pt);

        \path [name intersections={of=L1 and L2}];
        \coordinate[label={[C2col]below left:\footnotesize{P}}] (px)  at (intersection-1);
        \fill[C2col] (px) circle (2pt);

        \draw[pfeil] (0, -3) to  ++(0,-1);
      \end{scope}

      \begin{scope}[xshift = 1cm, yshift = 0cm] 
        \draw[C2col] (-2.5,2.5) -- (-2.5,-2.5);
        \fill[C2col] (-2.5,2.0) coordinate (l1) circle (2pt) node[above left]{\footnotesize{$R_1\sim T_2$}}
        (-2.5,0.0) coordinate (l2) circle (2pt) node[above left]{\footnotesize{$Q_1\sim S_2$}}
        (-2.5,-2.0)coordinate (l3) circle (2pt) node[above left]{\footnotesize{$P_1\sim P_2$}};

        \draw[C1col] (0,2.5) -- (0,-2.5);
        \fill[C1col] (0,2.0)coordinate (r1) circle (2pt) node[right,xshift=0.5em]{\footnotesize{$S_{3}\sim T_{3}$}}
        (0,-2.0)coordinate (r4) circle (2pt) node[right,xshift=0.5em]{\footnotesize{$R_{3}\sim Q_{3}$}};
        \draw[pfeil] (2.5,0) to ++(1,0);
      \end{scope}

      \begin{scope}
        [xshift = 8cm, yshift =0cm, looseness=5 ]
        \node at (2.75, 2.75) {$X$};
        \draw[rounded corners] (-2.5, 2.5) rectangle (2.5, -2.5);
        \path (0,0) coordinate(X);
        \draw[C2col] {(X) -- +(80:0.75) to[out=80, in=10] +(10:0.75) -- (X)}
        {(X) -- +(180+80:2.0) }
        {(X) -- +(180+10:0.6) };

        \draw[C1col] {(X) -- +(180-30:0.75) to[out=180-30, in=180+40] +(180+40:0.75) -- (X)}
        {(X) -- +(-30:2.0) }
        {(X) -- +(40:0.6) };

        \fill[Qs] (X) +(180+80:1.2) circle (2pt);
        \fill[Qs] (X) circle (3pt);
      \end{scope}
    \end{tikzpicture}
    \caption{conic and two lines, Case $A'$}\label{fig: conic and two lines A'}
  \end{figure}

 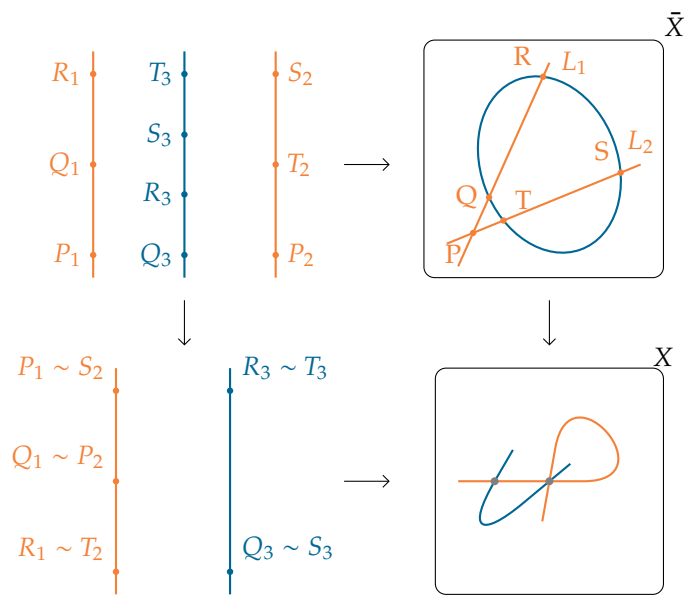
\begin{figure}
   \centering
   \begin{tikzpicture} [scale = .6,extended line/.style={shorten >=-#1,shorten <=-#1}]

  \begin{scope}[xshift = 0cm, yshift = 7cm]
    \draw[C2col] (-2,2.5) -- (-2,-2.5);
    \fill[C2col] (-2,2.0) circle (2pt) node[left]{\footnotesize{$R_1$}}
    (-2,0.0) circle (2pt) node[left]{\footnotesize{$Q_1$}}
    (-2,-2.0) circle (2pt) node[left]{\footnotesize{$P_1$}};

    \draw[C2col] (2,2.5) -- (2,-2.5);
    \fill[C2col] (2,2.0) circle (2pt) node[right]{\footnotesize{$S_2$}}
    (2,0.0) circle (2pt) node[right]{\footnotesize{$T_2$}}
    (2,-2.0) circle (2pt) node[right]{\footnotesize{$P_2$}};

    \draw[C1col] (0,2.5) -- (0,-2.5);
    \fill[C1col] (0,2.0) circle (2pt) node[left]{\footnotesize{$T_{3}$}}
    (0,0.66) circle (2pt) node[left]{\footnotesize{$S_{3}$}}
    (0,-0.66) circle (2pt) node[left]{\footnotesize{$R_{3}$}}
    (0,-2.0) circle (2pt) node[left]{\footnotesize{$Q_{3}$}};

    \draw[pfeil] (3.5, 0) to  ++(1,0);
    \draw[pfeil] (0, -3.0) to  ++(0,-1);
  \end{scope}

  \begin{scope}[xshift = 8.cm, yshift = 7cm]
    \node at (2.75, 3.1) {\footnotesize{$\bar{X}$}};
    \draw[rounded corners] (-2.75, 2.75) rectangle (2.5, -2.5);

    \draw[C1col,rotate=110, name path=eclipse] (0,0) ellipse (2cm and 1.5cm);
    \draw[C2col, name path=L1] (-2,-2.25) -- (0.0,2.25) node[right]{\footnotesize{$L_1$}};
    \draw[C2col, name path=L2] (-2.25,-1.75) -- (2,0) node[above]{\footnotesize{$L_2$}};
    \path [name intersections={of=L1 and eclipse}];
    \coordinate[label={[C2col]left:\footnotesize{Q}}] (p11)  at (intersection-1);
    \coordinate[label={[C2col]above left:\footnotesize{R}}] (p12) at (intersection-2);
    \path [name intersections={of=L2 and eclipse}];
    \coordinate[label={[C2col]above right:\footnotesize{T}}] (p21)  at (intersection-1);
    \coordinate[label={[C2col]above left:\footnotesize{S}}] (p22) at (intersection-2);

    \fill[C2col] (p11) circle (2pt)
    (p12) circle (2pt)
    (p21) circle (2pt)
    (p22) circle (2pt);

    \path [name intersections={of=L1 and L2}];
    \coordinate[label={[C2col]below left:\footnotesize{P}}] (px)  at (intersection-1);
    \fill[C2col] (px) circle (2pt);

    \draw[pfeil] (0, -3) to  ++(0,-1);
  \end{scope}

  \begin{scope}[xshift = 1cm, yshift = 0cm] 
    \draw[C2col] (-2.5,2.5) -- (-2.5,-2.5);
    \fill[C2col] (-2.5,2.0) coordinate (l1) circle (2pt) node[above left]{\footnotesize{$P_1\sim S_2$}}
    (-2.5,0.0) coordinate (l2) circle (2pt) node[above left]{\footnotesize{$Q_1\sim P_2$}}
    (-2.5,-2.0)coordinate (l3) circle (2pt) node[above left]{\footnotesize{$R_1\sim T_2$}};

    \draw[C1col] (0,2.5) -- (0,-2.5);
    \fill[C1col] (0,2.0)coordinate (r1) circle (2pt) node[above right]{\footnotesize{$R_3\sim T_{3}$}}
    (0,-2.0)coordinate (r4) circle (2pt) node[above right]{\footnotesize{$Q_{3}\sim S_3$}};

%
    \draw[pfeil] (2.5,0) to ++(1,0);
  \end{scope}

  \begin{scope}
    [xshift = 8cm, yshift =0cm, looseness=5 ]
    \node at (2.5, 2.75) {\footnotesize{$X$}};
    \draw[rounded corners] (-2.5, 2.5) rectangle (2.5, -2.5);
    \path (0,0) coordinate(X);
    \draw[C2col] {(X) -- +(80:0.75) to[out=80, in=0] +(0:0.75) -- (X)}
    {(X) -- +(180+80:0.9) }
    {(X) -- +(180:2.0) };

    \draw[C1col,looseness=4]
    {(X) -- +(180+40:0.75) to[out=180+40, in=180+60] ++(180:1.2) node[](p1){} -- +(60:0.8)}
    {(X) -- +(40:0.6)};

    \fill[Qs] (p1) circle(2.5pt);
    \fill[Qs] (X) circle (2.5pt);

  \end{scope}
\end{tikzpicture}
   \caption{conic and two lines,  Case $B$}\label{fig: conic and two lines B}
   \end{figure}

  \begin{description}
  \item[Case $A'$ and Case $A''$:] Assume that $\tau(P_1) = P_2$, that is, the
    preimage of one degenerate cusp consists solely of the intersection point
    $P$ of the two lines. Then all other preimages of nodes have to be
    equivalent under the equivalence relation of Remark \ref{rem: pushout
      equivalence relation} and there are up to renaming two possibilties
    $\tau'$ and $\tau''$: either $\tau|_C$ preserves the intersection $L_i\cap
    C$, that is, 
  \[ \tau'(Q_{3}) = R_{3} \text { and } \tau'(S_{3}) = T_{3},\]
    or it does not, that is,
    \[ \tau''(Q_{3}) = S_{3} \text { and } \tau''(R_{3}) = T_{3}.\]
In order to ensure the correct number of degenerate cusps we need to have
\[ 
\tau'\hat= \begin{pmatrix} P_1 & Q_1 & R_1\\ P_2 &  S_2 & T_2 \end{pmatrix} \text{ and } \tau''\hat= \begin{pmatrix} P_1 & Q_1 & R_1\\ P_2 & T _2 & S_2 \end{pmatrix}.
\]
    These constructions depends on one parameter, namely the choice of the conic. If we degenerate the conic to a pair of lines, we arrive at $X_{2,1}$ or $X_{2,2}$ from Table \ref{tab: four general lines}.

  \item[Case $B$:]
    If $\tau(P_1) \neq P_2$, then we can choose the involution on the lines to be given as
    \[  \begin{pmatrix} P_1 & Q_1 & R_1\\ S_2 &  P_2 & T_2 \end{pmatrix}.\]
    and $\tau|_C(Q_3) = S_3$ and thus $\tau|_C(R_3) = T_3$ such that the preimages of the two  degenerate cusps are $\{ P, Q, S\}$ and $\{ R, T\}$.

    This construction depends on one parameter, namely the choice of the conic. If we degenerate the conic to a pair of lines, we arrive at $X_{2,3}$ from Table \ref{tab: four general lines}.
  \end{description}

\item[$\bar D = \text{ two irreducible conics}$]
  By \eqref{eq: nodes on barD} there are two possibilities:
  \begin{description}

  \item[Case A: $\rho = 4$ and $\mu_1 = 1$]
    \begin{figure}
      \centering
      \begin{tikzpicture}
        [pfeil/.style = {->, every node/.style = { font = \scriptsize}},
        scale = .6]
        \begin{scope}[xshift = 12cm, yshift = 8cm]
          \node at (2.75, 2.75) {$\IP^2$};

          \draw[rounded corners] (-2.5, 2.5) rectangle (2.5, -2.5);
          \draw[C1col, name path = C1] (0,0) ellipse (1 and 1.5 ) ++(0,1.5) node[above] {};
          \draw[C2col, name path = C2] (0,0) ellipse (1.5 and 1 ) ++(1.5,0) node[right] {};;

          \fill [name intersections={of= C1 and C2, name=Q}]
          [Qs]
          \foreach \s in {1,...,4}
          {(Q-\s) circle (2pt) ++(-45+\s*90:.4) node  { $Q_\s$}};

          \draw[pfeil] (0, -3) to  node[left] {$\pi$} node[right] { normalisation} ++(0,-2);
        \end{scope}

        \begin{scope}[xshift = 0cm, yshift = 8cm,every node/.style = { font = \footnotesize}]
          \node at (-2.5,2) {$\tilde{\IP}^2$};
          \draw[C1col] (-2, 1)-- (2,1) node[right]{$\bar{C_1}$};
          \draw[C2col] (-2, -1)-- (2,-1) node[right]{$\bar{C_2}$};

          \fill [Qs]
          \foreach \s in {1,...,4}
          {(-2.5+\s,1) circle (2pt) node[C1col, above](p\s)  { $P_\s$}};

          \fill [Qs]
          \foreach \s[count=\si] in {1,2,3,4}
          {(-2.5+\si,-1) circle (2pt) node[C2col, below](q\s)  { $Q_\s$}};

          \draw[pfeil,<->] ([yshift=-0.4em]p1.south) to[out=-40, in =180+40]
          ([yshift=-0.4em]p2.south);
          \draw[pfeil,<->] ([yshift=-0.4em]p3.south) to[out=-40, in =180+40]
          ([yshift=-0.4em]p4.south);

          \draw[pfeil,<->] ([yshift=0.4em]q1.north) to[out=40, in =180-40]
          ([yshift=0.4em]q3.north);
          \draw[pfeil,<->] ([yshift=0.4em]q2.north) to[out=40, in =180-40]
          ([yshift=0.4em]q4.north);

          \draw[pfeil] (4,0) to  node[below] { of $\bar D$} node[above] { normalisation} ++(3,0);
          \draw[pfeil] (0, -3) to  node[left] {} node[right] { $/\tau$} ++(0,-2);
        \end{scope}

        \begin{scope} 
          \node at (-2.5,2) {};
          \draw[C1col] (-2, 1)-- (2,1) node[right]{};
          \draw[C2col] (-2, -1)-- (2,-1) node[right]{};
          \fill [Qs]
          \foreach \s[count=\si] in {$P_1 \sim P_2$,$P_3 \sim P_4$}
          {($(-3,1)+\si*(2,0)$) circle (2pt) node[C1col, above](p\si){\tiny{\s}}};

          \fill [Qs]
          \foreach \s[count=\si] in {$Q_1 \sim Q_3$,$Q_2 \sim Q_4$}
          {($(-3,-1)+\si*(2,0)$) circle (2pt) node[C2col, above](q\si){\tiny{\s}}};

          \draw[pfeil] (4,0) to  node[below] { } node[above] { normalisation} ++(3,0);
        \end{scope}

        \begin{scope}
          [xshift = 12cm, yshift =0cm,  looseness=5]
          \node at (2.75, 2.75) {$X$};
          \path (0,0) coordinate(X);
          \draw[C2col] {(X) -- +(80:0.75) to[out=80, in=10] +(10:0.75) -- (X)}
          {(X) -- +(180+80:2.0) }
          {(X) -- +(180+10:0.6) };

          \draw[C1col] {(X) -- +(180-30:0.75) to[out=180-30, in=180+40] +(180+40:0.75) -- (X)}
          {(X) -- +(-30:2.0) }
          {(X) -- +(40:0.6) };

          \fill[Qs] (X) circle (3pt);

        \end{scope}
      \end{tikzpicture}
      \caption{Case A: $\rho = 4$ and $\mu_1=1$}
      \label{fig: Case A}
    \end{figure}
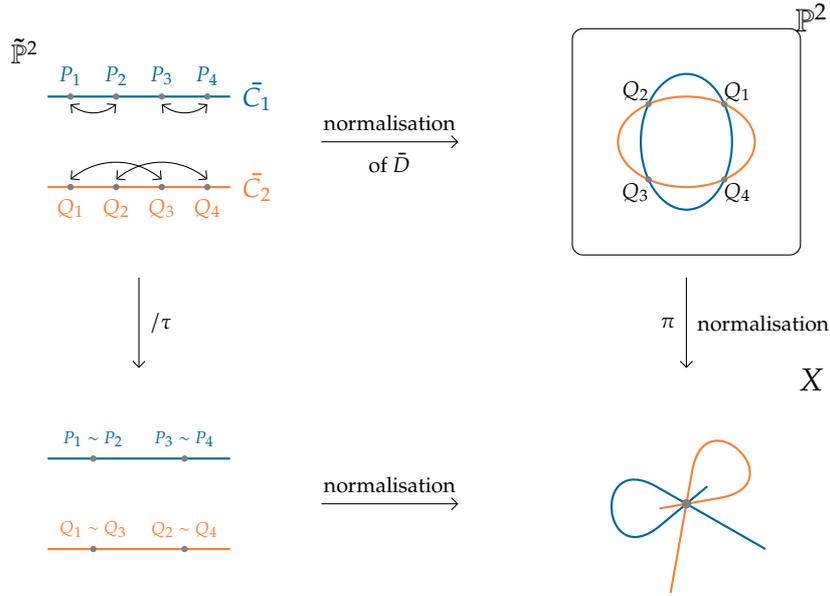
    In this case the involution preserves the two components of $\bar D^\nu$ and, in order for there to be only one degenerate cusp, one can name the points such that the involution is as Figure \ref{fig: Case A}.

    It is an elementary fact, that given a projective line with four marked points, there is always an involution exchanging two pairs of points (compare Example \ref{exam: case P 3 nodes}), so the desired involutions exist on any two smooth  conics in the pencil.

    The construction depends on the choice of the two conics in a pencil, that is, two parameters. If we let one of the conics degenerate to  a pair of lines then we can arrive at the possibilities considered in Case $A'$ and Case $A''$ above. Making both conics reducible gives the surfaces $X_{2,1}$ and $X_{2,2}$ from Table \ref{tab: four general lines}.

  \item[Case C: $\rho = 0 $ and $\mu_1 = 2$]
    \begin{figure}
      \centering
      \begin{tikzpicture}
        [pfeil/.style = {->, every node/.style = { font = \scriptsize}},
        scale = .6]

        \begin{scope}[xshift = 12cm, yshift = 8cm]
          \node at (2.75, 2.75) {$\IP^2 $};

          \draw[rounded corners] (-2.5, 2.5) rectangle (2.5, -2.5);
          \draw[C1col, name path = C1] (0,0) ellipse (1 and 1.5 ) ++(0,1.5) node[above] {$\sigma(C)$};
          \draw[C2col, name path = C2] (0,0) ellipse (1.5 and 1 ) ++(1.5,0) node[right] {$C$};;

          \fill [name intersections={of= C1 and C2, name=Q}]
          [Qs]
          \foreach \s in {1,...,4}
          {(Q-\s) circle (2pt) ++(-45+\s*90:.4) node  { $Q_\s$}};

          \draw[pfeil] (0, -3) to  node[left] {$\pi$} node[right] { normalisation} ++(0,-2);
        \end{scope}

        \begin{scope}[xshift = 0cm, yshift = 8cm]
          \node at (-2.5,2) {$\bar D^\nu$};
          \draw[ pfeil,<->, Red] (0,.5) to node[right]{$\tau$} (0,-.5);
          \draw[C1col] (-2, -1)-- (2,-1) node[right]{$\sigma(C)$};
          \draw[C2col] (-2, 1)-- (2,1) node[right]{$C$};

          \fill [Qs]
          \foreach \s in {1,...,4}
          {(-2.5+\s,1) circle (2pt) node[C2col, above]  { $Q_\s$}};

          \fill [Qs]
          \foreach \s[count=\si] in {1,3,4,1}
          {(-2.5+\si,-1) circle (2pt) node[C1col, below]  { $Q_\s$}}
          (1.5,-1) circle (2pt) node[C1col, below]  { $Q_1$};

          \draw[pfeil] (4,0) to  node[below] { of $\bar D$} node[above] { normalisation} ++(3,0);
          \draw[pfeil] (0, -3) to  node[left] {} node[right] { $/\tau$} ++(0,-2);
        \end{scope}

        \begin{scope} 
          \node at (-2.5,2) {$ D^\nu$};
          \draw[C12col] (-2, 0)-- (2,0) node[right]{
          };
          \fill [Qs]
          \foreach \s in {1,...,4}
          {(-2.5+\s,0) circle (2pt) node[C12col, above]  {\footnotesize $Q_\s$}};
          \draw[pfeil] (4,0) to  node[below] { of $ D$} node[above] { normalisation} ++(3,0);
        \end{scope}

        \begin{scope}
          [xshift = 12cm, yshift =0cm,  looseness=5]
          \node at (2.75, 2.75) {$X$};
          \draw[rounded corners] (-2.5, 2.5) rectangle (2.5, -2.5);
          \path (0,0) coordinate(P) ;

          \path (0,0) coordinate(X);
          \draw[C12col] {(X) -- +(40:0.75) to[out=40, in=-40] +(-40:0.75) -- (X)}
          {(X) -- +(180+40:0.9) }
          {(X) -- +(180-40:0.75) to [out=180-40, in=180+60] +(180+60:0.4) -- (X)}
          {(X) -- +(60:0.9) };
          \fill[Qs] (X) circle (3pt);

        \end{scope}
      \end{tikzpicture}
      \caption{Case C:  $\bar D$ = two irreducible conics, $\rho = 0, \mu_1 = 2$}\label{figold: general surface in P1}
    \end{figure}
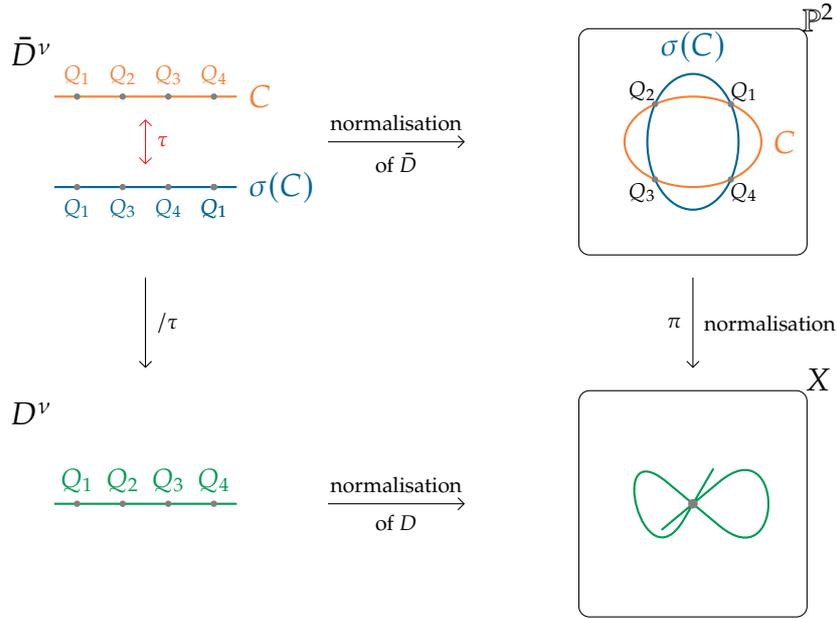
    Since an involution on $\IP^1$ has fixed points, $\tau $ exchanges the components
    $\bar D = C' + C''$, that is, $\tau$ is induced by an abstract isomorphism $\phi\colon C' \to C''$.

    Let us denote the four intersection points of the two conics $C', C''$ with $Q_1, \dots, Q_4$. We add the primes if we consider the points on the individual conics.
    By \eqref{eq: nodes on barD} we have two  degenerate cusps, say $R_1$ and $R_2$. Up to reindexing there are again two cases:
    \begin{description}
    \item[$\inverse\pi(R_1) = \{Q_1\}$, $\inverse\pi(R_2) = \{Q_2, Q_3, Q_4\}$:]
      The abstract isomorphism $\phi$ has the property, up to reindexing again,
      \[ \phi(Q_1') = Q_1'', \phi(Q_2') = Q_3'', \phi(Q_3')=Q_4'', \phi(Q_4') = Q_2''.\]

      Now consider the unique automorphism $\sigma$ of $\IP^2$ that acts on the $Q_i$, considered as point in the plane,  in the same way as $\phi$ and let $\sigma C'$ be the image of $C'$ under $\sigma$.
      The composition $\sigma\circ \inverse\phi\colon C''\to  \sigma C'$ is an abstract isomorphism of two plane conics fixing four points in the plane. By \cite[Es.~4.24]{FFP16}  it is actually induced by the identity on $\IP^2$, thus $C'' = \sigma C'$ and $\phi=\sigma|_C'$.

      Since $C''$ is determined by $C'$, this construction depends on the one parameter. If we let $C'=L_1+L_3$ the union of two lines, then the above construction still provides us with a suitable involution and it is straightforward to check, that it gives the case $X_{2,3}$ from Table \ref{tab: four general lines}.

    \item[$\inverse\pi(R_1) = \{Q_1, Q_2\}$, $\inverse\pi(R_2) = \{Q_3, Q_4\}$:]
      Assume there is such an involution on $\bar D^\nu$. Then the involution descends to $\bar D$ itself violating Proposition \ref{prop: case P exclude etale maps}.

      Put differently, the argument used in the previous case does not work because the morphism defined on the points will fix the given conic, see \cite[Es.~4.25]{FFP16}.

      Thus this case does not occur.
    \end{description}
  \end{description}

 \item[$\bar D = \text{a smooth or   nodal cubic and line}$]
 The involution has to preserve the line, because either the number of marked points on the two components
of $\bar D^\nu$ is different or they are not isomorphic.
But on a line with three marked points $\tau$ cannot induce a fixed-point-free involution on the marked points in violation of the Gorenstein-condition. 

Therefore this case cannot exist.
 \end{description}
 
 We have enumerated all possible cases for $\bar D$ and thus concluded the classification. 
 
  \begin{prop}
 Let $X$ be a Gorenstein stable surface with $K_X^2 = 1$ and $\chi(X) = 2$. If the normalisation $\bar X = \IP^2$ and $\bar D\subset \bar X$ is reducible, then $X$ arises as in Cases A, B, C in Section \ref{sssect: case P reducible} or as a degeneration thereof.
 \end{prop}

\begin{rem}
  By dimension reasons the surfaces constructed in Example \ref{exam: case P 3 nodes} cannot degenerate to the general surface in Case A. It remains to work out explicitly their relation to the other reducible cases.
\end{rem}

\subsection{Cases $(dP)$ and $(E_-)$}
\label{sect: cased dP and E_}
We proceed to the next case, starting with a well-known result, at least in the smooth case (see e.g. \cite[3.5 Theorem]{MR1440180}).
\begin{lem}
  Let $(\bar X, \bar D)$ be a log-canonical pair such that $K_{\bar X}+\bar D$
  is Cartier, $(K_{\bar X}+\bar D)^2 = 1$, and the minimal resolution of $\bar
  X$ is either a del Pezzo surface of degree $1$ or of type $E_-$. 
  Then $-K_{\bar X}$ is an ample Cartier divisor of square $1$ and $\bar D \in |-2K_{\bar X}|$.

  Moreover we have
  \[ R(\bar X, -K_{\bar X}) \isom \IC[x_1, x_2, y, z]/ (f_6)\]
  with variables of degrees $(1,1,2,3)$ and
  \begin{equation}\label{eq: del Pezzo} f_6 = z^2 + a_0 y^3 + a_2 y^2 + a_4y +a_6,
  \end{equation}
  where $a_i = a_i(x_1, x_2)$ is of degree $i$.

  If $\bar D$ is general in $|-2K_{\bar X}|$ then we can choose the coordinates such that $\bar D = \{ y =0\}$; the restriction of the anti-canonical ring to $\bar D$ gives a surjection
  \[ R(\bar X, -K_{\bar X}) \to \IC[x_1, x_2, y, z]/ (f_6, y )= \IC[x_1, x_2, z]/ (z^2 + a_6 ) = R(\bar D, K_{\bar D}).\]
\end{lem}
\begin{proof}
  Assume that $\bar{X}$ is a del Pezzo of degree 1, possibly with one elliptic singularity as in case $(E_-)$. Then according to the description in the List \ref{list}, $-K_{\bar X}$ is ample divisor and $K_{\bar{X}} + \bar{D} = -K_{\bar X}$. We can easily compute the canonical ring: for any $m \geq 0$ and for all $i > 0$
  \[
    H^{i}(X, -mK_X) = H^{i}(X, K_X + (-m -1)K_X ) = 0
  \]
and   Riemann-Roch gives us
  \[
    h^0(-mK_X) = \frac{1}{2}(-mK_X) (-mK_X - K_X) +1 = \frac{m(m+1)}{2} +1.
  \]
  Thus $h^0(-K_X) = 2, h^0(-2K_X) = 4$ and $h^0(-3K_X) = 7$. Let $x_1, x_2$ be generators of $H^0(-K_X)$, let $y$ be element in $H^0(-2K_X)$ which is not in subspace  $ S^2\langle x_1, x_2 \rangle$, and let $z$ be an element in $H^0(-3K_X)$ which is not in the subspace $S^3\langle x_1, x_2 \rangle \oplus \langle x_1 y, x_2 y\rangle$. By comparing the dimension of $H^0(-mK_X)$ and subspace generated by monomials in $x_1, x_2, y, z$ we obtain a the relation $z^2 + a_0 y^3 + a_2 y^2 + a_4 y + a_6 $ in degree 6. For example by comparing the Hilbert series,  we conclude that there are no other relations for any $m > 6$.
  Thus  the anti-canonical ring of a del Pezzo surface $\bar{X}$ is
  \[
    R(\bar{X}, -K_X) \isom \IC[x_1, x_2, y, z]/(f_6)
  \]
  with variables of degrees $(1, 1, 2, 3)$ and $f_6 = z^2 + a_0 y^3 + a_2 y^2 + a_4 y + a_6$.
  
  The last statment follows immediately.
  \end{proof}

Note that by the $\chi$-condition, given a pair $(\bar X, \bar D)$ as above, an
involution $\tau$ on $\bar D^\nu$ defines a Gorenstein stable surface with
$\chi(X) = 2$ if and only if it satisfies the Gorenstein gluing condition and
the resulting curve $D$ has arithmetic genus $1$. 

For simplicity, we restrict to the case where $\bar D$ is smooth, so that the
Gorenstein gluing condition is automatically satisfied, and we are looking for
curves $\bar D$ of genus $2$ which admit an elliptic involution. 
\begin{lem}\label{lem: genus 2 elliptic involution}
  Let $\bar D$ be a smooth curve of genus $2$ admitting an elliptic involution $\tau$, that is, $\bar D/\tau$ is an elliptic curve.

  Then decomposing pluricanonical forms into $\tau^*$-eigenspaces allows to choose generators of  the canonical ring such that
  \[ R(\bar D, K_{\bar D}) = \IC[x_1, x_2, z]/(z^2 + a_6),\]
  where $x_1, x_2, z$ have degrees $1,1,3$ respectively,
  \begin{equation}
    \label{eq: genus 2 with elliptic involution}
    a_6 = -(x_1^2 - \lambda_1x_2^2)(x_1^2 - \lambda_2x_2^2)(x_1^2 - \lambda_3x_2^2), \quad \lambda_i \in \IC^*, \end{equation}
\end{lem}
and $\tau$ acts via $(x_1, x_2, z)\mapsto (-x_1, x_2, z)$.
\begin{proof}
  If $\pi \colon \bar D \to \bar D/\tau=D$, then $\pi_*\ko_{\bar D} = \ko_D \oplus L$ and $K_{\bar D} = \pi^*(K_D +L) = \pi^*L$ for a line bundle $L$ of degree $1$ on $D$. 
    We get a decomposition into invariant and anti-invariant subspaces 
    \[ H^0(mK_{\bar D}) = H^0(D, \pi_*\pi^*L) \isom H^0(mL) \oplus H^0((m+1)L)\] by the projection formula. Choosing the generators such that the action of $\tau$ is as given, the form of the equation is mandated by the requirement that it be $\tau$-invariant. 
\end{proof}

\begin{prop}
  Let $(\bar X, \bar D, \tau)$ be the triple corresponding to a Gorenstein stable surface $X$ with $K_X^2 = 1$ and $\chi(X) = 2$ and such that $\bar X$ is of type \casedP.

  Then there exists $a_6$ as in \eqref{eq: genus 2 with elliptic involution} and $f_6$ as in \eqref{eq: del Pezzo} such that the inclusion map is induced by
  \[\IC[x_1, x_2, y, z] /(f_6) \to \IC[x_1, x_2, z] /(z^2 + a_6)\]
  and $\tau$ acts as in Lemma \ref{lem: genus 2 elliptic involution}
\end{prop}
\begin{proof}
  Follows immediately from the above lemma.
\end{proof}

We can now write down a family containing the surfaces discussed above as an open subset of:
\[\kw  = \left\{ (a_0, a_2, a_4, a_6)\in \IC[x_1, x_2] \mid a_6 \text{ as in \eqref{eq: genus 2 with elliptic involution}}\right\}\]

\begin{prop}
  The subset of $ \bar{\mathfrak{M}}^{(Gor)}_{1,2}$ parametrising surfaces with normalisation $(dP)$ or $(E_-)$ is irreducible of  dimension 10.
\end{prop}
\begin{proof} Let $R = \IC[x_1, x_2]$. Then this subset is dominated by an open subset of
  \[
    \kw = \left\{(a_0, a_2, a_4, a_6) \in R_0 \times R_2 \times R_4 \times R_6 \mid  a_6\  \text{as in } \eqref{eq: genus 2 with elliptic involution}  \right\},
  \]
  which is of dimension $1+3+5+3 = 12$. The choices made in the above construction fix the coordinates up to multiplication with non-zero numbers. In addition, we fixed the coefficient in front of $z^2$ to be $1$ as in \eqref{eq: del Pezzo} and the coefficient in front of $x_1^6$ in $a_6$ to be $1$ as in \eqref{eq: genus 2 with elliptic involution}. Thus  we have a remaining action of ${\IC^*}^2$ by multiplication on $x_2$ and $y$, and the dimension of the stratum is $12-2 = 10$.
\end{proof}

\subsection{Case ($E_{+}$)}
\label{sect: caseE}

Assume that $X$ is a Gorenstein stable surface with $K^2_X = 1$ and $\chi(\ko_X)=2$ and normalisation $\bar X = S^2E$ for an elliptic curve $E$. Using the notation from \eqref{diagr: pushout} we recall some facts from \cite{FPR15a}.

Consider
\[
  \begin{tikzcd}
    E\times E \rar{\sigma} \arrow{dr}{\oplus} & S^2 E\dar{\mathrm{alb}}\\ & E
  \end{tikzcd}.\]
Then $\mathrm{alb}$ is a $\IP^1$-bundle with section $C_0= \{(0, p)\in S^2E\mid p\in E\}$ and the fibre over $0$ is given by $F = \{ (p,-p)\in S^2E\mid p \in E\}$.
The canonical bundle is then $K_{S^2E} = -2C_0 +F$ and the conductor is a nodal curve of arithmetic genus $2$
\[ \bar D \in |3C_0 -F| = |C_0 - K_{S^2E}|.\]
By the $\chi$-condition, the conductor in $X$ has genus $0$ and thus $\bar D \to D$ is the canonical map induced by the hyperelliptic involution (if $\bar D$ is smooth or at least irreducible).

For the sake of completeness we first recall an observation from \cite[Rem.~5.3]{FPR17}.
\begin{lem}
  Let $E$ be an elliptic curve and $S^2E$ be its symmetric square.
  \begin{enumerate}
  \item There exist non-normal Gorenstein stable surfaces $X$ with $K^2_X = 1$ and $\chi(\ko_X)=2$ and normalisation $S^2E$.
  \item If $X$ is such a surface then the bicanonical map $X \to \IP^2$ is not a Galois cover.
  \end{enumerate}
\end{lem}
\begin{proof}
  \begin{enumerate}
  \item A general element $\bar D\in |3C_0-F|$ is smooth and, denoting by $\tau$ the hyperelliptic involution, the triple $(S^2E, \bar D, \tau)$ defines a surface with the required invariants by Koll\'ar's gluing theorem as explained in Section \ref{ssec: glue}.
  \item As in \cite{FPR17} Remark 5.3, a normalisation of a bi-double cover is again a bi-double cover, the canonical divisor of $\bar{X}$ is pullback of some $\ko_{\IP^2}(d)$ and thus either ample, anti-ample or trivial. Thus no bi-double cover can have a normalisation of type $E_+$. \qedhere
  \end{enumerate}
\end{proof}
Because of the second statement, a concrete (algebraic) description of an example could not be found in \cite{FPR17}. We will now give such a description yielding in fact a complete family parametrising an open subset of the stratum of surfaces in $\bar\gothM_{1,2}$ with normalisation the symmetric square of an elliptic curve.

We would like to compute the canonical ring of $X$ which is based on the following result of Koll\'ar:
\begin{prop}[\protect{\cite[Prop. 5.8]{Kollar2013}}]\label{prop: Kollar} Let $X$ be a
  Gorenstein stable surface. Define the different $\Delta = \text{Diff}_{\bar
    D^{\nu}}(0)$ by the equality $(K_{\bar D} + \bar D)|_{\bar D} = K_{\bar D} +
  \Delta$. 

  Then a section $s \in H^0(\bar X, m(K_{\bar X} + \bar D))$ descends to a section in $H^0(X,mK_X)$ if and only if the image of s in $H^0(\bar D^{\nu}, m(K_{\bar X} + \bar D))$ under the Residue map is $\tau$- invariant if  $m$ is even respectively $\tau$- anti invariant if $m$ is odd.
\end{prop}

To compute the canonical ring of $X$, we need to compute the ring of sections
\[R(S^2E, K_{S^2E}+\bar D) = R(S^2 E, C_0),\] the residue map to $R(\bar D,
K_{\bar D}) = R(\bar D, C_0|_{\bar D})$ including the action of the
hyperelliptic involution. The strategy is to pull back to $E\times E$ and then
to take invariants under the involution exchanging the factor. To simplify
notation we add indices to the factors $E\times E = E_1\times E_2$.

We consider the geometric situation
\begin{equation}\label{diagr: pushout S2E}
  \begin{tikzcd}
    E_1\times E_2\dar{\sigma}\rar[hookleftarrow]{\tilde \iota} & \tilde D\dar\\
    S^2E \dar{\pi}\rar[hookleftarrow]{\bar\iota} & \bar D\dar{\pi}[swap]{/\tau}
    \\
    X\rar[hookleftarrow]{\iota} &D
  \end{tikzcd}
\end{equation}

For linear series or spaces of sections on $E_1\times E_2$ we denote the invariant part under the involution interchanging the factors by a superscript ${}^+$. The next result is immediate.
\begin{lem}
  With the above notation we have
  \begin{gather*}
    \sigma^*C_0 = E_1\times\{0\} + \{0\} \times E_2,\\
    \sigma^*F = \Delta^- = \{(p, -p)\in E\times E\mid p\in E\} ,\\
    \sigma^*\bar D = \tilde D \in \sigma^*|3C_0-F| = |3\sigma^*C_0 -\Delta^-|^+.
  \end{gather*}
\end{lem}

We thus have  \[ H^0(S^2 E, mC_0) =  H^0(E\times E, m\sigma^* C_0)^+ \isom (H^0(E_1, m \cdot 0)\tensor H^0(E_2, m \cdot 0))^+.\]
Therefore,  if $v_1, \dots, v_m$ is a basis of $H^0(E, m\cdot 0)$ then a basis of  $(H^0(E_1, m \cdot 0)\tensor H^0(E_2, m \cdot 0))^+$ is given by
\[ (v_i\tensor v_i)_{i=1, \dots, m},\; (v_i\tensor v_j +v_j\tensor v_i)_{1\leq i<j\leq m}.\]
These are $m+ \frac{m(m-1)}2 = \frac{m(m+1)}2 = h^0(S^2 E, mC_0)  = \chi(S^2 E, mC_0)$ elements as predicted by Riemann-Roch.

We now choose for the elliptic curve a Weierstrass type equation
\begin{equation}\label{eq: Weierstrass}
f_i=y_i^2-(x_i^3+ax_iz_i^4+bz_i^6)\end{equation}
such that  $R(E_i, 0)\isom \IC[z_i,x_i,y_i]/(f_i)$ with generators in degrees $(1,2,3)$.

\begin{lem}
  The low-degree parts of $R(S^2E, C_0)$, identified with the invariant subring of $R(E_1\times E_2, \sigma^*C_0)$ are
  \begin{center}
    \begin{tabular}{cc}
    \toprule
      $m$ & Basis of $H^0(S^2E, m C_0)$ \\
      \midrule
      $1$ &$t_0 = z_1z_2$\\
      $2$ &$t_0^2 $,  $ t_1 = x_1x_2$, $t_2 = z_1^2x_2+ x_1z_2^2$\\
       $3$ & $t_0^3$, $ t_0 t_1$, $t_0 t_2$,\\
       & $ t_3 = y_1y_2$, $t_4 = z_1x_1y_2+ y_1 z_2x_2$, $t_5 = z_1^3y_2 + y_1z_2^3$ \\
      $4$ &  $t_0^4$, $t_0^2t_1$, $t_0^2t_2$, $t_0 t_3$, $t_0t_4$, $t_0t_5 $, $t_1^2$, $t_1t_2$, $t_2^2$,
      \\ & 
      $t_6 = z_1y_1x_2^2 + x_1^2z_2y_2$\\
      \bottomrule
    \end{tabular}
  \end{center}
In fact, $t_0, \dots, t_6$ generate the section ring $R(S^2E, m C_0)$.
\end{lem}
\begin{proof} 
  If we follow for $m \leq 3$ the outlined procedure we get the elements listed in the table. For $m = 4$ we have $H^0(E, 4\cdot 0) = \langle z^4, z^2x, zy, x^2\rangle$ the procedure gives an additional element
  \[ z_1^4x_2^2 + x_1^2z_2^4 = t_2^2-2 t_0^2 t_1,
  \]
  which is not part of the basis in the table. 
  
  To show that $t_0,  \dots, t_6$ generate the full invariant subring we argue as follows: take a basis element given in the form
  \[ (z_1^{a_1}x_1^{b_1}y_1^{c_1})
  (z_2^{a_2}x_2^{b_2}y_2^{c_2}) +
(z_1^{a_2}x_1^{b_2}y_1^{c_2})
(z_2^{a_1}x_2^{b_1}y_2^{c_1}).
   \]
Since there are the relations $f_i$ we may assume that $c_j\leq 1$. Dividing by the generators $t_0, t_1, t_3$ we may also assume that $a_1a_2 = b_1b_2 = c_1c_2 = 0$. The remaining possibilities are
  \begin{align*}
    x_1^a z_2^c + z_1^c x_2^a&\equiv t_2^{d_2}  t_5^{d_5} \mod (t_0, t_1)\tag {$c = 2d_2+3d_5$}\\
      x_1^a y_1 z_2^c + z_1^c x_2^a y_2&\equiv t_2^{d_2}  t_5^{d_5}t_4 \mod (t_0, t_1)\tag {$c-1= 2d_2+3d_5$}\\
      x_1^a z_2^cy_2 + z_1^cy_1 x_2^a&\equiv t_2^{d_2}  t_5^{d_5}t_6 \mod (t_0, t_1)\tag {$c-1 = 2d_2+3d_5$}
  \end{align*}
so we have already found all generators in $t_0, \dots, t_6$. 
  \end{proof}

Taking the invariant part of a weighted Segre embedding we want to get the
following diagram, where $X$ is embedded in $\IP(1,2,2,3,3)$ 
as complete intersection in degree $(6,6)$.
\[
  \begin{tikzcd}
    \tilde D \rar[hookrightarrow] \dar& E\times E \dar{\sigma} \rar[hookrightarrow] & \IP(1,2,3)\times \IP(1,2,3) \dar[dashed]{(t_0:\dots: t_6)}\\
    \bar D \rar[hookrightarrow]\dar &  S^2 E \rar[hookrightarrow]{R(C_0)}\dar{\pi} &  \IP(1,2,2,3,3,3, 4)\dar[dashed]\\
    D \rar[hookrightarrow] & X \rar[hookrightarrow]{R(K_X)} & \IP(1,2,2,3,3)
  \end{tikzcd},
\]

\begin{lem}\label{lem: what is l?} Let $s_F$ be a section defining $F$. Then the image of the composition
  \[ \begin{tikzcd}H^0(S^2 E, \bar D) \rar[hookrightarrow]{\cdot s_F}&  H^0(S^2E, 3C_0)  \rar[hookrightarrow] & H^0(E\times E, \sigma^*3C_0)\end{tikzcd}\]
  is spanned by the sections $t_4 =z_1 x_1 y_2+y_1 x_2 z_2, t_5= y_1 z_2^3+z_1^3 y_2$.
\end{lem}
\begin{proof}
  Consider the exact sequence
  \[ 0\to H^0(E\times E, \sigma^* \bar D) \to H^0(E\times E, 3 \sigma^*C_0) \to H^0(E\times E, \sigma^*3C_0|_{\Delta^-}).\]
  Since in the Weierstrass model inversion on the elliptic curve corresponds to changing the sign of the $y$-coordinate one can see that the only invariant sections of $\sigma^*3C_0$ vanishing on $\Delta^-$ are the ones given above.
\end{proof}

We now fix the section $s_{\bar D}$ defining $\bar D$. By Lemma \ref{lem: what is l?} there exist $\alpha, \beta \in \IC$ such that the image of $s_{\bar D}$  in $H^0(S^2E, 3C_0)$ is
\begin{equation}
\label{eq: l} 
s_4:=s_{\bar D} \cdot s_F  = \alpha t_4+\beta t_5.
 \end{equation}

\begin{thm}\label{thm: S2E}
There are explict elements $s_0, \dots, s_4$ in $R(E\times E, \sigma^*C_0)$, depending on sufficiently general choices made in \eqref{eq: Weierstrass} and \eqref{eq: l} such that
    \[ \IC[X, Y_1, Y_2, Z_1, Z_2] \onto  R(K_X) \overset{\pi^*}\isom \IC[s_0, \dots, s_4] \subset R(E\times E, \sigma^*C_0)^+\]
Realising the embedding $X\into \IP(1,2,2,3,3) $ via
$\IC[x, y_1, y_2, z_1, z_2]\onto \IC[s_0, \dots, s_4]$
 the surface is cut out by the equations   
  \begin{equation}
    \label{eq:canonical ring of S2E}
    \begin{split}
      z_1^2 + b_1(x,y_1,y_2) = 0,\\
      z_2 ^2 + xz_1a_2(x,y_1,y_2) + b_2(x,y_1,y_2) = 0,
    \end{split}
  \end{equation}
where \begin{align*}
         b_1 = &-(b^{2}x^{6}+abx^{4}{y}_{1}+b{y}_{1}^{3}+a^{2}x^{4}{y}_{2}-3bx^{2}{y}_{1}{y}_{2}+a{y}_{1}^{2}{y}_{2}-2ax^{2}{y}_{2}^{2}+{y}_{2}^{3}),\\
         a_2 = &-(2{\beta}^{2}x^{2}+2{\alpha}{\beta}{y}_{1}+2{\alpha}^{2}{y}_{2}),\\
         b_2 =  &-(2b{\beta}^{2}x^{6}+\left(2b{\alpha}{\beta}+a{\beta}^{2}\right)x^{4}{y}_{1}+b{\alpha}^{2}x^{2}{y}_{1}^{2} +{\beta}^{2}{y}_{1}^{3}+\left(-2b{\alpha}^{2}+4a{\alpha}{\beta}\right)x^{4}{y}_{2}\\
               &+(a{\alpha}^{2}-3{\beta}^{2})x^{2}{y}_{1}{y}_{2}+2{\alpha}{\beta}{y}_{1}^{2}{y}_{2}-4{\alpha}{\beta}x^{2}{y}_{2}^{2}+{\alpha}^{2}{y}_{1}{y}_{2}^{2} ).
       \end{align*}
       Therefore $X$ is an iterated double cover.
           \end{thm}
     
     \begin{rem}
      It was annoyingly difficult to get from the abstract existence result via gluing already contained in \cite{FPR17} to an explicit representation by equations, precisely because the surface does not admit a $(\IZ/2)^2$ symmetry. Surprisingly, all these surfaces turn out to carry an involution.
     \end{rem}

     \subsubsection{Proof of Theorem \ref{thm: S2E}}
     
     By Koll\'ar's result (Proposition \ref{prop: Kollar}), the canonical ring of $X$ is the pullback ring in the diagram
     \[ \begin{tikzcd}
         R(E\times E, \sigma^*C_0) ^+=
         R(S^2E, C_0)\rar[equal] & R(S^2E, K_{S^2E}+\bar D) \rar{\bar\iota^*} &  R(\bar D, K_{\bar D})\\
         &R(X, K_X) \rar[hookrightarrow] \uar[hookrightarrow] & R(D, K_X|_D)\uar[hookrightarrow]{\pi^*}
       \end{tikzcd}
     \]
     In the following we identify $R(S^2E, K_{S^2E}+\bar D)$ with $ R(E\times E, \sigma^*C_0) ^+$ via the pullback map $\sigma^*$.

     \begin{lem}\label{lem: S2E sections vanishing along Dbar}
       Let $m \geq 2$. Let $s_{\bar D}$ be the section defining $\bar D$ and $s_{F}$ be the section defining $F$, so that  $s_{\bar D}s_F = s_4$.
       Then the sequence
       \[ 0 \to H^0(S^2E, mC_0 -\bar D) \overset{s_{\bar D}}\to H^0(S^2E, mC_0) \to H^0(\bar D, mK_{\bar D}) \to0\]
       is exact. In particular, $h^0(S^2E, mC_0 -\bar D) = \frac{m(m+1)}2 - (2m - 1) = \frac{m(m-3)}2 +1$.
       \begin{center}
         \begin{tabular}{cc}
           $m$ & generators of  image  of $H^0(S^2E, m C_0-\bar D)$ in $H^0(S^2E, m C_0)$\\
           \midrule
           $1$ & $0$\\
           $2$ & $0$\\
           $3$ & $s_4 = \alpha t_4 + \beta t_5$\\
           $4$ & $t_0s_4 $,\\
               & $l_1 =  (b \alpha -a \beta ){t}_{0}^{4}+ \beta {t}_{0}^{2} {t}_{2}- \beta
                 {t}_{1}^{2} -\alpha {t}_{1} {t}_{2}-
                 \alpha {t}_{0} {t}_{3},$\\
               & $l_2 =b \beta {t}_{0}^{4}+ a \alpha {t}_{0}^{2} {t}_{2}+b \alpha {t}_{0}^{2}
                 {t}_{1}- \alpha {t}_{2}^{2}-
                 \beta {t}_{1} {t}_{2}+ \beta {t}_{0} {t}_{3}$
         \end{tabular}
       \end{center}
     \end{lem}
     \begin{proof}
     This can be computed with Macaulay2, see \cite{thesisAnh} for details.
      \end{proof}

     \begin{lem} \label{lem: compute ring S2E deg 1}
       There exists a choice of generators $A,B$ of degree $1$ and $C$ of degree $3$ such that the diagram
       \begin{equation}\label{diag: rings S2E}
         \begin{tikzcd}
           R(E\times E, \sigma^*C_0) ^+ \rar {\bar\iota^*} &  R(\bar D, K_{\bar D}) \rar[equal] &  \IC[A,B,C]/(C^2-g(A,B))\\
           & R(D, K_X|_D) \rar[equal]\uar[hookrightarrow] &\IC[A,B]\uar[hookrightarrow]
         \end{tikzcd}
       \end{equation}
       commutes and $\bar\iota^* t_0 = A$.
       
       The image of the residue map $\bar\iota^*$ is then the subring generated by $A, AB, B^2, B^3, C, CB$.
     \end{lem}
     \begin{proof}
       The fact that the right hand side of the diagram is of the given form follows from the fact that $\bar D$ has genus $2$ and $\bar D \to D$ is the quotient by the hyperelliptic involution.

       Considering only the part of degree $1$ in the rings we have
       \[\begin{tikzcd} H^0(E\times E, \sigma^*C_0) ^+ = \langle t_0 \rangle \rar[hookrightarrow]{\bar\iota^*}  & \langle A, B \rangle\end{tikzcd}\]
       and we can arrange
       $\bar\iota^* t_0 = A$ by a linear coordinate change not affecting the form of the equation for $\bar D$.

       For  $m \geq 2$ we consider  the exact sequence
\[
H^0 (S^2E, m(K_{S^2E}+\bar D)) \overset{\bar\iota^*}{\to} H^0(\bar D, mK_{\bar D}) \to H^1(m(K_{S^2E}+\bar D )- \bar D)= H^1((m-1) C_0 + K_{S^2E})=0,
\]
where Kodaira vanishing applies, since  $(m-1) C_0$ is ample if $m - 1$ is positive.
 Thus the map $\bar\iota^*$ is surjective for $m \geq 2$ and it is easy to see that the given elements generate the image as a ring. 
      \end{proof}
 \begin{lem}\label{lem: compute ring S2E deg 2}
       In Diagram \eqref{diag: rings S2E} consider the elements of degree $2$ giving
       \[\begin{tikzcd}
           \langle t_0^2, t_1, t_2\rangle  = H^0(S^2E, 2C_0 )  \rar{\bar\iota^*}[swap]{\isom} & H^0(\bar D, 2K_{\bar D}) = \langle A^2, AB, B^2\rangle.
         \end{tikzcd}
       \]
       Let $s_0 = t_0$, $s_1 = \alpha t_2+  \beta t_1$ and $s_2 = (2 b\alpha \beta-a \beta^2) t_0^2+b\alpha^2 t_1+(a\alpha^2+\beta^2) t_2$. Then keeping $A = \bar\iota^*t_0$ the generator $B$ can be chosen such that
       \[ AB = \bar\iota^*(s_1),\text{ and }  B^2  = \bar\iota^*\left(s_2\right). \]
     \end{lem}
     \begin{proof}
       Since the map is an isomorphism and by the choice of $\bar\iota^*t_0  = A$ from Lemma \ref{lem: compute ring S2E deg 1}
       we need to show that there is an essentially unique way to complete $s_0^2$ to  a basis $s_0^2, s_1, s_2$ of  $\langle t_0^2, t_1, t_2\rangle$ such that $s_0^2 s_2 -s_1^2 = 0$  after restriction to $\bar D$. 
      
      Identifying  $H^0(S^2E, 4C_0) = H^0(E\times E, 4\sigma^*C_0)^+$ and using the basis found in Lemma \ref{lem: S2E sections vanishing along Dbar}, we have to find elements $s_1, s_2\in \langle t_0^2, t_1, t_2\rangle $ such that
      \[ s_0^2 s_2 -s_1^2 \in \langle t_0s_4, l_1, l_2\rangle .\]
      Since the equation cannot be divisible by $t_0$, the required linear combination has to involve $l_1 $ and $l_2$. 
    On the other hand, ince $s_0^2, s_1, s_2$ have degree 2, the term  $t_0t_3$ has to be eliminated on the right hand side. Up to scaling there is a unique way to do so, that is, 
    \[ s_0^2 s_2 -s_1^2 \equiv \beta l_1 + \alpha l_2\mod t_0s_4.\]
        This can be solved to give 
        \[s_0 = t_0, s_1 = \alpha t_2 + \beta t_1\text{ and }s_2 = (2 b\alpha \beta-a \beta^2) t_0^2+b\alpha^2 t_1+(a\alpha^2+\beta^2) t_2, \]   
        since we have $t_0^2\left((2 b\alpha \beta-a \beta^2) t_0^2+b\alpha^2 t_1+(a\alpha^2+\beta^2) t_2\right) \equiv (\alpha t_2+  \beta t_1)^2 \mod t_0s_4$.
     \end{proof}

     \begin{lem}
       \label{lem: compute ring S2E deg 3}
       In  Diagram \eqref{diag: rings S2E} consider the elements of degree $3$ giving
       \[\begin{tikzcd}[sep=small, cramped]
           0 \rar & H^0(S^2E, 3C_0 -\bar D)\dar[equal] \rar & H^0(S^2E, 3C_0 )  \rar{\bar\iota^*}\dar[equal] & H^0(\bar D,3 K_{\bar D}) \dar[equal]\rar &0 \\
           0\rar& \langle  \alpha t_4+ \beta t_5\rangle \dar[equal]\rar[hookrightarrow] & \langle s_0^3, s_0s_1, s_0s_2, t_3, t_4, t_5\rangle \rar[twoheadrightarrow] & \langle A^3, A^2B, AB^2, B^3, C\rangle\rar & 0 \\
           0\rar & \langle \alpha t_4+ \beta t_5\rangle \rar[hookrightarrow] & H^0(X, 3K_X) \uar[hookrightarrow]\rar{\iota^*} &  \langle A^3, A^2B, AB^2, B^3 \rangle\uar[hookrightarrow]\rar & 0
         \end{tikzcd}
       \]
       With the choices of $s_0, s_1, s_2$ from Lemma \ref{lem: compute ring S2E deg 1} and Lemma \ref{lem: compute ring S2E deg 2} the element
       \[s_3 = (b^2 \alpha^3 + b\beta^3)t_0^3+ (ab\alpha^3 + 3b\alpha \beta^2 - a \beta^3)t_0t_1+ (a^2\alpha^3 + 3b\alpha^2 \beta)t_0t_2+ (-b\alpha^3 +a\alpha^2\beta + \beta^3)t_3\]
       satisfies $\bar\iota^*s_3 = B^3$ and it is unique with this property modulo $s_{\bar D}$.
     \end{lem}
     \begin{proof}
       In Diagram \eqref{diag: rings S2E} consider the elements of degree 3, this gives us
       \[\begin{tikzcd}
           \langle s_0^3, s_0s_1, s_0s_2, t_3, t_4, t_5 \rangle  = H^0(S^2E, 3C_0 )  \rar{\bar\iota^*} & H^0(\bar D, 3K_{\bar D}) = \langle A^3, A^2B, AB^2, B^3, C\rangle.
         \end{tikzcd}
       \]
       By previous lemmas we already identified $s_0, s_1, s_2$ and their images $A, AB, B^2$.
       Note that $C$ has no relations with other generators in this degree and clearly $ \alpha t_4+ \beta t_5$ is in the canonical ring because it restricts to zero on $\bar D$.
       Thus there is only way to map $\alpha t_4 + \beta t_5$ to $C$.
       We need only to identify the element $s_3 \in \langle t_0^3, t_0t_1, t_0t_2, t_3 \rangle$ such that $s_0^3s_3- s_0^2s_1s_2 = 0$ modulo the equation of $\bar D$.
       We can also do this in degree 4 by relation $s_0s_3 - s_1s_2 = 0$ after restriction to $\bar D$.
       From the equations of $s_1, s_2, l_1, l_2$ it is easy to see that $s_1s_2+b \alpha^2 l_1+ (a \alpha^2+ \beta^2)l_2 $ kills all terms of $t_1^2, t_1t_2, t_2^2$ and equals to
       \begin{multline*}t_0\left((b^2 \alpha^3+b \beta^3) t_0^3+(a b \alpha^3+3 b \alpha \beta^2-a \beta^3) t_0 t_1+(a^2 \alpha^3+3 b \alpha^2 \beta) t_0 t_2
       \right.
         \\
         \left.
         +(-b \alpha^3+a \alpha^2 \beta+\beta^3) t_3\right)
         \end{multline*}
         
       Dividing by $s_0 = t_0$ we arrive at the element $s_3$ given in the statement of the Lemma, which then by construction satisfies $(s_0s_3-s_1s_2)|_{\bar D}=0$
and we can choose $\bar\iota^*s_3 = B^3$.
     \end{proof}

     By Theorem \ref{thm: structure from FPR} we know that the canonical ring of $X$ is of the form \eqref{CanRing} and we have found the required generators in the required degrees, $s_0, \dots, s_4$. 
    Thus 
    \[ \IC[X, Y_1, Y_2, Z_1, Z_2] \onto  R(K_X) \isom \IC[s_0, \dots, s_4] \subset R(E\times E, \sigma^*C_0)^+\]
    with explicitly given generators. The relations can then be computed using Macaulay2 \cite{M2}, see \cite{thesisAnh} for details. The form of the equations shows that $X$ has a $\IZ/2$ symmetry by Proposition \ref{prop: aut}.
    
    This concludes the proof of Theorem \ref{thm: S2E}. 
%


\begin{thebibliography}{BHPV04}

\bibitem[AC12]{Arinkin2012}
Dima Arinkin and Andrei C\u{a}ld\u{a}raru.
\newblock When is the self-intersection of a subvariety a fibration?
\newblock {\em Advances in Mathematics}, 231(2):815 -- 842, 2012.

\bibitem[Ant18]{Anthes2018}
Ben Anthes.
\newblock {\em Gorenstein stable surfaces satisfying $K_X^{2} = 2$ and
  $\chi(\mathcal O_X)= 4$}.
\newblock PhD thesis, Philipps-Universit{\"a}t Marburg, 2018.

\bibitem[AP12]{alexeev-pardini12}
Valery Alexeev and Rita Pardini.
\newblock {Non-normal abelian covers}.
\newblock {\em Compos. Math.}, 148(4):1051--1084, 2012.

\bibitem[BDF07]{BaSa07}
G.~Bagnera and M.~De~Franchis.
\newblock Sopra le superficie algebriche che hanno le coordinate del punto
  generico esprimibili con funzioni meromorfe 4ente periodiche di 2 parametri.
\newblock {\em Rendiconti Acc. dei Lincei}, 16, 1907.

\bibitem[BH93]{Bruns-Herzog}
Winfried Bruns and J{\"u}rgen Herzog.
\newblock {\em Cohen-{M}acaulay rings}, volume~39 of {\em Cambridge Studies in
  Advanced Mathematics}.
\newblock Cambridge University Press, Cambridge, 1993.

\bibitem[BHPV04]{BHPV}
Wolf~P. Barth, Klaus Hulek, Chris A.~M. Peters, and Antonius {Van de Ven}.
\newblock {\em {Compact complex surfaces}}, volume~4 of {\em {Ergebnisse der
  Mathematik und ihrer Grenzgebiete. 3. Folge.}}
\newblock Springer-Verlag, Berlin, second edition, 2004.

\bibitem[Cat79]{catanese79}
Fabrizio Catanese.
\newblock Surfaces with {$K\sp{2}=p\sb{g}=1$} and their period mapping.
\newblock In {\em Algebraic geometry ({P}roc. {S}ummer {M}eeting, {U}niv.
  {C}openhagen, {C}openhagen, 1978)}, volume 732 of {\em Lecture Notes in
  Math.}, pages 1--29. Springer, Berlin, 1979.

\bibitem[Cat80]{Cat80}
F.~Catanese.
\newblock The moduli and the global period mapping of surfaces with $ k^2= p_g=
  1$: a counterexample to the global torelli problem.
\newblock {\em Compositio Mathematica}, 41(3):401--414, 1980.

\bibitem[CD89]{CossecDolgachev}
Fran\c{c}ois~R. Cossec and Igor~V. Dolgachev.
\newblock {\em Enriques surfaces. {I}}, volume~76 of {\em Progress in
  Mathematics}.
\newblock Birkh\"{a}user Boston, Inc., Boston, MA, 1989.

\bibitem[CE96]{CE96}
G.~Casnati and T.~Ekedahl.
\newblock {Covers of algebraic varieties. {I}. {A} general structure theorem,
  covers of degree {$3,4$} and {E}nriques surfaces}.
\newblock {\em J. Algebraic Geom.}, 5(3):439--460, 1996.

\bibitem[Do21]{thesisAnh}
Anh~Thi Do.
\newblock {\em Quadruple covers and Gorenstein stable surface with $K_X^2 = 1$
  and $\chi(X) = 2$}.
\newblock PhD thesis, Philipps University Marburg, 2021.
\newblock \url{https://archiv.ub.uni-marburg.de/diss/z2021/0299}.

\bibitem[Eis13]{Eisenbud2013}
David Eisenbud.
\newblock {\em Commutative Algebra: with a view toward algebraic geometry},
  volume 150.
\newblock Springer Science \& Business Media, 2013.

\bibitem[Far16]{MR3538518}
{\L}ucja Farnik.
\newblock A note on {S}eshadri constants of line bundles on hyperelliptic
  surfaces.
\newblock {\em Arch. Math. (Basel)}, 107(3):227--237, 2016.

\bibitem[FFP16]{FFP16}
Elisabetta Fortuna, Roberto Frigerio, and Rita Pardini.
\newblock {\em Projective geometry}, volume 104 of {\em Unitext}.
\newblock Springer, [Cham], italian edition, 2016.
\newblock Solved problems and theory review, La Matematica per il 3+2.

\bibitem[FPR15a]{FPR15inv}
Marco Franciosi, Rita Pardini, and S{\"o}nke Rollenske.
\newblock {Computing invariants of semi-log-canonical surfaces}.
\newblock {\em Math. Z.}, 280(3-4):1107--1123, 2015.

\bibitem[FPR15b]{FPR15a}
Marco Franciosi, Rita Pardini, and S{\"o}nke Rollenske.
\newblock {Log-canonical pairs and {G}orenstein stable surfaces with
  {$K_X^2=1$}}.
\newblock {\em Compos. Math.}, 151(8):1529--1542, 2015.

\bibitem[FPR17]{FPR17}
Marco Franciosi, Rita Pardini, and S\"{o}nke Rollenske.
\newblock Gorenstein stable surfaces with {$K^2_X=1$} and {$p_g>0$}.
\newblock {\em Math. Nachr.}, 290(5-6):794--814, 2017.

\bibitem[FPR18]{FPR18}
Marco {Franciosi}, Rita {Pardini}, and S\"onke {Rollenske}.
\newblock {Gorenstein stable Godeaux surfaces.}
\newblock {\em {Sel. Math., New Ser.}}, 24(4):3349--3379, 2018.

\bibitem[Fri12]{Friedman}
Robert Friedman.
\newblock {\em Algebraic surfaces and holomorphic vector bundles}.
\newblock Springer Science \& Business Media, 2012.

\bibitem[GH16]{GritsenkoHulek}
V.~Gritsenko and K.~Hulek.
\newblock Moduli of polarized {E}nriques surfaces.
\newblock In {\em K3 surfaces and their moduli}, volume 315 of {\em Progr.
  Math.}, pages 55--72. Birkh\"{a}user/Springer, [Cham], 2016.

\bibitem[Gie77]{gieseker1977global}
David Gieseker.
\newblock Global moduli for surfaces of general type.
\newblock {\em Inventiones Mathematicae}, 43(3):233--282, 1977.

\bibitem[GS]{M2}
Daniel~R. Grayson and Michael~E. Stillman.
\newblock {Macaulay2, a software system for research in algebraic geometry}.
\newblock Available at \url{http://www.math.uiuc.edu/Macaulay2/}.

\bibitem[Har77]{Hartshorne}
Robin Hartshorne.
\newblock {\em {Algebraic geometry}}.
\newblock Springer-Verlag, New York, 1977.
\newblock Graduate Texts in Mathematics, No. 52.

\bibitem[Kol96]{MR1440180}
J\'{a}nos Koll\'{a}r.
\newblock {\em Rational curves on algebraic varieties}, volume~32 of {\em
  Ergebnisse der Mathematik und ihrer Grenzgebiete. 3. Folge. A Series of
  Modern Surveys in Mathematics [Results in Mathematics and Related Areas. 3rd
  Series. A Series of Modern Surveys in Mathematics]}.
\newblock Springer-Verlag, Berlin, 1996.

\bibitem[Kol12]{kollar12}
Jan\'os Koll\'ar.
\newblock Moduli of varieties of general type.
\newblock In G.~Farkas and I.~Morrison, editors, {\em Handbook of Moduli:
  Volume II}, volume~24 of {\em Advanced Lectures in Mathematics}, pages
  131--158. International Press, 2012.

\bibitem[Kol13]{Kollar2013}
J{\'a}nos Koll{\'a}r.
\newblock {\em Singularities of the minimal model program}, volume 200 of {\em
  Cambridge Tracts in Mathematics}.
\newblock Cambridge University Press, Cambridge, 2013.
\newblock With a collaboration of S{\'a}ndor Kov{\'a}cs.

\bibitem[Kol21]{KollarModuli}
J{\'a}nos Koll{\'a}r.
\newblock {\em Moduli of varieties of general type}.
\newblock 2021.
\newblock book in preparation,
  \url{https://web.math.princeton.edu/~kollar/book/modbook20170720-hyper.pdf}.

\bibitem[Mor]{ksb88}
Shigefumi Mori.
\newblock The geometric compactification of the moduli of surfaces of general
  type (after koll{\'a}r, shepherd-barron, alexeev).

\bibitem[Par91]{pardini91}
Rita Pardini.
\newblock Abelian covers of algebraic varieties.
\newblock {\em J. Reine Angew. Math.}, 417:191--213, 1991.

\bibitem[Rei97]{reid97}
Miles Reid.
\newblock {Chapters on algebraic surfaces}.
\newblock In {\em {Complex algebraic geometry ({P}ark {C}ity, {UT}, 1993)}},
  volume~3 of {\em {IAS/Park City Math. Ser.}}, pages 3--159. Amer. Math. Soc.,
  Providence, RI, 1997.

\bibitem[Tzi09]{tziolas09}
Nikolaos Tziolas.
\newblock {$\mathbb Q$}-{G}orenstein deformations of nonnormal surfaces.
\newblock {\em Amer. J. Math.}, 131(1):171--193, 2009.

\end{thebibliography}

\end{document}